\documentclass[11pt]{article}

\usepackage{amsfonts, color}
\usepackage{amssymb,amsmath,amsthm, enumerate}
\usepackage{latexsym}
\usepackage{fullpage}
\usepackage{hyperref}
\usepackage{soul}
\usepackage{graphicx}
\newtheorem{theorem}{Theorem}[section]

\newtheorem{prop}[theorem]{Proposition}

\newtheorem{lemma}[theorem]{Lemma}
\newtheorem{cor}[theorem]{Corollary}

\theoremstyle{definition}

\newcounter{tenumerate}

\def\P{\mathbb{P}}

\newcommand{\one}{\1}

\renewcommand{\epsilon}{\varepsilon}

\newcommand{\1}{\mathbf{1}}

\DeclareMathOperator{\var}{Var}

\newcommand{\R}{{\mathbb R}}
\newcommand{\N}{{\mathbb N}}

\newcommand{\E}{{\mathbb E}}
\newcommand{\remove}[1]{}
\renewcommand{\le}{\leqslant}
\renewcommand{\ge}{\geqslant}
\renewcommand{\leq}{\leqslant}
\renewcommand{\geq}{\geqslant}

\newcommand{\cov}{\mathrm{Cov}}

\def\XXint#1#2#3{{\setbox0=\hbox{$#1{#2#3}{\int}$}
\vcenter{\hbox{$#2#3$}}\kern-.5\wd0}}

\newcommand{\Cb}{{\bf C}}
\newcommand{\Xc}{X^c}

\newcommand{\mcW}{V}
\newcommand{\brw}{{\vartheta}}
\newcommand{\mbrw}{{\xi}}

\def \BB{{\mathcal B}}
\def \Z{{\mathbb{Z}}}
\def \B{{\wp}}

\begin{document}

\title{Convergence in law of the maximum of the
two-dimensional discrete Gaussian free field}
\author{Maury Bramson\\
University of Minnesota
\thanks{Partially supported by NSF grants DMS-1105668 and DMS-1203201.} \and
Jian Ding \\
University of Chicago\thanks{Partially supported by NSF grant DMS-1313596.} \and Ofer Zeitouni\thanks{Partially supported by
NSF grant
DMS-1106627, a grant from the Israel Science
Foundation, and the Herman P. Taubman chair of Mathematics at the
Weizmann institute.}
\\Weizmann Institute \\ \& Courant Institute}

\date{{July 4, 2015}}
\maketitle

\begin{abstract}
  We consider the discrete two-dimensional Gaussian free field on a box of side length $N$,
  with Dirichlet boundary data, and prove the convergence of the law of
  the centered maximum of the field.
\end{abstract}

\section{Introduction}
The discrete Gaussian free field (GFF) $\{\eta_{v,N}: v\in V_N\}$, on the
box $V_N = ([0,N) \cap \mathbb{Z})^2$ 
with Dirichlet boundary data,
is the  
uniquely defined
mean zero Gaussian process that takes the value 0 on $\partial
V_N$ and satisfies the following Markov field condition for all
$v\in V_N\setminus
\partial V_N$: $\eta_{v,N}$ is distributed as a Gaussian random
variable with
variance $1$ and mean equal to the average over its immediate
neighbors, given
the GFF on $V_N\setminus \{v\}$. {(For $A\subseteq \mathbb Z^2$, 
        $\partial A$ designates the boundary vertices of $A$, i.e., $\{v\in A: \mbox{ there exists } 
        u\in \mathbb Z^2\setminus A \mbox{ with } u\sim v\}$.)}
One aspect of the GFF that has received intense attention recently is the
behavior of its maximum $\eta_N^*=\max_{v\in V_N}
\eta_{v,N}$. Of greatest relevance to this paper are the
papers \cite{BDG01}, where it is proved that
$\eta_N^*/(2\sqrt{2/\pi}\log N)\to 1$ in probability;
\cite{BZ10}, where it is proved (following partial results
in \cite{BDZ11}) that,
for
\begin{equation}
\label{eq-friday3}
m_N = 2\sqrt{2/\pi} \big(\log N - \tfrac{3}{8} \log\log N\big),\end{equation}
$\E\eta_N^*=m_N+O(1)$
and $\eta_N^*-m_N$ is
a tight sequence of random variables; and \cite{DZ12},
where rough asymptotics of the probability
$\P (\eta_N^*\geq m_N+x)$ are derived for large $x$.
Aspects of this model have been treated in both the mathematics and
physics literature;
we refer the reader to \cite{DZ12} for an extensive discussion
of the history of this problem.

Once it has been established that
the
fluctuations of $M_N:=\eta_N^*-m_N$ are of order $1$,
it is natural to study the convergence of the
laws of $M_N$, in particular, whether the laws of $M_N$ do indeed
converge.
Our goal in the current paper is to establish this convergence, as
stated in the following theorem.
\begin{theorem}
  \label{theo-mainintro}
  The law of the
   random variable $\eta_N^*-m_N$ converges in distribution
  to a nondegenerate law $\mu_\infty$ as $N\to\infty$.
\end{theorem}
\noindent
We will also provide a description of the limit law $\mu_\infty$ in
Theorems \ref{limit-law} and \ref{explicit-limit-law} of
Section
\ref{sec-coarsefine}.

Besides the intrinsic interest in the study of the GFF, we note that
it is an example of a logarithmically correlated model.
The behavior of the maxima for such models is conjectured to be universal;
see, e.g.,
\cite{CLD} for (non-rigorous) arguments using a renormalization-group
approach and links to the freezing transition of spin-glass systems, and
\cite{FLD} for further information on extreme distributions.
On the mathematical side, numerous results and conjectures have been formulated
for such models; see \cite{SDals} for recent progress.
Theorem \ref{theo-mainintro} above
provides a partial answer to \cite[Conjecture 12]{SDals}
\footnote{After the ArXiv posting of this paper, another case of 
        \cite[Conjecture 12]{SDals}, namely 
        the case of star-scale kernels, 
        was settled by Madaule \cite{madaule}. The fine structure of
        the extremal process for the GFF, as well as a different proof
        of the characterization of the limit law of the centered maximum of the GFF, was obtained in 
        the important paper 
        \cite{biskup}; see the latter paper for further references.}.

The proof of Theorem \ref{theo-mainintro} can be described roughly as follows.
We fix a large integer $K$,
partition
$V_N$ into $K^2$ boxes of side length $N/K$,
and introduce a Gaussian field $X_v^f$
that we refer to as the \textit{fine field}
and that 
{is a GFF with zero boundary on these sub-boxes -- it is constructed by subtracting from the original  
GFF  its conditional
expectation, given the sigma-algebra generated by the GFF on  the boundary
of these sub-boxes. 
The fine field has the advantage that, due to the Markov
property of the GFF, its values in disjoint boxes are independent.}
We define the \textit{coarse field}
as the difference $X_v^c=\eta_{v,N}-X_v^f$.
The coarse field is, of course, correlated over the whole box $V_N$,
but it is relatively smooth; in fact, for fixed
$K$ as $N\to\infty$,
the field obtained by
rescaling the coarse field onto a box of side length $1$ in $\R^2$
converges to a limiting Gaussian field that possesses continuous sample paths
on appropriate subsets of $[0,1]^2$ (essentially, away from the boundaries
of the sub-boxes).

An important step in our approach is the computation of the
tail probabilities of
the maximum of the fine field, when
restricted to a box of side length $N/K$,
together with the computation of the law of the location of the maximum
(in the scale $N/K$). These computations
are done by building on the tail estimates derived
in \cite{DZ12}, and using a
modified second moment method. Another important step is to show that
the maximum of the GFF occurs only at points where the fine field
is atypically large. Once these two steps are established,
we can describe the limit law of the GFF by an appropriate mixture
of random variables whose distributions
are determined by the tail of the fine field.
The mixture coefficients are determined by an (independent)
percolation pattern
of potential
locations of the maximum and by the limiting coarse field.

The structure of the paper is as follows. We first introduce, in Section
\ref{sec-coarsefine}, the coarse and fine fields
alluded to above, and restate Theorem \ref{theo-mainintro} in terms of
this decomposition (see Theorem \ref{limit-law}).
We then introduce, in Section \ref{sec-prelim},
auxiliary processes (branching random walk (BRW) and the
modified branching random walk (MBRW) introduced in
\cite{BZ10}), and
recall several Gaussian tools and estimates that will be used
throughout the paper.
The long and technical Section \ref{sec-limittail} is devoted
to the derivation of the limiting tail estimates for the maximum of GFF.
Once this is established, approximations of the law of  $\eta^*_N$
by local maxima of the fine field are presented in Section
\ref{sec-pairap}, which lead, in Section 6, to the proofs of Theorems \ref{limit-law}
and \ref{explicit-limit-law}.

Throughout the paper, we will assume that $N=2^n$ for some $n\in \mathbb{Z}_+$.
This does not affect the structure of the proof of 
Theorem \ref{theo-mainintro}, but substantially simplifies the notation in two ways.
First, choosing $K$ so that {$K$ divides $N$}, 
it allows us to choose the boxes $V_N^{K,i}$ (defined at the 
beginning of the next section) to be of the same size; recall that  
$N\to\infty$ before $K\to\infty$.  One can circumvent 
       {the assumption that $K$ divides $N$}
        by adjusting the size of the outer ring of boxes $V_N^{K,i}$ 
        (see Figure \ref{figure-1}); 
this will have no effect on the limit law because the
probability that the maximum of $\eta_{v,N}$ lies in the outer 
ring of boxes tends to zero as $K\to\infty$.   
The assumption $N=2^n$ is also employed at the beginning of
Section \ref{sec-limittail} in order to ensure that the length $L$ of the boxes $B$ defined there is divisible by $(1-\delta)N$, with $\delta > 0$ fixed
(see Figure \ref{figure-2}).
This restriction on $N$ is easily circumvented by adjusting slightly the width of the $\delta N$-strip along the perimeter of the box 
$V_N$ 
({so that both $K$ and $L$ divide $(1-\delta)N$, which 
allows general $N$}); 
this again does not change the
argument.

%
%

\medskip


\noindent{\bf Notation.}
%
For functions $F(\cdot)$ and $G(\cdot)$, we write $F \lesssim G$ or $F = O(G)$
if there exists an absolute constant $C>0$ such that
$F \leq C G$ everywhere in the domain. We write $F \asymp G$ if
$F \lesssim G$ and $G\lesssim F$. 
For a parameter $\alpha$,
we use $F\lesssim_\alpha G$ and $F\asymp_\alpha G$ if the constant $C$ above depends on the parameter $\alpha$. For a vector $x\in \R^d$,
$[x]$ will denote the lattice point in $\Z^d$ whose coordinates are
the integer parts of the corresponding coordinates of $x$.

\section{The coarse and fine fields and the limit result}
\label{sec-coarsefine}
Let $\eta_{v,N}$ denote the GFF in the box $V_N$. As mentioned above,
we assume for simplicity that
$N=2^n$. Recall that $\eta_N^*=\max_{v\in V_N} \eta_{v,N}$.
Fix $K=2^k$, with $k$ an integer. Tile the box $V_N$ with
$4^k$  shifts of $V_{N/K}$, denoted by $V_N^{K,i}$, and let
${\cal F}={\cal F}_{N,K}$ denote the sigma-algebra generated by
$\{\eta_{v,N}: v\in \cup_i \partial V_N^{K,i}\}$. (The boundaries of
$V_N^{K,i}$ do in fact overlap.) We now define
\begin{equation}
\label{eq-of1}
X_v^c=X_{v,N,K}^c=\E(\eta_{v,N}\mid {\cal F}_{N,K}),\quad
X_v^f=X_{v,N,K}^f=\eta_{v,N}-X_v^c\,.
\end{equation}
Note that $X_v^f$ vanishes on the boundaries of the boxes $V_N^{K,i}$,
that the Gaussian fields $\{X^f\}$ and $\{X^c\}$ are independent, and that the
fields $\{X_v^f\}_{v\in V_N^{K,i}}$ are independent for different $i$ and
identically distributed as the GFF with zero boundary condition in $V_N^{K,i}$.

Throughout the argument, we will need to consider points that are
within an appropriate distance from the boundary of the boxes $V_N^{K,i}$.
Toward this end, fix $\delta>0$, with $1/\delta$ a power of $2$, and define the boxes
\begin{equation}
  \label{eq-maury160113a}
  V_N^{K,\delta,i}=\{x\in V_N^{K,i}: d_\infty(x,\partial V_N^{k,i})\geq
\delta N/K\}\,;
\end{equation}
set
$$V^{K,\delta}_N=\cup_{i} V_N^{K,\delta,i}\,,\quad
\Delta_N=\Delta_N^{K,\delta}=V_N\setminus V^{K,\delta}_N\,.$$
In addition, we write $V_N^\delta=V_N^{1,\delta}$.
Note that $|\Delta_N|\leq 
{8} \delta |V_N|$.
\begin{figure}[htb]
\begin{center}
\includegraphics[width=80mm]{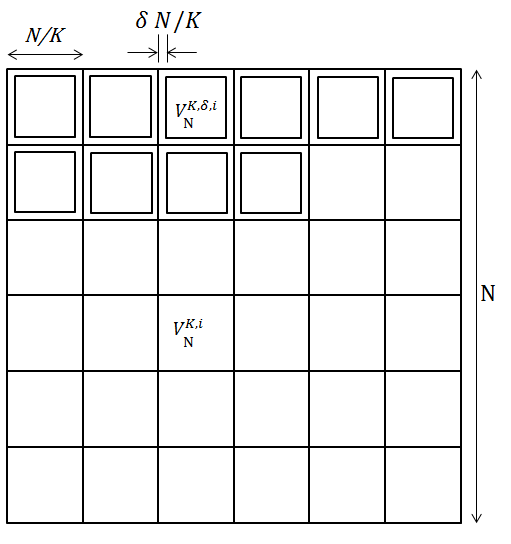}
\caption{The boxes $V_N$, $V_N^{K,i}$, $V_N^{K,\delta,i}$}
\label{figure-1}
\end{center}\end{figure}
\subsection{The coarse field limit}
We describe in this short subsection the scaling limit of the Gaussian
field $\Xc_{\cdot,N,K}$.
Introduce the covariance
\begin{equation}
  \label{eq-maury160113}
  \Cb^c_{N,K}(v,v')=\E(\Xc_{v,N,K} \Xc_{v',N,K})\,,\quad
v,v'\in V_N
\,.
\end{equation}
Consider the unit square $W=[0,1]^2$ and 
tile it with open squares 
$\{W^i\}_{i=1,\ldots,K^2}$ of side length $1/K$ each.
(We omit $K$ from the notation.)  
Define $W^{\delta,i}$ to be the 
closed subset
of $W^i$ consisting of points whose 
{$\ell_\infty$}-distance to the boundary
of $W^i$ is at least $\delta/K$, and introduce the
closed set $W^\delta=\cup_i W^{\delta,i}$.
Note that when $N,K,1/\delta$ are all powers of $2$, we have
$W^\delta=\mbox{\rm cl}(\{x/N: [x]\in V_N^{K,\delta}\})$; similarly,
$W^{i}=
\{x/N: [x]\in V_N^{K,i}\}^o$
and
$W^{\delta,i}=
\mbox{\rm cl}(\{x/N: [x]\in V_N^{K,\delta,i}\})$. 

In what follows, we let
$\{w_t\}_{t\geq 0}$ denote planar Brownian motion and we write
$\P^x (\E^x)$ for probabilities (expectations) involving the path
of $w_t$, with $w_0=x$.
Set $\tau_i=\min\{t: w_t\in \partial W^{i}\}$
and
$\tau=\min\{t: w_t\in \partial W\}$.
Introduce the Poisson kernels $p(x,z)$ and $p_i(x,z)$ as the functions
satisfying
$$\int_{\partial W} p(x,z) f(z) dz=\E^x(f(w_\tau))\,,
\quad
\int_{\partial W^{i}} p_i(x,z) g(z) dz=\E^x(g(w_{\tau_i}))\,,
$$
for any continuous functions $f,g$, 
{where $dz$ is one-dimensional Lebesgue
measure}. These Poisson kernels give
the exit measures of Brownian motion from $W$ and $W^i$.
For $x,x'\in \cup_\delta W^\delta$, set
$$
{\Cb}^c_{K}(x,x')=\left\{
\begin{array}{ll}
  \frac{2}{\pi}\left(\int_{\partial W} p(x,z)
\log|z-x'|dz-\int_{\partial W^i}p(x,z)\log|z-x'|dz \right)\,,
&
x,x'
\in W^{i} \, \mbox{\rm for some } i\\
\frac{2}{\pi}\left(\int_{\partial W} p(x,z)\log\left(\frac{|z-x'|}{|x-x'|}
\right) dz\right)\,,&
\mbox{\rm otherwise}
\end{array}\right.
\,.$$
Note that, for each fixed $\delta$,
$
{\Cb}^c_K$ is uniformly continuous on $W^\delta\times W^\delta$; in fact, it follows from
the definition that for some constant $c_\delta>0$,
\begin{equation}
        \label{eq-page4}
        \sup_{x\in W^\delta}
        {\Cb}_K^c(x,x)\leq \frac2\pi\log K+c_\delta\,.
\end{equation}

The following result is crucial for our approach and is easily verified.
\begin{lemma}
  \label{lem-coarselimit}
  Fix $\delta,K>0$. Then
  \begin{equation}
    \label{eq-of2}
   \sup_{v,v'\in (V_N^{K,\delta})^2}  |\Cb^c_{N,K}(v,v')-\Cb^c_{K}(v/N,v'/N)|\to_{N\to\infty} 0\,.
  \end{equation}
\end{lemma}
Note that
the limit $\Cb^c_{K}$ depends on $K$ and that
the convergence rate depends also
on $\delta$.
\begin{proof}
  Employing the orthogonal representation
$\eta_{v,N}=X_v^c+X_v^f$,
   \begin{equation}
     \label{eq-110113g}
     \E(X_v^c  X_{v'}^c)=
     \E(\eta_{v,N} \eta_{v',N})-
     \E(X_v^f  X_{v'}^f)\,.
   \end{equation}
 Recall that
 $$  \E(\eta_{v,N}\eta_{v',N})=
   \E^v(\sum_{n=0}^{\tau_N} {\bf 1}_{\{S_n=v'\}})\,,$$
   where, under the measure $\E^v$,
   $\{S_n\}$ denotes two-dimensional simple random walk starting from $v$ 
   and
   $\tau_N$ is the first exit time from $V_N$ (see
   \cite{BZ10} or \cite{DZ12}).
   Consider the potential kernel for two-dimensional simple random
   walk,
   \begin{equation}
     \label{eq-100113c}
     a(x)=\frac{2}{\pi}\log |x|+
   \frac{2\bar\gamma+\log 8}{\pi}+O(|x|^{-2})\,,
 \end{equation}
 with $a(0)=0$, and where $\bar\gamma$ is the Euler constant
 (see \cite[Theorem 4.4.4]{LL10}).
 From \cite[Theorem 4.6.2]{LL10},
 \begin{equation}
   \label{eq-lawler1}
   \E(\eta_{v,N} \eta_{v',N})=
   \E^v(\sum_{n=0}^{\tau_N} {\bf 1}_{\{S_n=v'\}})=
   \sum_{z\in\partial V_N} \P^v(S_{\tau_N}=z) a(z-v')-a(v'-v)\,.
 \end{equation}
 Similarly,
   when $v,v'\in V_N^{K,i}$ for some $i$,
 \begin{equation}
   \label{eq-lawler2}
   \E(X_v^f X_{v'}^f)=
   \sum_{z\in\partial V_N^{K,i}} \P^v(S_{\tau_N^{(i)}}=z) a(z-v')-a(v'-v)\,,
 \end{equation}
 where $\tau_N^i$ denotes the first exit time from $V_N^{K,i}$.
 On the other hand, when $v,v'$ do not belong to the same
 box $V_N^{K,i}$,
   $\E(X_v^f X_{v'}^f)=0$.
The conclusion follows from
the convergence of simple random walk to Brownian motion, 
the uniform continuity of $a(\cdot)$ on $\{z: \delta<|z|\leq 1\}$,
\eqref{eq-100113c}, \eqref{eq-lawler1} and \eqref{eq-lawler2}.
\end{proof}

It follows from Lemma \ref{lem-coarselimit}
that $\Cb^c_{K}$ is a
covariance function and therefore there exists a mean zero Gaussian field
$\{Z^c_{K,\delta}(x)\}_{x\in W^{K,\delta}}$ with covariance $\Cb^c_{K}$.

\subsection{The fine field limit}
We fix a function $g(K)$ that grows to $\infty$  with $K$.
(The choice $g(K)=\alpha \log \log K$, with an appropriately chosen
$\alpha>0$, will be used in Proposition \ref{prop-local}.) 
The following two propositions on the right tail behavior of
$\max_v \eta_{N/K}$ will be
demonstrated in Section \ref{sec-limittail}.  
They will be important ingredients for the
proof of Theorem \ref{limit-law} in several places in
Section \ref{sec-proof-of-theorem}.
\begin{prop}
  \label{prop-jian}
  Define the event
  $${\cal A}_{N,K}=\{\max_{v\in V_{N/K}} \eta_{v,N/K}\geq m_{N/K}+g(K)\}\,.$$
  There exists an absolute constant $\alpha^*>0$ so that
  \begin{equation}
    \label{eq-301212a}
    \lim_{K\to\infty}\limsup_{N\to\infty}
    \frac{\mathrm{e}^{\sqrt{2\pi} g(K)}}{g(K)} \P( {\cal A}_{N,K})
    =\lim_{K\to\infty}\liminf_{N\to\infty}
    \frac{\mathrm{e}^{\sqrt{2\pi} g(K)}}{g(K)} \P( {\cal A}_{N,K})
    = \alpha^*\,.
  \end{equation}
  Choose $v^*=v^*_{N/K}$ so that
  $\max_{v\in V_{N/K}} \eta_{v,N/K}= \eta_{v^*,N/K}.$
There exists a continuous function $\psi:(0,1)^2\to (0,\infty)$,
  with $\int_{(0,1)^2} \psi(y) dy=1$, such that, for any open
  set $A\subset (0,1)^2$ and any sequence $x_K\geq 0$
  not depending on $N$,
  \begin{eqnarray}
    \label{eq-301212b}
    &&
    \lim_{K\to\infty}\limsup_{N\to\infty}
    \frac{\mathrm{e}^{\sqrt{2\pi} x_K} g(K)}{g(K)+x_K}
    \P(
  \max_{v\in V_{N/K}} \eta_{v,N/K}\geq m_{N/K}+g(K)+x_K, Kv^*/N\in A|
    {\cal A}_{N,K})\nonumber\\
    &&=
    \lim_{K\to\infty}\liminf_{N\to\infty}
    \frac{\mathrm{e}^{\sqrt{2\pi} x_K} g(K)}{g(K)+x_K}
    \P(
  \max_{v\in V_{N/K}} \eta_{v,N/K}\geq m_{N/K}+g(K)+x_K, Kv^*/N\in A|
    {\cal A}_{N,K})\nonumber\\
    &&= \int_A \psi(y)dy\,,
  \end{eqnarray}
 with convergence being uniform in the sequence $x_K$.
\end{prop}
Fix $\delta>0$ small and recall that 
$$\mcW_{N}^\delta=\{x\in V_{N}: 
d_\infty(x, \partial V_{N})\geq \delta N\}\,.$$
Note that $\mcW_N^\delta$ consists of the points in $V_N$ which are at distance
at least $\delta N$ from the boundary of $V_N$; in particular, 
$\mcW_N^\delta$ is not the rescaled version of $W^\delta$.
In the construction of the limit process, we will need a version of 
Proposition
\ref{prop-jian} that restricts attention to $\mcW_{N/K}^\delta$. 
\begin{prop}
  \label{prop-jian-delta}
  Fix $\delta >0$ small and
  let $\alpha^*, \psi$ be as in Proposition \ref{prop-jian}.
  Set $m_\delta=\int_{(\delta,1-\delta)^2}\psi(y) dy$ and
  $\psi^\delta=\psi/m_\delta$.
  Define the event
  $${\cal A}_{N,K,\delta}=
  \{\max_{v\in \mcW_{N/K}^\delta} \eta_{v,N/K}\geq m_{N/K}+g(K)\}\,.$$
  Then
  \begin{equation}
    \label{eq-301212a-neq}
    \lim_{K\to\infty}\lim_{N\to\infty}
    \frac{\mathrm{e}^{\sqrt{2\pi} g(K)}}{g(K)} \P( {\cal A}_{N,K})= 
    \alpha^*
    m_\delta\,.
  \end{equation}
  Choose
 $v^*_\delta=v^*_{N/K,\delta}$ so that
  $\max_{v\in \mcW_{N/K}^\delta} \eta_{v,N/K}= \eta_{v^*_\delta,N/K}$.
  For any open
  set $A\subset (\delta,1-\delta)^2$ and any sequence $x_K\geq 0$
  not depending on $N$,
  \begin{equation}
    \label{eq-301212b-new}
    \lim_{K\to\infty}\lim_{N\to\infty}
    \frac{\mathrm{e}^{\sqrt{2\pi} x_K} g(K)}{g(K)+x_K}
    \P(
    \max_{v\in V_{N/K}^\delta} \eta_{v,N/K}\geq m_{N/K}+g(K)+x_K, Kv^*_\delta/N
  \in A|
    {\cal A}_{N,K,\delta})= \int_A \psi^\delta(y)dy\,,
  \end{equation}
 with convergence being uniform in the sequence $x_K$.
\end{prop}
\subsection{The limit process}
\label{subsec-limit}

Let $\psi^\delta$ and $\alpha^*$ be as in Proposition \ref{prop-jian-delta}.
In each $W^{\delta,i}$, choose a point 
$z_i^{K,\delta}$ that is distributed according to
the scaled analog of $\psi^\delta$,
with $z_i^{K,\delta}$ being chosen independently for different $i$.
Let $\{\B_i^{K,\delta}\}_{i=1,\ldots,K^2}$ 
denote independent Bernoulli random variables
with $\P(\B_i^{K,\delta}=1)= \alpha^*m_\delta
g(K)\mathrm{e}^{-\sqrt{2\pi}g(K)}$, and let
$\{Y_i^K\}_{i=1,\ldots,K^2}$ denote independent random variables
satisfying
\begin{equation}
  \label{eq-indeptails}
  \P(Y_i^K\geq x)=\frac{g(K)+x}{g(K)} \mathrm{e}^{-\sqrt{2\pi}x}\,,\quad
  x\geq 0\,.\end{equation}
Recall the limiting coarse field $Z^c_{K,\delta}$  and
define
$$G_i^{K,\delta}=
\B_i^K(Y_i^K+g(K))+Z^c_{K,\delta}(z_i^{K,\delta})-2\sqrt{2/\pi}\log K.$$
Set $G^*_{K,\delta}=\max_i G_i^{K,\delta}$ and denote by $\mu_{K,\delta}$ its law.
Note that $\mu_{K,\delta}$ does not depend on $N$.

Let $d(\cdot,\cdot)$ denote the L\'evy metric, 
which is compatible with the weak convergence
of probability measures on $\R$.
Our main result in the paper is the following.
\begin{theorem}
  \label{limit-law}
  Let $\mu_N$ denote the law of $\max_{v\in V_N}\eta_{v}^N-m_N$.
  Then, with $\mu_{K,\delta}$ defined as above,
  \begin{equation}
    \label{eq-limiteq}
    \lim_{\delta\searrow
    0} \limsup_{K\to\infty}
    \limsup_{N\to\infty}
    d(\mu_N,\mu_{K,\delta})=0\,.
  \end{equation}
  In particular, there exists a probability measure $\mu_\infty$ on $\R$
  such that
  $d(\mu_N,\mu_\infty)\to_{N\to\infty} 0$.
\end{theorem}
(The choice of working with the L\'evy metric in Theorem
\ref{limit-law} is for concreteness; any distance
compatible with weak convergence could be used.)
The limit $\mu_\infty$ in Theorem \ref{limit-law} can be described as follows.
A similar description was derived in \cite{biskup}.
\begin{theorem}
        \label{explicit-limit-law}
        There exist a random variable $Z>0$ and a constant $\alpha^* > 0$ 
        so that
        \begin{equation}
                \label{eq-July182014-thm}
        \mu_\infty\left( (-\infty,x] \right)=
        \E\left(e^{-\alpha^* Ze^{-\sqrt{2\pi}x}}\right)\,.
\end{equation}
\end{theorem}
Thus, the limit law of the maximum of the GFF is a randomly 
shifted Gumbel distribution.
{(In the statement of Theorem \ref{explicit-limit-law}, the constant $\alpha^*$ can be absorbed 
        into $Z$.
        Since $\alpha^*$ 
is determined by the tail behavior of the GFF (see 
Proposition \ref{prop-limiting-tail-gff}) whereas the random variable
$Z$ is determined 
by the coarse field, they play very different roles in the proof; for this reason, we keep both in the statement.)  Note that the random variable $Z$ 
plays the role of the derivative martingale in the theory of branching
Brownian motion, as in \cite{LS}.}

\section{Preliminaries}
\label{sec-prelim}
\subsection{Branching random walk and modified branching random walk}\label{sec:MBRW}
We first briefly review the construction of branching random walk (BRW)
and modified branching random walk (MBRW), which we construct
so as to simplify comparisons with the GFF.  As before, we let $V_N$ denote the box of side
length $N$ intersected with $\mathbb{Z}^2$ such that the lower left corner is at the
origin, and set $n = \log_2  N$.
For $j\in [0,\ldots,n-1]$, let $\mathfrak{B}_j$ be the collection of squared boxes in $\mathbb{Z}^2$ of side length $2^j$ with corners in $\mathbb{Z}^2$, and let $\mathfrak{BD}_j$ denote the subset of $\mathfrak{B}_j$ consisting of squares of the form $([0, 2^{j}-1] \cap \mathbb{Z})^2 + (i_1 2^j, i_2 2^j)$, with $i_1, i_2 \in \mathbb{N}$. For $v\in \mathbb{Z}^2$, let $\mathfrak{B}_j (v) = \{B\in \mathfrak{B}_j: v\in B\}$ be the collection of boxes in $\mathfrak{B}_j$ that contain $v$, and define $\mathfrak{BD}_j(v)$ to be the (unique) box in $\mathfrak{BD}_j$ that contains $v$. Furthermore, denote by $\mathfrak{B}_{N, j}$ the subset of $\mathfrak{B}_j$ consisting of boxes whose lower left corners are in $V_N$.
Let $\{\bar \phi_{N, j, B}\}_{j\geq 0, B\in \mathcal{BD}_j}$ be i.i.d. mean zero Gaussian variables with variance $\frac{2\log 2}{\pi}$, and define a branching random walk
\begin{equation}\label{eq-def-BRW}
        \vartheta_{v,N} = \mbox{$\sum_{j=0}^{n-1}$} \bar\phi_{N, j, \mathcal{BD}_j(v)}\,.
\end{equation}
For $j \in [0,\ldots,n-1]$
and $B\in \mathfrak{B}_{N, j}$, let $\phi_{N, j, B}$ be independent mean zero Gaussian variables with $\var (\phi_{N, j, B}) = \tfrac{2\log 2}{\pi}  \cdot 2^{-2j}$, and define
\begin{equation}\label{eq-define-b}
\phi_{N, j, B}=\phi_{N, j, B'}, \mbox{ for } B\sim_N B' \in \mathcal{B}_{N, j}\,,
\end{equation}
where $B\sim_N B'$ if and only if there exist $i_1, i_2\in \mathbb{Z}$ such that $B = (i_1 N, i_2 N) + B'$. (Note that, for any $B\in \mathcal{B}_j$, there exists a unique $B'\in \mathfrak{B}_{N, j}$ such that $B\sim_N B'$.)
Let $d_N(u, v) = \min_{w\sim_N v}|u-w|$ be the $\ell^2$ distance between $u$ and $v$ when considering $V_N$ as a torus, for all $u, v\in V_N$.
Finally, we define the MBRW $\{\xi_{v,N}: v\in V_N\}$ by
\begin{equation}\label{eq-MBRW}
        \xi_{v,N}= \mbox{$\sum_{j=0}^{n-1}
        \sum_{B\in \mathfrak{B}_j(v)}$} \phi_{N, j, B}\,.
\end{equation}
The motivation for introducing MBRW is that the MBRW approximates the GFF with high precision. That is, the covariance structure
of the MBRW approximates that of the GFF well. This
is elaborated in the next lemma.
\begin{lemma}\label{lem-covariance}
For any $0<\delta<1/100$,
there exists a constant $C = C_\delta$ such that, for all $n$,
\begin{align*}
|\cov(\xi_{u,N}, \xi_{v,N}) - \tfrac{2\log 2}{\pi} (n - \log_2(d_N(u, v)))|& \leq C, \mbox{ for all } u, v\in V_N\,,\\
\big|\cov(\eta_{u,N}, \eta_{v,N}) - \tfrac{2\log 2}{\pi} (n - (0\vee \log_2 |u-v|))\big| &\leq C, \mbox{ for all } u, v \in V'_N\,,
\end{align*}
where $V'_N$ is a box of side length $(1-2\delta)N$ placed at the center of $V_N$.
\end{lemma}
The first part of the lemma is in
\cite[Lemma 2.2]{BZ10},  whereas the second part can be found in 
\cite[Lemma 2.1]{Da06} (and in a slightly different form, dealing with the so called torus GFF, in \cite[Lemma 2.2]{BZ10}).

\subsection{A few Gaussian inequalities}
We will need the following two Gaussian comparison inequalities: the Sudakov-Fernique inequality, which compares the expected maximum of Gaussian processes (see, e.g., \cite{Fernique75} for a proof), and Slepian's comparison lemma \cite{Slepian62}, which compares the maximum of Gaussian processes in the sense of  ``stochastic domination''.
\begin{lemma}[Sudakov-Fernique]\label{lem-sudakov-fernique}
Let $\mathcal{A}$ be an arbitrary finite index set and let $\{X_a\}_{a\in \mathcal{A}}$ and $\{Y_a\}_{a\in \mathcal{A}}$ be two centered Gaussian processes such that
\begin{equation}\label{eq-compare-assumption}
\E(X_a - X_b)^2 \geq \E (Y_a - Y_b)^2\,, \mbox{ for all } a, b \in \mathcal{A}\,.
\end{equation}
Then $\E \max_{a\in \mathcal{A}} X_a \geq \E \max_{a\in \mathcal{A}} Y_a$.
\end{lemma}

\begin{lemma}[Slepian]\label{lem-slepian}
Let $\mathcal{A}$ be an arbitrary finite index set and let $\{X_a\}_{a\in \mathcal{A}}$ and $\{Y_a\}_{a\in \mathcal{A}}$ be two centered Gaussian processes such that \eqref{eq-compare-assumption} holds and $\var X_a = \var Y_a$ for all $a\in \mathcal{A}$.
Then $\P(\max_{a\in \mathcal{A}} X_a \geq \lambda) \geq \P(\max_{a\in \mathcal{A}} Y_a \geq \lambda)$, for all $\lambda\in \mathbb{R}$.
\end{lemma}

The following Borell-Tsirelson inequality is a central result in the theory of concentration of
measure; see, for example, \cite[Theorem 7.1, Equation (7.4)]{Ledoux89}.

\begin{lemma}\label{lem-gaussian-concentration}
        Consider a Gaussian process $\{\eta_x: x\in \mathcal{A}\}$, with 
        $\mathcal{A}$ 
finite, and set $\sigma
= \sup_{x \in \mathcal{A}} (\E(\eta_x^2))^{1/2}$. Then, for $\alpha>0$,
\[\P\left(\left|\max_{x\in \mathcal{A}} \eta_x - \E \,\max_{x \in \mathcal{A}
} \eta_x\right| > \alpha\right) \leq 2 \exp(-\alpha^2 / 2\sigma^2)\,.\]
\end{lemma}

We will often need to control the expectation of the
maximum of a Gaussian field in terms
of its covariance structure. This is achieved by Fernique's criterion
\cite{Fernique75}.
We quote a version suited to our needs, which follows straightforwardly
from the version in
\cite[Theorem 4.1]{adler} by
using, as the majorizing measure, the normalized
counting measure on $B$.
\begin{lemma}
  \label{lem-ferniquecriterion}
  There exists a universal constant $C_F>0$ with the following property.
  Let $B\subset\mathbb{Z}^2$ denote a (discrete) box of side length
  $b$ and assume $\{G_v\}_{v\in B}$
  is a mean zero Gaussian field satisfying
  $$\E(G_v-G_u)^2\leq |u-v|/b\,, \mbox{ for all } u, v \in B\,.$$
  Then
  \begin{equation}
    \label{eq-ferniquecriterion}
    \E\max_{v\in B} G_v \leq C_F \,.
  \end{equation}
\end{lemma}

\subsection{A Brownian motion estimate}

In (\ref{eq-mu-mu*-asympt}), we show that the probabilities for Brownian motion to stay below two close curves are asymptotically the same and, in  (\ref{eq-for-correction-1})
and (\ref{eq-BM-monotone}), we give bounds on the corresponding probabilities.  In Section \ref{sec-limittail}, the asymptotic probabilities, as $N\rightarrow\infty$, of the events in (\ref{eq-big-definition}) will be strongly tied to these bounds.
\begin{lemma}\label{lem-1DRW}
Let $C\geq 0$ be a fixed absolute constant, and
let $\{W_s: s\geq 0\}$ be a mean zero Brownian motion, started at
$0$, with variance $\sigma^2s$, $\sigma >0$, at time $s$.
For $y>1$ and $t>0$, define densities
$\mu_{t, y}(\cdot)$ and $\mu_{t, y}^*(\cdot)$
such that, for all intervals $I$,
\begin{equation} \label{eq-def-mu-mu*}
\begin{split}
\P&(W_t\in I; W_s \leq y \mbox{ for all } 0\leq s\leq t) = \int_I \mu_{t, y}(x) dx\,,\\
\P&(W_t\in I; W_s \leq y + y^{1/20} + C (s \wedge (t-s))^{1/20} \mbox{ for all } 0\leq s\leq t) = \int_I \mu^*_{t, y}(x) dx\,.
\end{split}
\end{equation}
Then there exists $\delta_y$, with $\delta_y \searrow_{y\to \infty} 0$, such that, for all $x\leq 0$ and $t>0$,
\begin{equation}\label{eq-mu-mu*-asympt}
 \mu_{t, y}^*(x) \leq (1+ \delta_y) \mu_{t, y}(x)\,.
\end{equation}
For all $x \leq y + y^{1/20}$,
\begin{equation}\label{eq-for-correction-1}
        \mu^*_{t, y}(x) \lesssim y (y+y^{1/20} - x) t^{-3/2} e^{-x^2/2t\sigma^2}\,.
\end{equation}
Furthermore,
\begin{equation}\label{eq-BM-monotone}
\frac{\mu^*_{t, y}(x_1)}{ \mu^*_{t, y}(x_2)} \leq \mathrm{e}^{-\frac{x_1^2 - x_2^2}{2t
\sigma^2 }} \mbox{ for all } 0\leq x_2 \leq x_1 \leq y + y^{1/20}\,.
\end{equation}
\end{lemma}
Note that by setting $C=0$ and $\bar y=y+y^{1/20}$,   
\eqref{eq-BM-monotone} also implies that 
\begin{equation}\label{eq-BM-monotone-bis}
\frac{\mu_{t, \bar y}(x_1)}{ \mu_{t, \bar y}(x_2)} \leq 
\mathrm{e}^{-\frac{x_1^2 - x_2^2}{2t
\sigma^2 }} \mbox{ for all } 0\leq x_2 \leq x_1 \leq \bar y\,. 
\end{equation}
\begin{proof}[Proof of Lemma \ref{lem-1DRW}]
 In order to show \eqref{eq-BM-monotone},  note that $\{W_s - sW_t/t: 0\leq s\leq t\}$ is distributed as a Brownian bridge that is independent of the value $W_t$. 
{So, the probability that the Brownian motion $\{W_s:  0\leq s\leq t\}$  stays below a given curve, after conditioning on $W_t = x$, is decreasing with $x$. Combined with the fact that $\mu^*_{t, y}(x)>0$ for all $x\leq y + y^{1/20}$, } this implies the ratio between $\mu^*_{t, y}(x_1)$ and $\mu^*_{t, y}(x_2)$
is bounded above by the ratio of the densities for the Brownian motion at $x_1$ and $x_2$.

We next prove \eqref{eq-mu-mu*-asympt}. 
By the reflection principle, for all 
$w \geq 0$,
\begin{equation}\label{eq-reflection-principle}
\mu_{t, y}(y-w) = \frac{1}{\sqrt{2\pi t} \sigma} (\mathrm{e}^{-\frac{(w-y)^2}{2t \sigma^2}} - \mathrm{e}^{-\frac{(w+y)^2}{2t \sigma^2}}).
\end{equation}
 Let $\tau$ be a global maximizer of $\{W_s: 0\leq s\leq t\}$,
 i.e., $W_\tau=\max_{0\leq s\leq t} W_s$,
 and write $\psi_{t, y}(s) = y + y^{1/20} + C (s \wedge (t-s))^{1/20} $ and $\psi_{t, y}^*(j) = \max_{s\in [j, j+1]} \psi_{t, y}(s)$.
Summing over different $j$ for $\tau \in [j, j+1]$ and
integrating over locations for $W_j$, we obtain from the Markov property
applied at time $j$ and then again at time 
$\min\{s\geq j: W_s\geq y\}\wedge (j+1)$
that, for all $x\leq 0$,
\begin{align*}
&\mu_{t, y}^*(x) - \mu_{t, y}(x)\\
& \leq \sum_{j=0}^{\lfloor t \rfloor}\int_{-\infty}^{\psi_{t, y}^*(j)} 
\mu_{j, \psi_{t, y}^*(j)}(\lambda)
\P(\lambda + \max_{0\leq s\leq 1} W_s
\geq y) \max_{j\leq s\leq j+1} 
\, 
\max_{y\leq y'\leq \psi_{t, y}^*(j)}\mu_{t-s, \psi_{t, y}^*(j) - y'}(x-y') d\lambda\,.
\end{align*}
{In the preceding inequality, the sum over $j$ corresponds to the intervals $[j, j+1)$, with the Brownian motion achieving its maximum at $\tau\in [j, j+1).$ (Note that, on the relevant event, $y\leq W_\tau \leq \psi_{t, y}(\tau)$,  integration over $\lambda$ includes all possible locations of the Brownian motion at time $j$; the term $\P(\lambda + \max_{0\leq s\leq 1} W_s
\geq y) $ bounds the probability for the Brownian motion to hit $y$ (i.e., $W_\tau \geq y$) given its location at time $j$; and the term $\max_{y\leq y'\leq \psi_{t, y}^*(j)}\mu_{t-s, \psi_{t, y}^*(j) - y'}(x-y') $ bounds the probability for the Brownian motion to stay below $W(\tau)$ over the time interval $(\tau, t]$.) }
By \eqref{eq-BM-monotone-bis}, $\max_{y\leq y'\leq \psi_{t, y}^*}\mu_{t-s, \psi_{t, y}^*(j) - y'}(x-y') \leq \mu_{t-s, \psi_{t, y}^*(j) - y}(x-y)$. Therefore,
\begin{align}\label{eq-integration}
\mu_{t, y}^*(x) - \mu_{t,y}(x)
&\lesssim \sum_{j=0}^{\lfloor t \rfloor}\int_{-\infty}^{\psi_{t, y}^*(j)}  \mu_{j, \psi_{t, y}^*(j)}(\lambda)\mathrm{e}^{-\frac{(\lambda - y)_-^2}{2 \sigma^2}}\max_{j\leq s\leq j+1}\mu_{t-s, \psi_{t, y}^*(j) -y}(x-y) d\lambda
=:\sum_{j=0}^{\lfloor t \rfloor} \Phi_j\,,
\end{align}
where $(\lambda - y)_- = \one_{\lambda \leq y} |\lambda - y|$.
Applying \eqref{eq-reflection-principle}, we obtain that, for both $j\leq y^2/(\log y)^2$ and $j\geq t - y^2/(\log y)^2$,
\begin{equation}\label{eq-Phi-j-1}
\Phi_j \lesssim \mathrm{e}^{-(\log y)^2/4 \sigma^2} (y-x) t^{-3/2}\,.
\end{equation}
Again, by the reflection principle, for all $\lambda \leq \psi_{t, y}^*(j)$, $x\leq 0$ and $j-1\leq s\leq j$,
\begin{align*}
\mu_{j, \psi_{t, y}^*(j)} (\lambda) &\lesssim j^{-3/2} (\psi_{t, y}^*(j) -\lambda) \psi_{t, y}^*(j) \\
\mu_{t-s, \psi_{t, y}^*(j) - y} (x-y)& \lesssim (\psi_{t, y}^*(j) - y) (\psi_{t, y}^*(j) - x) (t-j)^{-3/2}\,.
\end{align*}
For $y^2 /(\log y)^2\leq j\leq t-y^2/(\log y)^2$,
plugging the above estimates into the integral in
\eqref{eq-integration} leads to
\begin{align}\label{eq-Phi-j-2}
\Phi_j &\lesssim (\psi_{t, y}^*(j) - y)^3  j^{-3/2} (t-j)^{-3/2} (\psi_{t, y}^*(j)  - x) \psi_{t, y}^*(j) \,.\end{align}
Plugging \eqref{eq-Phi-j-1} and \eqref{eq-Phi-j-2} into
\eqref{eq-integration} and summing over $j$, we obtain
\begin{align*}
\mu_{t, y}^*(x) - \mu_{t, y}(x) \lesssim (y-x) y t^{-3/2} (y^2 \mathrm{e}^{-\frac{(\log y)^2}{4\sigma^2}} +y^{-1/2} \log y) \lesssim \mu_{t, y}(x) (y^2 \mathrm{e}^{-\frac{(\log y)^2}{4\sigma^2}} +y^{-1/2} \log y)\,,
\end{align*}
which completes the proof of \eqref{eq-mu-mu*-asympt} (where we have
used $\mu_{t, y}(x) \asymp y(y-x) t^{-3/2}$).

It remains to show \eqref{eq-for-correction-1}. The result follows by using the same decomposition for the location of the global maximizer $\tau$ of Brownian motion in a manner analogous to the proof of \eqref{eq-mu-mu*-asympt}, together with similar straightforward computations. We omit further details.
\end{proof}

\subsection{Refined estimates on the right tail of the maximum for the GFF}

In \cite{DZ12}, it is shown that
\begin{equation}\label{eq-right-tail}
\P(\eta^*_N> m_N+\lambda)
\asymp \lambda\mathrm{e}^{-\sqrt{2\pi} \lambda} \mbox{ for all } 1\leq \lambda \leq \sqrt{\log N}\,.
\end{equation}
In this subsection, we give a preliminary upper bound on the right tail of the
maximum of GFF over subsets of $V_N$, for values of $\lambda$
that include $\lambda \gg \sqrt{\log N}$.
To do this, we first obtain an upper bound on the probability of BRW taking atypically large values.
We will use the notation of
Subsection~\ref{sec:MBRW} and, for convenience, will view each $\vartheta_{v,N}$ as the value
at time $n=\log_2 N$ of a mean zero Brownian motion $\{\vartheta_{v,N}(t): 0\leq t\leq n\}$ with variance rate $\frac{2\log 2}{\pi}$.
More precisely, we associate to each Gaussian variable $\bar\phi_{N, j, B}$ an independent Brownian motion with
variance rate $\frac{2\log 2}{\pi}$ that runs for one unit of time and ends
at the value $\bar\phi_{N, j , B}$, 
{with
$$\vartheta_{v,N}(t) := \sum_{j=0}^{[t]} \bar\phi_{N, j, \mathcal{BD}_j(v)} + \bar \phi_{N, j, \mathcal{BD}_j(v)}(t - [t])\,.$$}
For $\beta>0$, define
\begin{align}\label{eq-def-G-N-prelim}
G_{N}(\beta) &= \bigcup_{v\in V_N}\bigcup_{0 \leq t \leq n}\{\vartheta_{v, N}(t) \geq \beta + 1 + \frac{m_N}{n}t+ 10 (\log(t \wedge (n-t)))_+\}\,.
\end{align}

\begin{lemma}\label{lem-a-priori}
There exists an absolute constant $C>0$ such that
$\P(G_{N}(\beta))  \lesssim (\beta\vee 1)
\mathrm{e}^{-\sqrt{2\pi} \beta} \mathrm{e}^{-\beta^2/Cn}$ for all $\beta>0$.
\end{lemma}
\begin{proof}
We may assume that $\beta>\beta_0$, for some $\beta_0>0$ large.
 For any $v\in V_N$, write $\bar \vartheta_{v,N}(t) = \vartheta_{v,N}(t) - \tfrac{m_N t}{n}$.
{For $j=1,\ldots,n$,
         define the probability measure $\mathbb{Q}_j$ by
\begin{equation}\label{eq-change-of-measure-1}
\frac{d\mathbb{Q}_j}{d\mathbb{P}} = \mathrm{e}^{\tfrac{\pi m_N}{(2\log 2)
n}
\bar \vartheta_{v,N}(j) + \tfrac{\pi m_N^2}{(4\log 2) n}\cdot\tfrac{j}{n} }\,.
\end{equation}
For $j=n$, we omit the subscript $j$ from the notation.}
Then, under 
{$\mathbb{Q}_j$},
{$\{\bar \vartheta_{v,N}(t): 0\leq t\leq j\}$}
is a mean zero Brownian motion
with variance rate $\frac{2\log 2}{\pi}$.
{We write $\frac{d\mathbb{Q}_j}{d\mathbb{P}}(x)$ for 
the expression in \eqref{eq-change-of-measure-1} when $\bar\vartheta_{v,N}(j)=x$.}

For $0\leq t\leq n$, write $\psi_{N, t, \beta} = \beta+1+
10 (\log (t \wedge (n-t)))_+
$. For $j \in [1,\ldots,n]$, let $\chi_{N, j}(\cdot)$ be the density function such that, for all $I\subseteq \mathbb{R}$,
$$\P(\bar \vartheta_{v,N}(t) \leq \psi_{N, t, \beta} \mbox{ for all } t\leq j, \bar \vartheta_{v,N}(j) \in I ) = \int_I \chi_{N,j}(x) dx\,.$$
By a straightforward union bound,
\begin{align}\label{eq-G-N-v}
        \nonumber
        \P(G_{N}(\beta)) &\leq \sum_{j=1}^{n-1} 4^{j+1}
        \int_{-\infty}^{\psi_{N,j,\beta}} \chi_{N, j}(x) 
        \P(
        {\exists s\in [0, 1]: x +\gamma 
                W(s)    \geq \psi_{N, s, \beta}}) dx\\
        & \quad +
4       \P(\exists s\in [0, 1]:  \bar \vartheta_{v,N}(s)  
            \geq \psi_{N, s, \beta}) 
        \,,
\end{align}
{where $\{W_s: 0\leq s\leq 1\}$ is an independent 
Brownian motion and $\gamma =\sqrt{2\log 2}$.}
Applying \eqref{eq-for-correction-1} (with $j$ in place of $t$ and $\psi_{N, j, \beta}$ 
in place of $y + y^{1/20}$), we see that, for $x\leq \psi_{N, j, \beta}$ and $j\geq 1$,
\begin{equation}
        \label{disp-july162014}
        (\frac{d \P}{d \mathbb{Q}_{j}} (x))^{-1} \chi_{N, j}(x) \lesssim j^{-3/2} \psi_{N, j, \beta} (\psi_{N, j, \beta} - x)e^{-\pi x^2/4j\log 2 }\,.
\end{equation}
Combined with \eqref{eq-change-of-measure-1} 
{and using that $(j/n)\log n\leq \log j$},  this yields
\begin{equation}\label{eq-nu}
\chi_{N, j}(x) \lesssim 4^{-j} \mathrm{e}^{-(\sqrt{2\pi}- O((\log n)/n))x} \mathrm{e}^{-\pi x^2/4j\log 2 } \psi_{N, j, \beta} (\psi_{N, j, \beta} - x) \,.
\end{equation}
{On the other hand,
\begin{align}\label{eq-to-add-details-1}
        \P(\exists s\in [0,1]: x + 
        \gamma W(s)
  \geq \psi_{N, s, \beta})\lesssim \exp(-(\min_{s\in [j, j+1]} \psi_{N, s, \beta} - x)^2/(2\gamma^2))\,.
\end{align}
}
Since the right hand side of \eqref{eq-to-add-details-1} decays (essentially) exponentially in $|\psi_{N, j, \beta} - x|^2$, the typical value of $x$ that will contribute to the integral of \eqref{eq-G-N-v} will be close to $\psi_{N, j, \beta}$. With this observation in mind,
plugging \eqref{eq-to-add-details-1} and \eqref{eq-nu} into \eqref{eq-G-N-v}, we obtain that, for an absolute constant $C>0$,
\begin{align*}
\P(G_{N}(\beta)) &\lesssim \sum_{j=1}^{n} (\beta\vee 1) 4^j 4^{-j} ( j \wedge (n+1-j))^{-2} \mathrm{e}^{-\sqrt{2\pi}\beta}  \mathrm{e}^{-\beta^2/Cn} \lesssim 
(\beta\vee 1) \mathrm{e}^{-\sqrt{2\pi}\beta}  \mathrm{e}^{-\beta^2/Cn}\,,
\end{align*}
where the power 2 in $(j \wedge (n+1-j))^2$ is obtained from the slackness term $10 (\log (j \wedge (n-j)))_+$ (with room to spare).
\end{proof}

\begin{lemma}\label{lem-prelim-tail}
There exist absolute constants $C,C'>0$ such that, for all $N\in \mathbb{N}$ and $z\geq 1$,
\begin{equation}\label{eq-large-tail}
\P(\max_{v\in V_N} \eta_{v,N}\geq m_N + z ) \le C' z \mathrm{e}^{-\sqrt{2\pi} z} \mathrm{e}^{-C^{-1} z^2/n} \,.
\end{equation}
Furthermore,
\begin{equation}\label{eq-max-subset}
\P(\max_{v\in A} \eta_{v,N}\geq m_N + z -y) \lesssim (\frac{|A|}{|V_N|})^{1/2} z \mathrm{e}^{-\sqrt{2\pi} (z-y)}   \mbox{ for all } z\geq 1,y\geq 0\,,
\end{equation}
for all $A\subseteq V_N$.
\end{lemma}
\begin{proof}
By the same argument as in \cite[Lemma\ 2.6]{DZ12}, for an absolute constant $\kappa \ge 0$ and any $A \subseteq V_N$,
$$\P(\max_{v\in A} \eta_{v,N}\geq \lambda) \leq \P(\max_{v\in 2^\kappa A}\vartheta_{v,2^\kappa N} \geq \lambda) \mbox{ for all } \lambda \in \mathbb{R}\,,$$
where $\{\vartheta_{v,2^\kappa N}\}$ is a BRW on 2D box of side length $2^\kappa N$.  By this inequality and a change of variable (replacing $2^\kappa N$ by $N$), it suffices to prove  \eqref{eq-large-tail} and \eqref{eq-max-subset} for BRW. For $\beta>0$, define
\begin{align*}
F_{ v, N}(\beta, z) &= \{\vartheta_{v,N}(t)\leq \beta + 1 + \frac{m_N}{n}t+ 10 \log(t\wedge (n-t))_+ \mbox{ for all } 0 \leq t \leq n;\vartheta_{v,N} \geq m_N + z\}\,.
\end{align*}
We first prove the BRW version for \eqref{eq-large-tail}, using the notation $\mu_{n, z}(\cdot)$ and $\mu^*_{n, z}(\cdot)$ from \eqref{eq-def-mu-mu*} (with variance rate $\sigma^2 = \frac{2\log 2}{\pi}$). By \eqref{eq-for-correction-1}, and using $d\P/d\mathbb{Q}$ as in Lemma \ref{lem-a-priori} (with $z'$ satisfying $z' + (z')^{1/20} = z+1$),
\begin{align}
\P(F_{v, N}(z, z)) &= \int_{z}^{z+ 1} \frac{d \P}{d\mathbb{Q} }(x+m_N)\mu^*_{n, z'}(x )dx \nonumber\\
&\leq \int_{z}^{z+1} 4^{-n} n^{3/2} \mathrm{e}^{-(\sqrt{2\pi}- O((\log n)/n))x}  \mathrm{e}^{- \pi z^2/4n\log 2} \cdot z n^{-3/2} dx \nonumber\\
&\lesssim 4^{-n} z  \mathrm{e}^{-\sqrt{2\pi} z}\mathrm{e}^{- \pi z^2/4n\log 2} \mathrm{e}^{O((\log n)/n) z} \,.\label{eq-E-v-est}
\end{align}
Inequality \eqref{eq-large-tail} follows by summing the above inequality over $v\in V_N$ and applying Lemma~\ref{lem-a-priori}.

We next demonstrate \eqref{eq-max-subset}. First note that, if $z - y + (|V_N|/|A|)^{1/4} \leq 1$, then \eqref{eq-max-subset} holds automatically. We next consider the case when $z - y + (|V_N|/|A|)^{1/4} \geq 1$ and set $\beta = z - y + (|V_N|/|A|)^{1/4}$. By Lemma~\ref{lem-a-priori},
\begin{equation}\label{eq-G-N-beta}
\P(G_{N}(\beta)) \lesssim (\beta \vee 1)
\mathrm{e}^{-\sqrt{2\pi}(z-y)}  \mathrm{e}^{-\sqrt{2\pi} (|V_N|/|A|)^{1/4}}\,.
\end{equation}
Analogous to the derivation of \eqref{eq-E-v-est}, we obtain
$$\P(F_{v, N}(\beta, z-y)) \lesssim 4^{-n} 
(\beta\vee 1) (|V_N|/|A|)^{1/4}  \mathrm{e}^{-\sqrt{2\pi}(z-y)}\,.$$
{Therefore, we get that
$$\P(\max_{v\in A} \eta_{v,N}\geq m_N + z -y) \lesssim  \P(G_N(\beta)) + \sum_{v\in A} \P(F_{v, N}(\beta, z-y)) \lesssim (\frac{|A|}{|V_N|})^{1/2}  z \mathrm{e}^{-\sqrt{2\pi} (z-y)} \,,$$
where we have used the facts that $\beta \leq 2z (|V_N|/|A|)^{1/4} $ and $|V_N| = 4^n$. This completes the verification of 
\eqref{eq-max-subset}.}
\end{proof}

\subsection{Robustness of the maximum for the GFF}

The following lemma shows a form of robustness for the maximum of the GFF under perturbation.
\begin{lemma}\label{lem-gff-perturb}
Let $\phi_{u, N}$ be independent variables such that, for all $u\in V_N$ 
{and $y\geq 0$},
$$\P(\phi_{u, N} \geq 1+ y) \leq \mathrm{e}^{-y^2}\,.$$
There exist absolute 
{constants $c,C>0$} such that, for any 
$\epsilon>0$, $N\in \mathbb{N}$ and 
{$- c\epsilon^{-1/2} \leq x\leq \sqrt{\log N}$},
\begin{equation}\label{eq-prob-perturb}
\P(\max_{u\in V_N} ( \eta_{u,N} + \epsilon \phi_{ u, N}) \geq m_N+x) \leq   \P(\max_{u\in V_N} \eta_{u,N} \geq m_N+x - \sqrt{\epsilon}) (1+ O(\mathrm{e}^{-C^{-1} \epsilon^{-1}}))\,.
\end{equation}
Furthermore, $\E \max_{u\in V_N} (\eta_{u, N} + \epsilon \phi_{u, N}) \leq \E \max_{u\in V_N} \eta_{u, N} + C\sqrt{\epsilon} + C \epsilon^2 \one_{\epsilon \geq 1}$.
\end{lemma}
\begin{proof}
Consider first the case $\epsilon\leq 1$.
Setting $\Gamma_y = \{u\in V_N: y/2 \leq \epsilon \phi_ {u, N} \leq y\}$, we have
\begin{align}\label{eq-for-ofer}
\P(\max_{u\in  V_N}  (\eta_{u,N} + \epsilon \phi_{u, N}) \geq m_N+x) &\leq  \P(\max_{u\in  V_N} \eta_{u,N} \geq m_N+x- \sqrt{\epsilon}) \nonumber \\
&\qquad+ \sum_{i=0}^\infty \E( \P(\max_{u\in \Gamma_{2^i \sqrt{\epsilon}}} \eta_{u,N} \geq m_N + x - 2^i \sqrt{\epsilon} \mid \Gamma_{2^i \sqrt{\epsilon}})), \,
\end{align}
where the conditioning on the right side of the last
display means conditioning on the locations of points
$u\in V_N$ where the random field $\phi_{u,N}$ takes the
prescribed values.
By Lemma~\ref{lem-prelim-tail} (for
{any $x$, setting
        $z= (x\vee 1)$ and $y = 2^i \sqrt{\epsilon}-(x\vee 1)$}),
\begin{align*}
\sum_{i=0}^\infty \E( \P(\max_{u\in \Gamma_{2^i \sqrt{\epsilon}}} \eta_{u,N} \geq m_N + x - 2^i \sqrt{\epsilon} \mid \Gamma_{2^i \sqrt{\epsilon}}))  \lesssim \frac{x \vee 1}{\mathrm{e}^{\sqrt{2\pi}x}} \sum_{i=0}^\infty \E ( |\Gamma_{2^i \sqrt{\epsilon}}| /N^2)^{1/2} \mathrm{e}^{\sqrt{2\pi} 2^i \sqrt{\epsilon}} \,.
\end{align*}

A simple computation yields
$\E( |\Gamma_{2^i \sqrt{\epsilon}}| /N^2)^{1/2}
\leq \mathrm{e}^{-4^i (C\epsilon)^{-1}}$, for an absolute constant $C>0$.
Combining this with the last two displays, it 
follows that, for an absolute constant 
{$C^*>0$},
\begin{equation*}
\sum_{i=0}^\infty \E( \P(\max_{u\in \Gamma_{2^i \sqrt{\epsilon}}} \eta_{u,N} \geq m_N + x - 2^i \sqrt{\epsilon} \mid \Gamma_{2^i \sqrt{\epsilon}})) \lesssim \frac{
{x\vee 1}}{\mathrm{e}^{\sqrt{2\pi}x}} \mathrm{e}^{- (C^*\epsilon)^{-1}}\,.
\end{equation*}
Together with \eqref{eq-right-tail}, 
{this implies, for 
        $x>-1/(2\sqrt{2\pi}C^*\epsilon)$,
$$\sum_{i=0}^\infty \E( \P(\max_{u\in \Gamma_{2^i \sqrt{\epsilon}}} \eta_{u,N} \geq m_N + x - 2^i \sqrt{\epsilon} \mid \Gamma_{2^i \sqrt{\epsilon}})) \lesssim \P(\max_{u\in  V_N} \eta_{u,N} \geq m_N+x- \sqrt{\epsilon}) \cdot \mathrm{e}^{-
        \frac12 (C^*\epsilon)^{-1}}\,.$$
The inequality \eqref{eq-prob-perturb}, for $\varepsilon \le 1$, follows from this display and \eqref{eq-for-ofer}.
}

We next estimate the expectation of the maximum. Let 
$\bar M_N = \eta^*_N \vee (m_N - c\epsilon^{-1/2}-\sqrt{\epsilon})$.
By the estimate on the left tail in \cite[Theorem 1.1]{Ding11}, 
{$\E (\bar M_N - \eta^*_N) \leq C\sqrt{\epsilon}$ 
}, 
where $C$ is an absolute constant. (Theorem 1.1 in \cite{Ding11} actually gives a much sharper estimate.) Set 
{$\tilde M_N = \max_{u\in V_N} \left(\eta_{u, N} + 
\epsilon \phi_{u, N}\right)$}. By \eqref{eq-prob-perturb},
{
        \begin{eqnarray*}
                \E \tilde M_N & \leq &m_N-c\epsilon^{-1/2}+
                \int_0^\infty \P(\tilde M_N>m_N-c\epsilon^{-1/2}+x) dx\\
                &\leq& m_N-c\epsilon^{-1/2} +
                (1+e^{-C^{-1}\epsilon^{-1}})\int_0^\infty 
                \P(\eta_N^*>m_N-c\epsilon^{-1/2}-\sqrt{\epsilon}+x) dx\\
                &\leq& m_N-c\epsilon^{-1/2} +
                (1+e^{-C^{-1}\epsilon^{-1}})
                \int_{-c\epsilon^{-1/2}-\sqrt{\epsilon}}^\infty 
                \P(\eta_N^*>m_N+x) dx\\
                &\leq& \sqrt{\epsilon}+ 
                        (m_N-c\epsilon^{-1/2}-\sqrt{\epsilon})
                +\int_{-c\epsilon^{-1/2}-\sqrt{\epsilon}}^\infty 
                \P(\eta_N^*>m_N+x) dx\\
                &&+
                e^{-C^{-1}\epsilon^{-1}}
                \int_{-c\epsilon^{-1/2}-\sqrt{\epsilon}}^\infty
                P(\eta_N^*>m_N+x) dx \\
                &\leq& \sqrt{\epsilon}+\E\eta_N^*+\E(\bar M_N-\eta_N^*)+
                \E(\eta_N^*-m_N-c\epsilon^{-1/2}-\sqrt{\epsilon})_+
                e^{-C^{-1}\epsilon^{-1}}
                        \,.
                \end{eqnarray*}
        This completes the proof of the lemma for $\epsilon \leq 1$.}

The case $\epsilon>1$ is simpler and follows by repeating the same argument
with $\Gamma_{2^i\epsilon^2}$ replacing $\Gamma_{2^i\sqrt{\epsilon}}$. We omit
further details.
\end{proof}

\subsection{A covariance computation}
\label{sec-covariance}
We will also need the following
estimate on the coarse field $X^c_{\cdot,N,K}$.
\begin{lemma}
  \label{lem-coarsecov}
  There exists a constant $c_\delta$, not depending on $K,N$, such that,
  for all $i$ and all $v,v'\in V_N^{K,\delta,i}$,
  \begin{equation}
    \label{eq-100113a}
    \E
    {(}(X_v^c-X_{v'}^c)^2
    {)}
    \leq c_\delta \left(\frac{|v-v'|}{N/K}\right)^2\,.
  \end{equation}
 Moreover, there exist constants $C_\delta,C_\delta',
  K_0(\delta)>0$, so that,
  {for  all $v,v'\in V_N^{K,\delta}$}
  and
  all $K>K_0(\delta)$,
  \begin{equation}
    \label{eq-100113anew}
    C_\delta 
    { \left(\left(\frac{|v-v'|}{N/K}\right)^2 \wedge 1\right)}\leq 
    \E(X_v^c-X_{v'}^c)^2\,
    \end{equation}
    and,  
    {for all $i$ and all $v,v'\in V_N^{K,\delta,i}$,}
    \begin{equation}
            \label{eq-100113anewadd}
    |\E 
    {(}(X_v^c)^2
    {)}-\E
    {(}(X_{v'}^c)^2
    {)}|
    \leq 
    C_\delta' \frac{|v-v'|}{N/K} \,.
    \end{equation}
\end{lemma}
\begin{proof}
  We assume
  that $i$ is fixed  and
  suppress it from the notation, since the estimates will not depend
  on $i$. For convenience, set $V_N^{K,i}=V_{N/K}$.

Recalling \eqref{eq-lawler1}, it follows that 
\begin{equation}
        \label{eq-morning1}
        \E(\eta_{v,N}-\eta_{v',N})^2=
   2a(v-v')+
  \sum_{z\in \partial V_N}
  \left(\P^v(S_{\tau_N}=z)-P^{v'}(S_{\tau_N}=z)\right)
  \left(a(z-v)-a(z-v')\right),
  \end{equation}
  where $\tau_N$ denotes the exit time of the simple random walk from $V_N$.
  By \eqref{eq-100113c} and $v\neq v'$, one has 
$$\max_{z\in \partial V_N}
 |a(z-v)-a(z-v')|\leq 2 \frac{|v-v'|}{
 {\delta N}}+(C K^2/(\delta^2 N^2))\,,$$
 where 
 we have used the lower bound 
 {$\delta N$}
 on the distances from $v$ and $v'$ 
to $\partial    V_N$. Similarly, using  
that 
  $\P^v(S_{\tau_N}=z)$ is harmonic in $v$ and applying the
Harnack estimates (see, e.g., 
\cite[Theorem 6.3.8]{LL10}), one gets that, for any $v,v'$ and 
$z\in \partial V_N$,
$$|\P^v(S_{\tau_N}=z)-
\P^{v'}(S_{\tau_N}=z)|\leq C_\delta  N^{-1}
 \frac{|v-v'|}{
 {\delta N}}.$$
 Together with \eqref{eq-morning1}, the last two displays imply that
 \begin{equation}
   \label{eq-january2a}
   |\E(\eta_{v,N}-\eta_{v',N})^2-
   2a(v-v')|\leq  
    C_\delta \left(\frac{|v-v'|}{
    {N}}\right)^2 \,.
  \end{equation}
        Applying this formula to $V_{N/K}$, and using  
        \begin{equation}
          \label{eq-january2b}
          \E(X_v^c-X_{v'}^c)^2=\E
          {(\eta_{v,N}-\eta_{v',N})^2}-
        \E(X_v^f-X_{v'}^f)^2
      \end{equation}
        and the observation that $X_v^f$ is a GFF in the box $V_{N/K}$, yields 
    \eqref{eq-100113a}.

The argument for  \eqref{eq-100113anew} is 
more delicate, due to possible cancellations. While one could argue 
along the lines of the above argument using a comparison 
with Brownian motion, we prefer to use a more direct argument, 
employing the notion of
resistance. Recall (see, e.g., \cite{LP} for background on electrical
networks and their relation with the Green function of simple random walk)
that 
   $\E(\eta_{v,N}-\eta_{v',N})^2$ is twice the effective resistance 
   between $v$ and $v'$ in the square grid resistor network of side $N$ 
   with wired boundary, denoted by $R_{\rm eff}^{v,v',N}$, and
   that $R_{\rm eff}^{v,v',\infty}=a(v-v')$. 
   
   We first consider the case where $v,v'\in V_N^{K,\delta,i}$ for some $i$.
   We claim that, in this case,
   \begin{equation}
     \label{eq-january2c}
     R_{\rm eff}^{v,v',N}\leq R_{\rm eff}^{v,v',\infty}-
    c_\delta \left(\frac{|v-v'|}{N}\right)^2 \,.
  \end{equation}
  Applying \eqref{eq-january2c}, the representation 
  \eqref{eq-january2b} and the estimate 
   \eqref{eq-january2a} yield 
 \eqref{eq-100113anew}.
     
 To prove \eqref{eq-january2c}, recall that 
 $$R_{\rm eff}^{v,v',N}=\min_{\{i(e)\}\in I_{v,v',N}}
 \{ \sum_e i^2(e)\}\,,$$
 where the sum is over
 all edges and $I_{v,v',N}$ denotes the unit flows from $v$ to $v'$ with
 wired boundary at $\partial V_N$. 
 Similarly, 
 {denoting by $I_{v, v'}$ the unit flows from $v$ to $v'$ in the infinite lattice $\mathbb Z^2$,}
 $$R_{\rm eff}^{v,v',\infty}=\min_{\{i(e)\}\in I_{v,v'}}
   \{ \sum_e (i(e))^2\}\,.$$
 Since 
 {each flow in $ I_{v,v'}$ restricted to $V_N$ gives a valid flow in $I_{v, v', N}$}, letting $i^\infty$ denote the flow 
 achieving $R_{\rm eff}^{v,v',\infty}$, it follows that 
 \begin{equation}
	 \label{eq-44new}
	 R_{\rm eff}^{v,v',N}\leq 
 R_{\rm eff}^{v,v',\infty} -\sum_{e\not\in V_N} (i^\infty(e))^2\,.
 \end{equation}
 In order to prove \eqref{eq-january2c}, it thus remains to estimate
 $i^\infty(e)$.

 Toward this end, note (see, e.g., \cite[Proposition 4.7]{LP}) that,
 if $e$ is the edge $(w,w')$, then
 $i^\infty(e)=2[a(w-v')-a(w-v)-a(w'-v')+a(w'-v)]$. We recall a strengthened
 version of \eqref{eq-100113c}, namely, 
 $$
 a(x)=\frac{2}{\pi}\log |x|+
   \frac{2\bar\gamma+\log 8}{\pi}+
   \frac{1}{6\pi} \cdot \frac{\Re(x^4)}{|x|^6}+
   O(|x|^{-4})\,,
   $$
  which was proved in \cite{FU96}. 
   ({Here, $\Re(x^4)$ is the real part of $x^4$}.)
   Setting  $\bar e=w'-w$, $\bar s=(v-v')/|v-v'|$
   and $\bar u=(w-v)/|w-v|$,
   we then obtain that 
   \begin{equation}\label{eq-flow-estimate}
   i^\infty(e)=\frac{|v-v'|}{|w-v|^2}\langle \bar e,\bar 
   s\rangle[1-\langle \bar u,\bar s\rangle]+O\left(
   \frac{|v-v'|^2}{|w-v|^3}\right)\,,\end{equation}
   {where $\langle \bar e, 
   \bar s\rangle$ denotes the inner product of $\bar e$ and $\bar s$,
   which are viewed as vectors in $\mathbb R^2$}.
   In particular,
\begin{equation}\label{eq-flow}
\sum_{e=(w,w'): |w|\geq N}
   (i^\infty(e))^2\geq C |v-v'|^2\sum_{r=N}^\infty \frac{1}{r^3} +O\left(
   \left(\frac{|v-v'|}{N}\right)^3\right)\,.\end{equation}
     This yields \eqref{eq-january2c} and therefore 
     completes the proof of  \eqref{eq-100113anew}
     in the case when $v, v'\in V_{N}^{K, \delta, i}$ for some $i$.

     {We next consider the case when $v\in V_N^{K, \delta, i}$, $v'\in V_N^{K, \delta, i'}$ for $i\neq i'$. The strategy of the proof is similar to that given above when
$v,v'\in V_{N}^{K, \delta, i}$. Denote by $R_{\mathrm{eff}}(v, \partial V_N^{K, \delta, i})$ the effective resistance between $v$ and $\partial V_N^{K, \delta, i}$ (with similar notation for $R_{\mathrm{eff}}(v, \partial V_N^{K, \delta, i'})$).  

When $|v-v'| \geq 4 N/K$, we denote by $B(v, 2N/K)$ the box centered at $v$ with side length $2N/K$ and obtain that 
          \begin{equation}
          \label{eq-1new}
          \sum_{e= (w, w'): w, w'\not\in V_{N}^{K, \delta, i} \cup V_N^{K, \delta, i'}} (i^\infty(e))^2 \geq  \sum_{e= (w, w'): w, w'\in B(v, 2N/K) \setminus V_N^{K, \delta, i}} (i^\infty(e))^2  \geq C_\delta
\end{equation}
(with perhaps a new choice of $C_\delta$).
          Here, we used the fact that the total amount of flow through any cut set between $v$ and $v'$ is at least 1, and that, in $B(v, 2N/K) \setminus V_N^{K, \delta, i}$, there are $N/2K$ disjoint cut sets, each of which is of cardinality $O(N/K)$; since the sum of $(i^\infty(e))^2$ is at least
$O(K/N)$ by the Cauchy-Schwarz inequality, \eqref{eq-1new} follows for $v,v'$ satisfying $|v-v'|\geq 4N/K$.

When $|v- v'| \leq 4 N/K$ (recall that then $|v-v'|\geq \delta N/K$ also holds since $v,v'$ are assumed to belong to different boxes $V_N^{K,\delta,i},V_N^{K,\delta,i'}$
of side length $N/K$), we obtain that, for $\hat C\geq 8$,
     \begin{align*}
     \sum_{e= (w, w'): w, w'\not\in V_{N}^{K, \delta, i} \cup V_N^{K, \delta, i'}} (i^\infty(e))^2 \quad &\geq \sum_{e= (w, w'): |w-v| \geq \hat C N/K} (i^\infty(e))^2 \\
     & \geq C |v-v'|^2\sum_{r=\hat CN/K}^\infty \frac{1}{r^3} +O\left(
   \left(\frac{|v-v'|}{\hat C N/K}\right)^3\right) \,,
   \end{align*}
   where the last inequality follows from \eqref{eq-flow-estimate}. Choosing $\hat C$ large enough so that the term $O\left(
   \left(\frac{|v-v'|}{\hat C N/K}\right)^3\right)$ is absorbed, it follows that 
   \begin{equation}
   \label{eq-2new}
    \sum_{e= (w, w'): w, w'\not\in V_{N}^{K, \delta, i} \cup V_N^{K, \delta, i'}} (i^\infty(e))^2  \geq C_\delta\,.
\end{equation}
  }
  
  Combining \eqref{eq-1new} and \eqref{eq-2new}, and using the same flow and argument that led to \eqref{eq-44new},
 we obtain that, 
  for any $v,v'$ in disjoint boxes $V_N^{K,\delta,i},V_N^{K,\delta,i'}$,
     $$R_{\mathrm{eff}}^{v, v', \infty} - R_{\mathrm{eff}}(v, \partial V_N^{K, \delta, i}) -R_{\mathrm{eff}}(v', \partial V_N^{K, \delta, i'}) \geq C_\delta\,.$$
     Together with the representation 
  \eqref{eq-january2b} and the estimate 
   \eqref{eq-january2a}, this completes the proof of
 \eqref{eq-100113anew} in case $v$ and $v'$ belong to different boxes $V_N^{K,\delta,i},V_N^{K,\delta,i'}$.

The proof of \eqref{eq-100113anewadd} is similar to the proof of \eqref{eq-100113a}  .
    First note that
    $$ \E\eta_{v,N}^2=\sum_{z\in \partial V_N} \P^v(S_{\tau_N}=z)a(z-v)\,.$$
    So, fixing a point $z_0\in \partial V_N$,
    \begin{eqnarray*}
            &&|\E\eta_{v,N}^2-
    \E\eta_{v',N}^2|\\
    &\leq&
    \sum_{z\in \partial V_N} [\P^v(S_{\tau_N}=z)-\P^{v'}(S_{\tau_N}=z)]a(z-v)+
    \sum_{z\in \partial V_N} \P^{v'}(S_{\tau_N}=z)|a(z-v)-a(z-v')|\\
    &\leq&
    \sum_{z\in \partial V_N} [\P^v(S_{\tau_N}=z)-\P^{v'}(S_{\tau_N}=z)][a(z-v)-a(z_0,v)]+
    \sum_{z\in \partial V_N} \P^{v'}(S_{\tau_N}=z)|a(z-v)-a(z-v')|\\
    &\leq&
    C_\delta'
    \sum_{z\in \partial V_N} |\P^v(S_{\tau_N}=z)-\P^{v'}(S_{\tau_N}=z)|
    +\sum_{z\in \partial V_N} \P^{v'}(S_{\tau_N}=z)|a(z-v)-a(z-v')| \,.\\
    \end{eqnarray*}
    The conclusion follows by straight forward manipulation of the above quantities.
  \end{proof}

\section{The limiting tail of the GFF maximum}
\label{sec-limittail}

Recall that $\{\eta_{v,N}: v\in V_N\}$ denotes the GFF on the
two-dimensional box $V_N$ and that $\eta_N^*=
\max_{v\in V_N}\eta_{v,N}$.
For an open set $A \subseteq (0, 1)^2$, let $N A = \{v\in V_N :  v/N \in A\}$.
The main result of this section is the following proposition.
\begin{prop}\label{prop-limiting-tail-gff}
There exists an absolute constant  $\alpha^*>0$ such that
\begin{equation}
\label{eqnewMB1}
\lim_{z\to \infty}\limsup_{N\to \infty}|z^{-1} \mathrm{e}^{\sqrt{2\pi }z}\P(\eta^*_N \geq m_N +z) - \alpha^*| =0\,.
\end{equation}
Furthermore, there exists a continuous function $\psi: (0, 1)^2 \mapsto (0, \infty)$ with $\int_{[0, 1]^2} \psi(x) dx =1$ such that, for any open set $A\subseteq (0, 1)^2$,
\begin{equation}
\label{eqnewMB2}
\lim_{z\to \infty}\limsup_{N\to \infty}|z^{-1} \mathrm{e}^{\sqrt{2\pi }z}\P(\max_{v\in NA}\eta_{v,N} \geq m_N +z) - \alpha^* \int_{A} \psi(x) dx| =0\,.
\end{equation}
\end{prop}
The following corollary follows quickly from Proposition \ref{prop-limiting-tail-gff}.
\begin{cor}\label{cor-no-two-large-clusters}
For any open box $A\subseteq (0, 1)^2$,
$$\lim_{z\to \infty} \limsup_{N\to \infty} z^{-1}\mathrm{e}^{\sqrt{2\pi }z}\P(\max_{v\in N A}\eta_{v,N} \geq m_N+z,  \max_{v\in V_N \setminus N A} \eta_{v,N} \geq m_N+ z) = 0\,.$$
\end{cor}
\begin{proof}
For $0<\delta<1/10$, let $V_\delta = \{v\in [0, 1]^2:
 \mathrm{dist}(v, \overline{A}) \leq \delta\}$,
 where $\overline{A}$
denotes the
closure of $A$.
Consider the open sets $A$ and $V_\delta^c$.
By Proposition~\ref{prop-limiting-tail-gff} and
the inclusion-exclusion principle,
$$\lim_{z\to \infty} \limsup_{N\to \infty} z^{-1}\mathrm{e}^{\sqrt{2\pi }z}\P(\max_{v\in N A} \eta_{v,N} \geq m_N + z, \max_{v\in N V_\delta^c} \eta_{v,N} \geq m_N+z) = 0\,.$$
Applying \eqref{eq-max-subset} of
Lemma~\ref{lem-prelim-tail} to the set $V_N\setminus (NA\cup NV_\delta^c)$
and letting $\delta\searrow 0$ completes the proof of the corollary.
\end{proof}

Before proceeding to the proof of
Proposition \ref{prop-limiting-tail-gff}, we
show that Propositions \ref{prop-jian} and \ref{prop-jian-delta} follow directly from it and
Corollary \ref{cor-no-two-large-clusters}. Recall that
$g(\cdot)$ is
a function with $g(K) \rightarrow\infty$ as $K\rightarrow\infty$.

\medskip

\noindent
{\it Proof of Propositions \ref{prop-jian} and
\ref{prop-jian-delta} (assuming Proposition
\ref{prop-limiting-tail-gff} and Corollary \ref{cor-no-two-large-clusters}).}
We prove Proposition
\ref{prop-jian}.
Display \eqref{eq-301212a} is simply a reformulation of (\ref{eqnewMB1}).
%
%
In order to prove \eqref{eq-301212b}, it suffices to consider the case when $A$ is an open box.
Using (\ref{eqnewMB2}) and Corollary~\ref{cor-no-two-large-clusters}, we obtain
$$\lim_{K\to\infty} \limsup_{N\to \infty}\left|
\frac{\mathrm{e}^{\sqrt{2\pi} (x_K +g(K))}}{g(K)+x_K} \P( \max_{v\in V_{N/K}} \eta_{v,N/K}\geq m_{N/K}+g(K)+x_K, \frac{K}{N}v^*\in A) - \alpha^*\int_A \psi(y) dy\right|=0\,.$$
Combining this with \eqref{eq-301212a}, the desired equality \eqref{eq-301212b} follows by
Bayes' formula.

The proof of Proposition \ref{prop-jian-delta} is analogous, but instead using 
%
%
(\ref{eqnewMB2}) (rather than (\ref{eqnewMB1})), with $A=(\delta,1-\delta)^2$. We omit further details.
\qed

\begin{figure}[htb]
\begin{center}
\includegraphics[width=120mm]{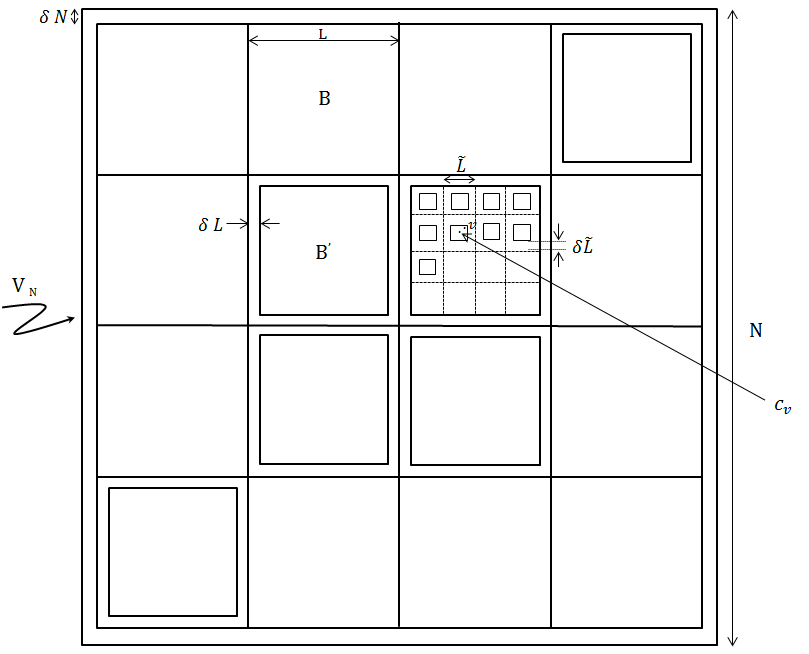}
\caption{The boxes $B\in\mathcal{B}_N$, $B'$, $\tilde B_{v,N}$}
\label{figure-2}
\end{center}\end{figure}


In order to prove Proposition~\ref{prop-limiting-tail-gff},
we will study a sparse version of the lattice $V_N$.
Consider $0<\delta<1/100$ chosen independently of  the other constants. Let $V'_N \subseteq V_N$ be a box in the center of $V_N$ with side length $N ' = (1 - 2\delta) N$. Let $L$, $\tilde L$ and $h$ be integer-valued functions of $z$,
with $h=L/\tilde L$, that satisfy
\begin{equation}\label{eq-L-tilde-L}
\tilde L \geq 2^{z^4}, \quad h\leq \log z, \mbox{ and } h \to_{z\to \infty} \infty\,.
\end{equation}
Besides \eqref{eq-L-tilde-L}, 
the only other
assumption we impose on $L$, $\tilde L$ and $h$ 
is that they do not depend on $N$.
In particular, in this section, when taking multiple limits, we will
let $N\rightarrow\infty$
before taking other limits.
(Note that the number of boxes with side length $L$ will go to infinity before
$L$ increases.
This order differs from that in, e.g.,
Proposition \ref{prop-jian}, where the lengths of boxes of side length
$N/K$ go to infinity before the number of such boxes is allowed to increase.)
Throughout the rest of this section, we write
$$n = \log_2 N,  \,\ell = \log_2 L, \mbox{ and }\tilde \ell = \log_2 \tilde L\,.$$  By \eqref{eq-L-tilde-L}, $\tilde \ell \geq z^4$ and $\ell \leq \tilde \ell+\log_2\log z$.

Let $\mathcal{B}_N$ be the collection of boxes of side length $L$ obtained by partitioning $V'_N$ into $((1-2\delta)N/L)^2$ sub-boxes. For every $B\in \mathcal{B}_N$, let $B' \subseteq B$ be the box in the center of $B$ with side length $(1 - 2\delta)L$, and let $\tilde{\mathcal{B}}_B$ be the collection of $((1-2\delta)L/\tilde L)^2$  boxes of side length 
$(1-2\delta)\tilde L$ placed inside $B'$ such that every two boxes are at least $2\delta \tilde L$ distance apart.
(This collection is obtained by removing from $B$ a grid-patterned set
of width
$2\delta \tilde L$.)
Set $\tilde{\mathcal B}_N = \cup_{B\in \mathcal{B}_N} \tilde {\mathcal B}_B$ and, for each $\tilde B \in \tilde {\mathcal{B}}_N$, denote by $c_{\tilde B}$ the center of $\tilde B$. Furthermore, for each $v\in V_N$, we denote by $B_{v, N}$ and $\tilde B_{v, N}$ the boxes in $\mathcal{B}_N$ and $\mathcal{\tilde B}_N$ that contain $v$ (if they exist), respectively. Write $\tilde V_N = \cup_{\tilde B\in \mathcal{\tilde B}_N} \tilde B$. Finally, for $v\in \tilde V_N$, denote by $c_v = c_{\tilde B_{v, N}}$ the center of the $\tilde B$-box that contains $v$.
(By center, we mean the vertex in $V_N$ that is closest to the geometric center
of $\tilde B_{v,N}$ among those
vertices both of whose coordinates are smaller than those of the geometric
center.)
The following proposition is the crucial step in proving
Proposition~\ref{prop-limiting-tail-gff}.
\begin{prop}\label{prop-limiting-tail-gff-delta}
For any $0<\delta\leq 1/100$, there exists a constant  $\alpha_{\delta}^*>0$ such that
$$\lim_{z\to \infty}\limsup_{N\to \infty}|z^{-1} \mathrm{e}^{\sqrt{2\pi }z}\P(\max_{v\in \tilde V_N} \eta_{v,N} \geq m_N +z) - \alpha_{\delta}^*| =0\,.$$
Furthermore, there exists a continuous function $\psi_\delta: [\delta, 1-\delta]^2 \mapsto (0, \infty)$, with
 $\int_{[\delta, 1-\delta]^2} \psi_\delta(x) dx = 1$, and a continuous function $\psi: (0, 1)^2 \mapsto (0, \infty)$, with $\psi_\delta(x) \to \psi(x)$
uniformly in $x$ on closed sets as $\delta \searrow 0$,
such that, for any open set
$A\subseteq [\delta, 1-\delta]^2$,
$$\lim_{z\to \infty}\limsup_{N\to \infty}|z^{-1} \mathrm{e}^{\sqrt{2\pi }z}\P(\max_{v\in NA \cap \tilde V_N}\eta_{v,N} \geq m_N +z) - \alpha^*_\delta \int_{A} \psi_\delta(x) dx| =0\,.$$
\end{prop}
Proposition~\ref{prop-limiting-tail-gff} follows immediately from 
Proposition \ref{prop-limiting-tail-gff-delta}
and Lemma~\ref{lem-prelim-tail} by letting 
$\delta\searrow 0$, since,  by
Lemma~\ref{lem-prelim-tail},
$\sup_{\delta\leq 1/100} \alpha_\delta^*<\infty$,
while, from the definition of $\tilde V_N$, $\alpha_\delta^*$ 
is monotone decreasing in $\delta$
and  thus $\alpha_\delta^*\to_{\delta\to 0}
\alpha^*\in (0,\infty)$. 

The rest of the section is devoted to the proof of Proposition \ref{prop-limiting-tail-gff-delta}. In what follows, we consider $\delta > 0$ to be fixed and suppress any dependence on $\delta$ in the notation except in cases where it is important to stress the dependence.

\subsection{A mixture of MBRW and GFF}
In the proof of Proposition
\ref{prop-limiting-tail-gff-delta}, we will approximate the GFF by a mixture
of a MBRW (in coarse scales) and a copy of the GFF (in fine scales).
The approximation consists of
two main steps. First, in analogy with the coarse-fine decomposition introduced
in Section
\ref{sec-coarsefine}, but employing very different scales,
we
write  the GFF as a sum of two independent Gaussian fields. The ``fine'' field
will consist of independent copies
of the GFF in the boxes $B_{v,N}$, while the
``coarse'' field will be approximated by a Gaussian field that is piecewise constant 
over each of the smaller boxes $\tilde B_{v,N}$.
In the second step, we then further approximate the coarse field by a
MBRW.

\medskip
\noindent{\bf Step 1.} For $v\in \tilde V_N$, define (in analogy with
\eqref{eq-of1}, except that box sizes are different)
\begin{align}
X_{v,N} &=  \E (\eta_{v,N} \mid \mathcal{F}_{\partial B_{v, N}}) \mbox{ and } Y_{v,N} = \eta_{v,N}  - X_{v,N}\,,\nonumber\\
\tilde \eta_{v,N} & = X_{c_{v},N} + Y_{v,N}\,. \label{eq-def-tilde-eta}
\end{align}
Note that, 
{for each $B\in \mathcal B_N$,} the process $\mathcal{Y}_B = \{Y_{v,N}: v\in B\}$ is distributed as a GFF on $B$ with Dirichlet boundary data. Moreover,
as in the decomposition into the coarse and fine fields in Section \ref{sec-coarsefine},
\begin{align}
&\{X_{v,N}: v\in \tilde V_N\} \mbox{ is independent of } 
\{Y_{v,N} : v\in \tilde V_N\} \nonumber\\
        \label{eq-independence}
&
\!\!\!\!\!\!\!\!
\!\!\!\!\!\!\!\!
\!\!\!\!\!\!\!\!
\!\!\!\!\!\!\!\!
\!\!\!\!\!\!\!\!
\!\!\!\!\!\!\!\!
\!\!\!\!\!\!\!\!
\mbox{ and }\\
& \{\mathcal Y_B: B\in \mathcal{B}_N\} \mbox{ are independent }.
\nonumber
\end{align}

We first show that the limiting right tail for the maximum of $\{\eta_{\cdot,N}\}$ can be approximated by that of $\{\tilde \eta_{\cdot,N}\}$.
We start with the following preparatory  lemma.
This is the only place where the assumption
$h\to\infty$ is used.
\begin{lemma}\label{lem-green-functions-1}
        For every choice of sequence $h=h(z)$ satisfying
        \eqref{eq-L-tilde-L},
there exists $N_0=N_0(z,\ell,\tilde \ell)$ and
$\epsilon_z$ with $\epsilon_z \searrow_{z\to \infty} 0$ such that, for all $u, v\in \tilde V_N$ and all $N>N_0$,
$$(1 - \epsilon_z/\log N)^2 \E X_{u,N} X_{v,N} \leq \E X_{c_u,N} X_{c_v,N}  \leq (1+\epsilon_z/\log N)^2 \E X_{u,N} X_{v,N}\,.$$
\end{lemma}
\begin{proof}
We first consider the case when $u$ and $v$ belong to different boxes in $\mathcal{B}_N$.
Setting $d_{u,v} = \|u-v\|_2$, we have $d_{u, v} \gtrsim L$.
In addition, by the independence in \eqref{eq-independence},
$$\E X_{u,N} X_{v,N} = \E \eta_{u,N} \eta_{v,N} \mbox{ and } \E X_{c_u,N} X_{c_v,N} = \E \eta_{c_u,N} \eta_{c_v,N}\,.$$
Let $H_{v, N}$ and $H_{ c_v, N}$ be the
exit measures on
$\partial V_N$ for random walks started at $v$ and $c_v$, respectively. By \cite[Proposition 8.1.4]{LL10}, $\|H_{ v, N} -  H _{c_v, N}\|_{\mathrm{TV}} \lesssim \tilde L/N$, where $\|\mu - \nu\|_{\mathrm{TV}}$ denotes   the total variation distance between  measures $\mu$ and $\nu$. Combined with 
the Green function estimates of simple random walk \cite[Proposition 4.6.2]{LL10},
this implies
\begin{equation}\label{eq-green-function-case-1}
|\E \eta_{u,N} \eta_{v,N} - \E \eta_{c_u,N} \eta_{c_v,N} |\lesssim \frac{\tilde L}{N} + \log (1+ \frac{\tilde L}{d_{u, v}}) \lesssim  \frac{\tilde L}{d_{u, v}} \lesssim \frac{\tilde L}{L} \frac{\log (\frac{N}{d_{u, v}})\vee 2)}{\log N} \asymp\frac{\tilde L}{L} \frac{\E X_{u,N} X_{v,N}}{\log N}\,.
\end{equation}
Since $\E X_{u,N}X_{v,N}=\E\eta_{u,N}\eta_{v,N}$, the last display demonstrates the lemma, when $u,v$ belong to different
boxes in $\mathcal{B}_N$, by choosing $\epsilon_z=C\tilde L/L$,
for some fixed, absolute constant $C$.

We next consider $u, v\in B$ for a given $B\in \mathcal{B}_N$. Let $H'_v$ and $H'_{c_v}$ be the
exit measures on
$\partial B$ for random walks started at $v$ and $c_v$, respectively. By \cite[Proposition 8.1.4]{LL10}, $\|H'_v -  H' _{c_v}\|_{\mathrm{TV}} \lesssim \tilde L/L$. Combined with \cite[Lemma 4.6.2]{LL10}, this implies that
\begin{align*}
|\E X_{u,N} X_{v,N} -  \E X_{c_u,N} X_{c_v,N} |& \leq |\E \eta_{u,N} \eta_{v,N} - \E \eta_{c_u,N} \eta_{c_v,N} | + |\E Y_{u,N} Y_{v,N} - \E Y_{c_u,N} Y_{c_v,N} |\\
&\lesssim \tilde L/N + \tilde L/L \lesssim \tilde L/L\,.
\end{align*}
Furthermore,
$$\E X_{u,N}X_{v,N} = \E \eta_{u,N} \eta_{v,N} - \E Y_{u,N} Y_{v,N} = 
\frac2\pi \log N - O(\log L)\,,$$
where the last equality uses 
Lemma \ref{lem-covariance} and the observation 
that $\{Y_{v,N}\}$ is a GFF in the box $B$.
Together, the last two displays imply that
$$|\E X_{u,N} X_{v,N} -  \E X_{c_u,N} X_{c_v,N} | \lesssim (\tilde L/L) \cdot (\E X_{u,N} X_{v,N} /\log N)\,. $$
Setting $\epsilon_z=C\tilde L/L$, with $C$ a fixed absolute constant,
demonstrates the lemma when $u,v$ belong
to the same box in $\mathcal{B}_N$.
\end{proof}

We next compare the maxima of $\eta_{\cdot, N}$ and $\tilde \eta_{\cdot, N}$.
\begin{lemma}\label{lem-eta-tilde-eta}
There exist $\delta_z$, with $\delta_z \searrow_{z\to \infty} 0$,
such that
\begin{eqnarray}
\label{eq-ofer230113b}
&&\liminf_{z\to \infty} \liminf_{N\to \infty} \frac{\P(\max_{v\in \tilde V_N} \eta_{v,N} \geq m_N+z)}{ \P(\max_{v\in \tilde V_N} \tilde \eta_{v,N} \geq m_N+z +\delta_z)}  \geq 1,\,\\
\label{eq-ofer230113c}
&&\limsup_{z\to \infty} \limsup_{N\to \infty} \frac{\P(\max_{v\in \tilde V_N} \eta_{v,N} \geq m_N+z)}{ \P(\max_{v\in \tilde V_N} \tilde \eta_{v,N} \geq m_N+z  - \delta_z)} \leq 1  \,.
\end{eqnarray}
\end{lemma}
\begin{proof}
Choose $\epsilon_z$ as in Lemma~\ref{lem-green-functions-1}. For $v\in \tilde V_N$, define
$$\zeta_{v,N} = (1 - \epsilon_z/\log N) X_{v,N} + Y_{v,N} + \sqrt{\epsilon_z} \bar \phi_{ v, N},$$
where $\bar \phi_{v, N}$ are independent Gaussian variables with variances such that $\var \zeta_{v, N} = \var \tilde \eta_{v,N}$. Because of Lemma~\ref{lem-green-functions-1}, $\var \bar\phi_{ v, N}$ is bounded
uniformly in $v$ and $N$.  Using Lemma \ref{lem-green-functions-1} again implies
\begin{equation}\label{eq-cor-bar-tilde}
\E \zeta_{u,N} \zeta_{v,N} \leq \E \tilde \eta_{u, N} \tilde \eta_{v, N} \mbox{ for all }u, v\in \tilde V_N\,.
\end{equation}
Combined with Lemma~\ref{lem-slepian}, this implies
 \begin{equation}\label{eq-eta-bar-tilde}
 \P(\max_{u\in \tilde V_N} \zeta_{u,N} \geq \lambda) \geq \P(\max_{u \in \tilde V_N} \tilde \eta_{u,N} \geq \lambda) \mbox{ for all } \lambda\in \mathbb{R}\,.
 \end{equation}
Since $\zeta_{u,N}=\eta_{u,N}-\epsilon_z X_{u,N}/\log N+\sqrt{\epsilon_z} \bar \phi_{u,N}$,
 it follows from this that, for all $z$ large enough so that $10\epsilon_z<\epsilon_z^{1/4}$,
\begin{align*}
\P(\max_{u\in \tilde V_N} \zeta_{u,N} \geq m_N+z) &\leq  \P(\max_{u\in \tilde V_N} (\eta_{u,N} + \sqrt{\epsilon_z} \bar \phi_{ u, N})
 \geq m_N+z - \epsilon_z^{1/4}) + \P( \max_{u\in \tilde V_N}  X_{u,N} \geq 10 \log N) \,.
\end{align*}
(The constant $10$ is chosen with the future
bound \eqref{eq-of230113} in mind.)

An identical proof to that of Lemma~\ref{lem-gff-perturb} (applied to the maximum in $\tilde V_N$, as opposed to $V_N$) shows that, for
$\epsilon>0$ and $x>0$,
$$\P(\max_{u\in \tilde V_N}  (\eta_{u,N} + \epsilon \bar
\phi_{ u, N}) \geq m_N+x)  \leq   \P(\max_{u\in \tilde V_N} \eta_{u,N} \geq m_N+x - \sqrt{\epsilon}) (1+ O(\mathrm{e}^{-C^{-1} \epsilon^{-1}}))\,.$$
Substituting $\epsilon=\epsilon_z^{1/2}$ and $x=z-\epsilon_z^{1/4}>0$, for
large $z$, and using the union bound
\begin{equation}
\label{eq-of230113}
\P(\max_{u\in\tilde V_N}X_{u,N}\geq 10\log N)\leq
N^2 \max_{u\in V_N}\P(X_{u,N}\geq 10\log N)=O(N^{-4})
\end{equation}
 implies that
\begin{equation}\label{eq-bound-eta-bar}
\P(\max_{u\in \tilde V_N} \zeta_{u,N} \geq m_N+z)  
\leq  O(N^{-4}) +  \P(\max_{u\in \tilde V_N} \eta_{u,N} \geq m_N+z - 
2 \epsilon_z^{1/4}) (1+ O(\mathrm{e}^{-C^{-1} 
\epsilon_z^{-1/2}}))\,.\end{equation}
(The estimate on the right hand side of \eqref{eq-of230113} follows from
$\max_{u\in V_N} \E X_{u,N}^2\le (2/\pi)\log N +O(1)$ and the Gaussian tail bound.)
Together, \eqref{eq-eta-bar-tilde} and (\ref{eq-of230113})  imply \eqref{eq-ofer230113b}, with $\delta_z=2\epsilon_z^{1/4}$.
(We have used the result that, for fixed $z$, the numerator
in \eqref{eq-ofer230113b} is bounded below by a positive function of $z$, as $N\to\infty$, 
{so that the error term $O(N^{-4})$ can be easily absorbed}; this can be shown by, e.g.,
an easy adaptation of the argument in \cite[Theorem 1.1]{Ding11} that 
dealt with $V_N$ rather than $\tilde V_N$.)

We next turn to the proof of
\eqref{eq-ofer230113c}, which is similar in spirit. 
For every $v\in \tilde V_N$, define
$$\hat \eta_{v,N} = (1 - \epsilon_z/\log N) X_{c_{v},N} + Y_{v,N} + \sqrt{\epsilon_z} \hat \phi_{v,N},$$
where $\hat \phi_{v, N}$ are independent Gaussian variables with variances chosen so that $\var \hat \eta_{v, N} = \var \eta_{v,N}$. We see that $\var \hat \phi_{v, N}$ is bounded
uniformly in $N$ and $v$, by Lemma~\ref{lem-green-functions-1}.  By this lemma,
$$\E \hat \eta_{u,N} \hat \eta_{v,N} \leq \E \eta_{u,N} \eta_{v,N} \mbox{ for all }u, v\in \tilde V_N\,.$$
Lemma~\ref{lem-slepian} implies that
\begin{equation}\label{eq-eta-hat-slepian}
\P(\max_{u\in \tilde V_N} \hat \eta_{u,N}  \geq \lambda) \geq \P(\max_{u\in \tilde V_N} \eta_{u,N }\geq \lambda) \mbox{ for all } \lambda\in \mathbb{R}\,.
\end{equation}
With overwhelming probability in $N$,
$X_{c_v,N}>0$ if $\hat \eta_{v,N} =\max_w  \hat\eta_{w,N}$;
on this event,
%
\begin{equation}\label{eq-hat-tilde-eta}
\max_v \hat  \eta_{v,N} \leq \max_v(X_{c_v} + Y_{v,N} + \sqrt{\epsilon_z} \hat \phi_{v, N})  = \max_v (\tilde \eta_{v,N}+ \sqrt{\epsilon_z} \hat \phi_{v, N})\,.\end{equation}
By \eqref{eq-cor-bar-tilde}, for any  $\Gamma \subseteq \tilde V_N$,
\begin{align*}\P(\max_{v\in \Gamma} \tilde \eta_{v,N}  \geq \lambda) \leq \P(\max_{v\in \Gamma} \bar \eta_{v,N}   \geq \lambda) = \P(\max_{v\in \Gamma} (1 - \epsilon_z/\log N)\eta_{v,N} + \sqrt{\epsilon_z} \hat \phi_{v,N} \geq \lambda )\,.
\end{align*}
Repeating the argument in the proof of Lemma~\ref{lem-gff-perturb}, we obtain
$$\P(\max_{u\in \tilde V_N}  (\tilde \eta_{u,N} + \sqrt{\epsilon_z} \hat \phi_{N, z}) \geq m_N+z)  \leq   \P(\max_{u\in \tilde V_N} \tilde \eta_{u,N} \geq m_N+z - 2\epsilon_z^{1/4}) + z \mathrm{e}^{-\sqrt{2\pi }z} O(\mathrm{e}^{C^{-1} \epsilon_z^{-1/2}})\,.$$
Together with \eqref{eq-eta-hat-slepian} and \eqref{eq-hat-tilde-eta}, this completes the proof of the lemma.
\end{proof}

\medskip

\noindent{\bf Step 2.} Define
\begin{equation}
  \label{eq-oops240113a}
  \Xi_N = \{c_{\tilde B}: \tilde B\in \mathcal{\tilde B}_N\}.
\end{equation}
We next approximate $\{X_{v,N}: v\in \Xi_N\}$ by a MBRW, by using the notation of
Section~\ref{sec:MBRW}. 
Let $\phi_{N,j,B}$ and $\xi_{v,N}$ be as in Subsection \ref{sec:MBRW} and,
for $v\in \Xi_N$, define
$$S_{v,N} = \mbox{$\sum_{j=\ell}^n \sum_{B\in \mathfrak{B}_j(v)}$} \phi_{N, j, B}\,.$$
Note that, for $B\in \mathcal{B}$, the process $\{\xi_{v,N} - S_{v,N}: v\in B \cap \Xi_N\}$ is a MBRW (projected onto $\Xi_N$) that is defined with respect to the box $B$, except that the
torus wraps around with respect to $V_N$, rather than $B$.
However, since $\Xi_N \cap B$ is distance $\delta L$ away from $\partial B$, it is clear that
this modification only changes
the covariance for any pair of vertices by up to an additive constant $C = C_\delta$, which depends only on $\delta$.
Therefore, by Lemma~\ref{lem-covariance},
$$|\cov(\xi_{u,N} - S_{u,N}, \xi_{v,N} - S_{v,N}) - \cov(Y_{u,N}, Y_{v,N})| \leq C_\delta \mbox{ for all } u, v \in \Xi_N\,.$$
Lemma~\ref{lem-covariance} also implies
$$|\cov(\xi_{u,N} , \xi_{v,N}) - \cov(\eta_{u,N}, \eta_{v,N})| \leq C_\delta \mbox{ for all } u, v \in \Xi_N\,.$$
Together, these two inequalities imply
\begin{equation}\label{eq-MBRW-X-S}
|\cov(S_{u,N}, S_{v,N}) - \cov(X_{u,N} , X_{v,N})| \leq C_\delta \mbox{ for all } u, v \in \Xi_N \,.
\end{equation}

Next, let $r=r_{\delta,h}$ be a sequence of integers (specified in 
Lemma \ref{lem-X-MBRW} below) and, for $v\in \Xi_N$, define
\begin{equation}
  \label{eq-def-S-up}
  S^{\mathrm{up}}_{v, N, r} =   \mbox{$\sum_{j=\ell}^{n-r}
  \sum_{B\in \mathfrak{B}_j(v)}$} \phi_{N, j, B}  \mbox{ and }
  S^{\mathrm{lw}}_{v, N, r} =   \mbox{$\sum_{j=\ell + r}^{n}
  \sum_{B\in \mathfrak{B}_j(v)}$} \phi_{N, j, B}\,.
\end{equation}
Also, define
\begin{equation}\label{eq-def-X-up-lw}
X_{v, N, r}^{\mathrm{up}} =  S^{\mathrm{up}}_{v, N, r}  + \phi_{v, N, r} \mbox{ and } X^{\mathrm{lw}}_{v, N, r} =  S^{\mathrm{lw}}_{v, N, r} + a_{v, N, r}  \phi\,,
\end{equation}
where $\phi_{v, N, r}$ are independent mean zero
Gaussian variables so that $\var  X^{\mathrm{up}}_{v, N, r} = \var X_{v, N}$, and $\phi$ is a standard independent Gaussian variable, with $a_{v, N, r}$ chosen so that $\var X^{\mathrm{lw}}_{v, N, r} = \var X_{v,N}$.
\begin{lemma}\label{lem-X-MBRW}
        There exists $C_\delta$ so that, with $r=r_{\delta,h}=C_\delta
        \log h
        > 0$  and 
 $N\in \mathbb{N}$,
\begin{equation}
        \label{eq-morning3}
        \E  X^{\mathrm{up}}_{v, N, r} X^{\mathrm{up}}_{u, N, r} \leq \E X_{u, N} X_{v, N} \leq \E  X^{\mathrm{lw}}_{v, N, r} X^{\mathrm{lw}}_{u, N, r} \mbox{ for all } u, v\in \Xi_N\,.
\end{equation}
\end{lemma}
\begin{proof}
{For $u=v$, the statement holds since  $\var  X^{\mathrm{up}}_{v, N, r} = \var X_{v, N} = \var  X^{\mathrm{lw}}_{v, N, r} $.}
For $u, v\in \Xi_N$, with $0<d_N(u, v) \leq 2^{-r}N$,
$$\E X_{v, N, r}^{\mathrm{up}} X_{u, N, r}^{\mathrm{up}}  = \E S_{v, N, r}^{\mathrm{up}} S_{u, N, r}^{\mathrm{up}}  \leq \E S_{v, N} S_{u, N}  - r/2\,.$$
Employing \eqref{eq-MBRW-X-S}, for $r\ge r_0$, with $r_0$ depending only on $\delta$, 
one has $\E S_{v,N}S_{u,N}\leq \E X_{v,N}X_{u,N}+r/2$ and therefore
 $\E X_{v, N, r}^{\mathrm{up}} X_{u, N, r}^{\mathrm{up}} \leq \E X_{v,N} X_{u,N}$. If $d_N(u, v) > 2^{-r}N$, then
$$\E X_{v, N, r}^{\mathrm{up}} X_{u, N, r}^{\mathrm{up}}  = \E S_{v, N, r}^{\mathrm{up}} S_{u, N, r}^{\mathrm{up}}  = 0 \leq \E X_{u,N} X_{v,N}\,.$$
This demonstrates the left inequality. 

The right inequality 
requires more work.  We first note that,
for given $|u-v|\leq 2^{\ell+r/2}$
and a constant $C$ not 
depending on $\ell$ and $r$, 
\begin{equation}
        \label{eq-1001}
        \E(X_{v,N,r}^{\mathrm{lw}}-
        X_{u,N,r}^{\mathrm{lw}})^2\leq C \sqrt{A(u,v)}[2^{-r/2}+\sqrt{A(u,v)
        }/r],
\end{equation}
where $A(u,v)=(|u-v|/L)^2\wedge 1$.
To show this, we use
$$
        \E(X_{v,N,r}^{\mathrm{lw}}-
        X_{u,N,r}^{\mathrm{lw}})^2=
        \E(S_{v,N,r}^{\mathrm{lw}}-
        S_{u,N,r}^{\mathrm{lw}})^2+ (a_{v,N,r}-a_{u,N,r})^2\,.$$
        One can check that the first term on the right hand side of this display 
        gives the exponential term 
        in \eqref{eq-1001}. For the $1/r$ term in \eqref{eq-1001}, we use
        \begin{equation}
        \label{eq-4new}a_{v,N,r}^2=\E X_{v,N}^2-\frac{2\log 2}{\pi}(n-\ell-r)=
        \E\eta_{v,N}^2-\E Y_{v,N}^2-\frac{2\log 2}{\pi}(n-\ell-r)=r+b(v)\end{equation}
        where $b(v)=O(1)$ and $|b(v)-b(u)|\leq C\sqrt{A(u,v)}$;
        in the last estimate we used the fact that $Y_{v,N}$ is 
        a GFF and \eqref{eq-100113anewadd}, noting that the truncation
        in the definition of $A(u,v)$ is employed for $u,v$ belonging 
        to different
        boxes $B_u,B_v\in {\mathcal B}_N$.
        Since $\sqrt{r+C+\delta}-\sqrt{r+C} \le \delta / (2\sqrt{r})$, the
         term $1/r$  in \eqref{eq-1001} also follows.

        From the lower bound of Lemma \ref{lem-coarsecov}, 
        {there exists a constant $\tilde C_\delta>0$ such that}
$$\E(X_{v,N}-X_{u,N})^2\geq { \tilde C_\delta} A(u,v)\,.$$
Together with \eqref{eq-1001},
this implies the right inequality in \eqref{eq-morning3} 
for $r$ satisfying
$C 2^{-r/2}\le \sqrt{A(u,v)}/2$, i.e., for $r\geq C_\delta
\log h$.   On the other hand, 
 by \eqref{eq-MBRW-X-S}, for $|u-v|\geq 2^{\ell+r/2}$ and
$r/2>2C_\delta+C'_{\delta}$ (where 
$C'_{\delta}$ is a constant depending only on $\delta$),
$$\E X_{v,N}X_{u,N}\leq \E S_{v,N}S_{u,N}+r/4\leq 
\E X_{v,N,r}^\mathrm{lw}
 X_{u,N,r}^\mathrm{lw},$$
which implies the right inequality in 
\eqref{eq-morning3} in this case as well.
\end{proof}
In light of  Lemma \ref{lem-X-MBRW}, we set $r = C_{\delta}\log h$ from now on 
and drop $r$ from the notation. In particular, we write
$a_{v,N}=a_{v,N,r}$ for $a_{v,N,r}$ as in \eqref{eq-def-X-up-lw}.
By Lemma~\ref{lem-covariance}, for some constant $C$ independent of $N$,
\begin{equation}\label{eq-a-N-v}
a_{v, N}^2  \leq C \,.
\end{equation}
For $v\in \tilde V_N$, we write $v = \bar v + \hat v$, with 
$\bar v = c_{ B_v}$. 
We next claim that there exists a continuous function $g_1: (0, 1)^2 \mapsto [0,\infty)$ and a function  $g_2:(-L/2,L/2)^2\mapsto \R$ so that 
\begin{equation}\label{eq-a-N-v-converge}
  \limsup_{N\to\infty} \sup_{x\in [\delta,(1-\delta)]^2}
  |a_{c_{B_{[xN]}+\hat v}, N}^2 
        -(g_1(x)+ g_2(\hat v ))|=0.
        \end{equation}
 
To see \eqref{eq-a-N-v-converge}, 
        set $B=B_{[xN]}$, $c_x=c_{B_{[xN]}}$ and  $u=c_x+\hat v$. Applying \eqref{eq-4new} and \cite[Lemma 4.6.2]{LL10},
we can write 
$$ a_{u, N}^2=\sum_{w\in \partial B}\P^u(S_{\tau_B}=w)\left[\sum_{z\in \partial V_N} \P^w(S_{\tau_{V_N}}=z)f(z,w,u)\right]\,,$$
where 
$$f(z,w,u)=
(a(z-u)-a(w-u)-\frac{2\log 2}{\pi}(n-\ell-r))=\frac{2\log 2}{\pi}\log (|z-u|/N)+
g(\ell,r,w-u)+O(1/N)\,,$$
for some function $g$, where \eqref{eq-100113c} was used in the second equality. Note that
$\P^u(S_{\tau_B}=w)$ is independent of $N$ and depends only on the relative
position of $u\in B$ and $w\in \partial B$. Applying
\cite[Proposition 8.1.4]{LL10}, one obtains that 
$$a_{u,N}^2=
\sum_{w\in \partial B}
\P^u(S_{\tau_B}=w) 
\frac1N \sum_{z\in \partial V_N} k(c_x/N,z/N) f(z,w,u)+O(1/N),$$
where $k:(0,1)^2\times \partial[0,1]^2 \to \R$ is an explicit function that is uniformly continuous in the first argument on compact subsets of $(0,1)^2$.
Therefore,
$$a_{u,N}^2=\frac{2\log 2}{\pi}\log (|z-c_{x}|/N)-
\sum_{w\in \partial B}
\P^u(S_{\tau_B}=w)
 k(c_x/N,w/N) g(\ell,r,w-u)+O(1/N)\,,$$
and  \eqref{eq-a-N-v-converge} follows. (We remark that it is not difficult to conclude from the above argument that, in fact, $g_2(\hat v)=\bar g_2(\hat v/L)+O(1/L)$ for a function $\bar g_2$ that is continuous in the interior of $(-1/2,1/2)^2$. We will not use this fact.)

        
        Finally, for every $v\in \tilde V_N$, define
$$\eta^{\mathrm{up}}_{v,N} = X^{\mathrm{up}}_{c_{v},N} + 
Y_{v,N} \mbox{ and } \eta^{\mathrm{lw}}_{v,N} =
X^{\mathrm{lw}}_{c_{v},N} + Y_{v,N} \,,$$
with $\{Y_{v,N}\}$ independent of the fields
$X^{\mathrm{up}}$ and $X^{\mathrm{lw}}$.
By Lemma~\ref{lem-X-MBRW} and Lemma~\ref{lem-slepian},
for all $N\in \mathbb{N}$ and $\lambda \in \mathbb{R}$,
\begin{equation}\label{eq-eta-up-lw}
\P(\max_{v\in \tilde V_N} \eta^{\mathrm{lw}}_{v,N} \geq \lambda )\leq \P(\max_{v\in \tilde V_N} \tilde \eta_{v,N} \geq \lambda) \leq \P(\max_{v\in \tilde V_N} \eta^{\mathrm{up}}_{v,N} \geq \lambda)\,.
\end{equation}
Therefore, by Lemma~\ref{lem-eta-tilde-eta},
\eqref{eq-max-subset} (applied to $V_N\setminus \tilde V_N$)
 and \eqref{eq-right-tail},
\begin{equation}\label{eq-lower-right-tail}
        \liminf_{z\to \infty}\liminf_{N\to\infty}
        \frac{\P(\max_{v\in \tilde V_N} \eta^{\mathrm{up}}_{v,N} \geq  m_N + z)}
        {z \mathrm{e}^{-\sqrt{2\pi} z}}
        \gtrsim 
        1\,.
\end{equation}
The fields $\eta^{\mathrm{up}}$ and $\eta^{\mathrm{lw}}$
are the approximations of the GFF that will be employed in
the proof of Proposition \ref{prop-limiting-tail-gff-delta}.
\subsection{Enumeration of the large clusters determines the limiting tail}
 Write $\gamma = \sqrt{2\log 2/\pi}$ and
 let $\Xi_N$ be as in \eqref{eq-oops240113a}.
For $v\in \Xi_N$, set $n_{v,N} = \gamma^{-2}\var X_{v,N}$;
one can check that $n-n_{v,N}$ is of order $\ell$.
For convenience, we now view each $X_{v,N}$ as
the value at time $n_{v,N}$ of a Brownian motion with variance rate $\gamma^2$.
More precisely, we assign to
each Gaussian variable $\phi_{N, j, B}$ in
\eqref{eq-def-S-up}
an independent Brownian motion, with variance rate $\gamma^2$,
that runs for $2^{-2j}$ time units and ends at the value $\phi_{N, j , B}$.
 In the same manner, we associate to
$\phi_{v, N, r}$
in
\eqref{eq-def-X-up-lw}
a Brownian motion of variance rate $\gamma^2$ that runs for $\gamma^{-2}\var \phi_{v, N, r}$ time units and ends at the value $\phi_{v, N, r}$.
For the
Gaussian variable $\phi$ in \eqref{eq-def-X-up-lw},
we employ a standard Brownian motion $\{W_t: t\in [0, 1]\}$, with $W_1 = \phi$. When adding $a_{v,N}\phi$ to a
vertex $v\in \Xi_N$, as in
\eqref{eq-def-X-up-lw},
we consider the Brownian motion $W_{v, N, t} = \gamma W_{\gamma^{-2} a^{-2}_{v, N}t}$, with $t\in [0, \gamma^2 a^2_{v, N} ]$.

We now define a Brownian motion $\{X^{\mathrm{up}}_{v,N}(t) : 0\leq t\leq n_{v,N}\}$ ($\{X^{\mathrm{lw}}_{v,N}(t): 0\leq t\leq n_{v, N}\}$) by concatenating each of the previous Brownian motions associated with $v$, with earlier times corresponding to larger boxes (where we view $\phi_{v,N,r}$ as associated with a box of size $0$ and $\phi$ as associated with a box of size $\infty$). From our construction, we see that $X_{v,N}^{\mathrm{up}}(n_{v,N}) = X_{v,N}^{\mathrm{up}}$ and $X_{v,N}^{\mathrm{lw}}(n_{v,N}) = X_{v,N}^{\mathrm{lw}}$. We write $n^* = \max_{v\in \Xi_N} n_{v,N}$ and define
\begin{equation}\label{eq-big-definition}
\begin{split}
E_{v, N}^{\mathrm{up}}(z) &= \{X_{v,N}^{\mathrm{up}}(t) \leq z + \frac{m_N}{n}t \mbox{ for all } 0\leq t\leq n_{v,N}, \mbox{ and } \max_{u\in \tilde B_v} \eta_{u,N}^{\mathrm{up}} \geq m_N + z\}\,,\\
F_{v, N}^{\mathrm{up}}(z) &= \{X_{v,N}^{\mathrm{up}}(t) \leq z + \frac{m_N}{n}t  + 10 (\log (t \wedge (n^*-t)))_+ + z^{1/20}\\
&\qquad\mbox{ for all } 0\leq t\leq n_{v,N}, \mbox{ and } \max_{u\in \tilde B_v} \eta_{u,N}^{\mathrm{up}} \geq m_N + z\}\,,\\
G_{N}^{\mathrm{up}}(z) &= \bigcup_{v\in \Xi_N}
\bigcup_{0\leq t\leq n_{v,N}}\{X_{v,N}^{\mathrm{up}}(t) >
z+ \frac{m_N}{n}t  + 10 (\log (t \wedge (n^*-t)))_+ + z^{1/20}\}\,,\\
E_{v, N}^{\mathrm{lw}}(z) &= \{X_{v,N}^{\mathrm{lw}}(t) \leq z + \frac{m_N}{n}t \mbox{ for all } 0\leq t\leq n_{v,N}, \mbox{ and } \max_{u\in \tilde B_v} \eta_{u,N}^{\mathrm{lw}} \geq m_N + z\}\,.
\end{split}
\end{equation}
Also define
$$\Lambda^{\mathrm{up}}_{N, z} = \sum_{v\in \Xi_N} \one_{E_{v, N}^{\mathrm{up}}(z) }\,,   \Gamma^{\mathrm{up}}_{N, z} = \sum_{v\in \Xi_N}
\one_{F_{v, N}^{\mathrm{up}}(z) }\,,
\Lambda^{\mathrm{lw}}_{N, z} = \sum_{v\in \Xi_N}
\one_{E_{v, N}^{\mathrm{lw}}(z) }
\,,$$
and, for a box $A\subseteq [\delta, 1-\delta]^2$, define
$$\Lambda_{N, z}^{\mathrm{lw}}(A) = \sum_{v\in \Xi_N \cap N A} \one_{E_{v, N}^{\mathrm{lw}}(z) }\,.$$
In words, the random variable $\Lambda_{N,z}^{\mathrm{up}}$ counts the number
of boxes in $\tilde {\mathcal B}_N$ whose ``backbone'' path $X_{v,N}^{\mathrm{up}}(\cdot)$
stays below a linear path connecting $z$ to roughly $m_N + z$, so that
one of its ``neighbors'' achieves a terminal value that is at least
$m_N+z$; the random variable $\Gamma_{N,z}^{\mathrm{up}}$ similarly counts boxes in $\tilde {\mathcal B}_N$
whose backbone is constrained to stay below a slightly
``upward bent'' curve.
Clearly, $E_{v, N}^{\mathrm{up}}(z) \subseteq F_{v, N}^{\mathrm{up}}(z)$ always holds,
as does $\Lambda_{N,z}^{\mathrm{up}}(z) \le \Gamma_{N,z}^{\mathrm{up}}(z)$.

Note that it follows from their definitions that,
for fixed $v$, the processes $X_{v,N}^\mathrm{up}(\cdot)$ and
$X_{v,N}^\mathrm{lw}(\cdot)$ have the same distribution. Furthermore,
for any fixed $v\in\Xi_N$, and in particular for $c_v=v$,
\begin{equation}
\label{eqnewMB3}
\max_{u\in \tilde B_v} \eta_{u,N}^{\mathrm{up}}=
X_{c_v,N}^{\mathrm{up}}+
\max_{u\in \tilde B_v} Y_{u,N}\,,
\quad
\max_{u\in \tilde B_v} \eta_{u,N}^{\mathrm{lw}}=
X_{c_v,N}^{\mathrm{lw}}+
\max_{u\in \tilde B_v} Y_{u,N}\,,
\end{equation}
where
$\max_{u\in \tilde B_v} Y_{u,N}$ is independent of both
$X_{c_v,N}^{\mathrm{up}}$ and
$X_{c_v,N}^{\mathrm{lw}}$.
Therefore,
for any fixed $v\in\Xi_N$,
the events $E_{v,N}^\mathrm{up}(z)$ and $E_{v,N}^{\mathrm{lw}}(z)$ have the same probability (neither of which depends on the choice of $r=r_{\delta,h}$), which implies that
\begin{equation}
  \label{eq-morningcoffee}
  \E\Lambda_{N,z}^{\mathrm{lw}}=\E \Lambda_{N,z}^{\mathrm {up}}\,.
\end{equation}
The main result of this subsection is the following proposition.
\begin{prop}\label{prop-gff-first-moment-dictates}
There exist $\delta_z\geq 0$ with $\delta_z \searrow_{z\to \infty} 0$ such that,
for any open box $A\subseteq [\delta, 1-\delta]^2$,
\begin{equation}
  \label{eq-prop4.7A}
  \limsup_{z\to \infty} 
  \limsup_{N\to \infty} \frac{\P(\max_{v\in \tilde V_N \cap N A}
  \eta_{v,N} \geq m_N +z )}{\E \Lambda_{N, z - \delta_z}^{\mathrm{lw}}(A)} 
  \leq 1\leq   \liminf_{z\to \infty} 
  \liminf_{N\to \infty} \frac{\P(\max_{v\in \tilde V_N \cap N A} 
  \eta_{v,N} \geq m_N +z )}{\E \Lambda_{N, z + \delta_z}^{\mathrm{lw}}(A)}\,.
\end{equation}
In particular,
\begin{equation}
\label{eq-prop4.7}
\limsup_{z\to \infty} \limsup_{N\to \infty} 
\frac{\P(\max_{v\in \tilde V_N} \eta_{v,N} \geq m_N +z )}{\E 
        \Lambda_{N, z - \delta_z}^{\mathrm{lw}}}  \leq 1\leq   
        \liminf_{z\to \infty} 
        \liminf_{N\to \infty} \frac{\P(\max_{v\in \tilde V_N} \eta_{v,N} 
\geq m_N +z )}{\E \Lambda_{N, z + \delta_z}^{\mathrm{lw}}}\,.
\end{equation}
\end{prop}
The proof of \eqref{eq-prop4.7A} does not require more work than the
proof of \eqref{eq-prop4.7}, but it does involve more notation;
for the sake of economy, we therefore only prove \eqref{eq-prop4.7}. The display \eqref{eq-prop4.7}
is an immediate consequence of
Lemma~\ref{lem-eta-tilde-eta}, \eqref{eq-eta-up-lw},
\eqref{eq-morningcoffee} and the next proposition.
(The analog of Proposition~\ref{prop-first-moment-is-king}, but with
$\E\Lambda_{N,z}^\cdot(A)$ and $\tilde V_N\cap N A$ replacing
$\E\Lambda_{N,z}^\cdot$ and $\tilde V_N$, also holds using 
the same argument.)
\begin{prop}\label{prop-first-moment-is-king}
With notation as above,
\begin{align*}
\lim_{z\to \infty} \limsup_{N\to \infty}
\frac{\P(\max_{v\in \tilde V_N} \eta_{v,N}^{\mathrm{lw}} \geq m_N +z )}
{\E \Lambda_{N, z}^{\mathrm{lw}}}& =
\lim_{z\to \infty} \liminf_{N\to \infty} \frac{\P(\max_{v\in \tilde V_N}
\eta_{v,N}^{\mathrm{lw}} \geq m_N +z )}{\E \Lambda_{N, z}^{\mathrm{lw}}} =
1\,,\\
\lim_{z\to \infty} \limsup_{N\to \infty} \frac{\P(\max_{v\in \tilde V_N}
\eta_{v,N}^{\mathrm{up}} \geq m_N +z )}{\E \Lambda_{N, z}^{\mathrm{up}}}& =
\lim_{z\to \infty} \liminf_{N\to \infty} \frac{\P(\max_{v\in \tilde V_N}
\eta_{v,N}^{\mathrm{up}} \geq m_N +z )}{\E \Lambda_{N, z}^{\mathrm{up}}} =
1\,.
\end{align*}
\end{prop}

In order to prove Proposition \ref{prop-first-moment-is-king},
we separately derive
upper and lower bounds.  For these bounds,
we consider truncations of the MBRW profile with respect to certain
upper and lower curves, as in the
definitions of $F_{v, N}^{\mathrm{up}}(z)$ and $E_{ v, N}^{\mathrm{lw}}(z)$
in \eqref{eq-big-definition}.  In defining these truncations,
the following two requirements are crucial:
\begin{itemize}
  \item[(1)] The two truncations result asymptotically in the
same probability; this will be shown
in Lemma~\ref{lem-Gamma-Lambda}
(the underlying reason being the bounds in Lemma~\ref{lem-1DRW}).
\item[(2)]
  After truncation with respect to the lower curve, the resulting second moment
  is asymptotically the same as the corresponding first moment; this
  will be shown in Lemma~\ref{lem-second-moment}. (In the lemma, we will examine
  $\Lambda_{N, z}^{\mathrm{lw}}$, rather than the total number of vertices whose paths lie below a given curve and end above $m_N +z$; this leads to an improvement of the
  bound in, e.g.,
  \cite{DZ12}, and allows us to give precise asymptotics for our tail estimates.)
  \end{itemize}

We first compare $\Lambda_{N, z}^{\mathrm{up}}$ and
$\Gamma_{N, z}^{\mathrm{up}}$, and start with the following estimate.
\begin{lemma} \label{lem-decorelate-tilde-B}
There exists $N_0=N_0(z,\ell,\tilde \ell, \delta)$ such that,
for any $\tilde B_1\neq
\tilde B_2 \in \mathcal{\tilde B}_N$, any $\lambda_1, \lambda_2>0$ and
any $N>N_0$,
\begin{align}
&\P(\max_{u\in \tilde B_1} Y_{u,N} \geq \ell m_N/n + \lambda_1, \max_{u\in \tilde B_2} Y_{u,N} \geq \ell m_N/n + \lambda_2) \nonumber \\
\lesssim & (\log z)^C \ell^{-3} (\lambda_1+\log \ell)(\lambda_2 + \log \ell) \mathrm{e}^{-\sqrt{2\pi} ( \lambda_1 + \lambda_2)} \mathrm{e}^{- C^{-1} (\lambda_1^2 + \lambda_2^2 )/\ell}  \label{eq-tilde-B-1-2}\,,
\end{align}
where $C>0$ is an absolute constant. Moreover, for $\lambda_1>-\log\ell+1$,
\begin{equation}\label{eq-tilde-B-1}
\P(\max_{u\in \tilde B_1} Y_{u,N} \geq \ell m_N/n + \lambda_1) \lesssim(\log z)^C \ell^{-3/2}(\lambda_1+\log \ell) \mathrm{e}^{-\sqrt{2\pi} \lambda_1} \mathrm{e}^{-C^{-1} \lambda_1^2/\ell} \,.
\end{equation}
\end{lemma}
Comparing (\ref{eq-tilde-B-1-2}) and (\ref{eq-tilde-B-1}), note that
the constant $C$ in \eqref{eq-tilde-B-1-2} has been chosen large enough to
absorb the contribution of the correlation between the two events to the upper
bounds (up to a power of $\log z$).
\begin{proof}
We give a proof for \eqref{eq-tilde-B-1-2} and omit the proof of \eqref{eq-tilde-B-1} (which is simpler).
Let $\hat B_1$ and $\hat B_2$ be boxes of side length
$\tilde L (1 - \delta)$ that have the same centers as $\tilde B_1$ and $\tilde B_2$ (and thus $\tilde B_i\subset \hat B_i$).
For $u\in \tilde B_i$ (i=1,2), we define
$$\Phi_{u, N} = \E (Y_{u,N} \mid \{Y_{w,N} :  w\in \partial \hat B_i\}) \mbox{ and } \Psi_{u, N} = Y_{u,N} - \Phi_{u, N}\,.$$
Clearly, $\{\Psi_{u, N} : u\in \tilde B_1\}$ is independent of $\{\Psi_{u, N}: u\in \tilde B_2\}$.
Repeating the computations from Lemma
\ref{lem-coarsecov}, we have that, for $u,v$ in the same box $\tilde B_i$,
\begin{equation}
\label{eq-friday1}
\E(\Phi_{u, N}-\Phi_{v, N})^2\leq c(\delta) \frac{|u-v|}{\tilde L(1-2\delta)}\,.
 \end{equation}
 Furthermore, as in \eqref{eq-110113g} in the proof of Lemma \ref{lem-coarselimit},
 $\var \Phi_{u,N}$ can be represented as the difference of the variances
 of GFFs in boxes of side length $L$ and $(1-\delta)\tilde L$, which leads to
\begin{equation}
\label{eq-friday2}
\sigma^2 : = \max_{u\in \tilde B_1 \cup \tilde B_2} \var \Phi_{u, N} \lesssim \log\log z\,,\end{equation}
where we have used $h=L/\tilde L\leq \log z$ and Lemma \ref{lem-covariance}.
By \eqref{eq-friday1}
and Lemma \ref{lem-ferniquecriterion} (applied to
boxes of side length $\tilde L(1-2\delta)$),
$\E W_i \lesssim 1$, where $W_i := \max_{u\in \tilde B_i} \Phi_{u, N}$ for $i \in \{1, 2\}$.
By the last inequality, \eqref{eq-friday2} and
Lemma~\ref{lem-gaussian-concentration}, there exists
a constant $C>0$ such that
\begin{equation}
  \label{eq-ouf240113}
  \P(W_i \geq \lambda-1) \lesssim \mathrm{e}^{-\lambda^2/(2C \log\log z)}\,.
\end{equation}

We now write
\begin{align*}
&\P(\max_{u\in \tilde B_1} Y_{u,N} \geq \ell m_N/n + \lambda_1, \max_{u\in \tilde B_2} Y_{u,N} \geq \ell m_N/n + \lambda_2)\\
  \lesssim&
  \int_0^\infty \max_{i\in \{1, 2\}}\P(W_i\geq \lambda-1)
  \P(\max_{u\in \hat B_1} \Psi_{u,N}
  \geq \ell m_N/n + \lambda_1 - \lambda) \P(\max_{u\in \hat B_2}
  \Psi_{u,N}
  \geq \ell m_N/n+\lambda_2 - \lambda) d\lambda\\
  =&
  \int_0^\infty \max_{i\in \{1, 2\}}\P(W_i\geq \lambda-1)
  \P(\max_{u\in \hat B_1} \Psi_{u,N}
  \geq m_\ell  + \bar c \log \ell+\lambda_1 - \lambda+O(\ell\log n/n))
 \\
 &\quad \quad\quad \quad
 \cdot \P(\max_{u\in \hat B_2}
  \Psi_{u,N}
  \geq m_\ell+\bar c\log \ell+\lambda_2 - \lambda+O(\ell\log n/n)) d\lambda\,,
\end{align*}
where $\bar c=3\sqrt{2/\pi}/4$; we have used \eqref{eq-friday3} in the equality
to approximate $\ell m_N /n$ by $m_{\ell}$.
From the last estimate, \eqref{eq-large-tail} in Lemma~\ref{lem-prelim-tail}
(applied in the boxes $\hat B_1,\hat B_2$ instead of $V_N$)
and  \eqref{eq-ouf240113}, it follows that
\begin{align*}
&\P(\max_{u\in \tilde B_1} Y_{u,N} \geq \ell m_N/n + \lambda_1, \max_{u\in \tilde B_2} Y_{u,N} \geq \ell m_N/n + \lambda_2)\\
  \lesssim &  \int_0^\infty \mathrm{e}^{-\lambda^2/(2 C\log\log z)} \ell^{-3} (\lambda_1+\log \ell)(\lambda_2+\log \ell) \mathrm{e}^{-\sqrt{2\pi} (\lambda_1 + \lambda_2 - 2\lambda)}  \mathrm{e}^{- C^{-1} ((\lambda_1 - \lambda)^2 + (\lambda_2 - \lambda)^2 )/\ell}  d\lambda\\
\lesssim&   \ell^{-3} (\lambda_1+\log \ell)(\lambda_2+\log \ell) \mathrm{e}^{- C^{-1} (\lambda_1^2 + \lambda_2^2 )/\ell} \int_0^\infty \mathrm{e}^{-\lambda^2/(2 C\log\log z)} \mathrm{e}^{2(\sqrt{2\pi} + C(\lambda_1 + \lambda_2)/\ell)\lambda} d\lambda\\
\lesssim&   \ell^{-3} (\lambda_1+\log \ell)(\lambda_2+\log \ell) \mathrm{e}^{-\sqrt{2\pi} ( \lambda_1 + \lambda_2)} \mathrm{e}^{-\frac{\lambda_1^2 + \lambda_2^2}{C \ell}}\mathrm{e}^{2C(\log\log z) (\sqrt{2\pi} + \frac{C}{\ell}(\lambda_1 + \lambda_2))^2}\\
\qquad\qquad& \times \int_0^\infty \mathrm{e}^{-\frac{(\lambda -2C(\log\log z)
(\sqrt{2\pi} + C(\lambda_1 + \lambda_2)/\ell) )^2}{2 C\log\log z}} d\lambda\\
  \lesssim &   (\log z)^{C'} \ell^{-3} (\lambda_1+\log \ell)(\lambda_2 + \log \ell) \mathrm{e}^{-\sqrt{2\pi} ( \lambda_1 + \lambda_2)} \mathrm{e}^{- C'^{-1} (\lambda_1^2 + \lambda_2^2 )/\ell} \,,
\end{align*}
where $C'$ is a large absolute constant. (In the last inequality, we used the assumption that $z^4\leq \ell$.)
\end{proof}
\begin{lemma}\label{lem-Gamma-Lambda}
For $\Lambda^{\mathrm{up}}_{N, z}$ and $\Gamma^{\mathrm{up}}_{N, z}$ as above,
\begin{equation}
  \label{eq-clear240113}
\lim_{z\to \infty} \liminf_{N\to \infty}
\frac{\E \Lambda^{\mathrm{up}}_{N, z}}{\E \Gamma^{\mathrm{up}}_{N, z}}
= 1\,.
\end{equation}
\end{lemma}
(Of course, the $\liminf$ in \eqref{eq-clear240113} also implies the same statement, but with $\limsup$ replacing $\liminf$,
since the ratio is always bounded above by $1$.)
\begin{proof}
To simplify notation, we drop the superscript ``$\mathrm{up}$" from the
notation in this proof.
For any $v\in \Xi_N$,
we write $\bar X_{v,N}(t) = X_{v,N}(t) - \tfrac{m_N t}{n}$,
and define the probability measure $\mathbb{Q}$ by
\begin{equation}\label{eq-change-of-measure}
\frac{d\mathbb{P}}{d\mathbb{Q}} = \mathrm{e}^{-\tfrac{m_N}{n\gamma^2} \bar X_{v,N}(n_{v,N}) - \tfrac{m_N^2}{2 \gamma^2 n^2} n_{v,N}}\,.
\end{equation}
Under $\mathbb{Q}$, $\bar X_{v,N}(t)$ is a Brownian motion
with variance rate $\gamma^2$.

We continue to use the notation $\mu_{t, y}(\cdot)$ and
$\mu_{t, y}^*(\cdot)$ from \eqref{eq-def-mu-mu*} (with variance rate $\sigma^2 = \gamma^2$).
With a slight abuse of notation,
we write $d\P/d\mathbb{Q}=(d\P/d\mathbb{Q})(\bar X_{v,N})$. We have
\begin{align*}
\P(F_{v, N} (z)\setminus E_{v, N} (z))  \leq & \int_{0}^{z+z^{1/20}}
\frac{d \mathbb{P}}{d\mathbb{Q}} (x) \mu^*_{n_{v,N}, z} (x)
\P(\max_{u\in \tilde B_{v, N}} Y_{u,N} \geq \ell m_N/n + z- x ) dx \\
&+ \int_{-\infty}^0 \frac{d \mathbb{P}}{d\mathbb{Q}} (x)
(\mu^*_{n_{v,N}, z} (x) - \mu_{n_{v,N}, z}(x))
\P(\max_{u\in \tilde B_{v, N}} Y_{u,N} \geq \ell m_N/n + z- x ) dx\\
&\lesssim 4^{-n^*} z^{3} (z+\log \ell) (\log z)^C \ell^{-3/2} \mathrm{e}^{-\sqrt{2\pi} z} + \delta_z \P(E_{v, N}(z))\,,
\end{align*}
where $\delta_z \searrow_{z\to \infty} 0$;
the last inequality follows for large $\ell$ by rewriting $\ell m_N/n$ in terms
of $m_\ell$,
and applying Lemmas~\ref{lem-1DRW} and
\ref{lem-decorelate-tilde-B}. Therefore,
\begin{equation}\label{eq-difference-Gamma-Lambda}
\E \Gamma_{N, z}- \E \Lambda_{N, z} \lesssim z^{3} (z+\log \ell) (\log z)^C \ell^{-3/2} \mathrm{e}^{-\sqrt{2 \pi}z} + \delta_z \E \Lambda_{N, z}\,.
\end{equation}
By \eqref{eq-lower-right-tail} and Lemma~\ref{lem-a-priori} (applied with
$\beta=z+z^{1/20}$),
\begin{equation}
\label{eq-lower-Gamma}
\liminf_{z\rightarrow\infty}\liminf_{N\rightarrow\infty} \frac{\E \Gamma_{N, z}}{ z \mathrm{e}^{-\sqrt{2\pi} z}} \gtrsim 1\,.
\end{equation} Since $z^4 \le \ell$,
 \eqref{eq-difference-Gamma-Lambda}  and (\ref{eq-lower-Gamma}) imply (\ref{eq-clear240113}),
which completes the proof of the lemma.
\end{proof}

We next estimate the second moment of $\Lambda^{\mathrm{lw}}_{N, z}$.
\begin{lemma}\label{lem-second-moment}
With notation as above,
\begin{equation}
\label{newequation72}
\lim_{z\to \infty} \limsup_{N \to \infty} \frac{\E (\Lambda^{\mathrm{lw}}_{N, z})^2}{\E \Lambda^{\mathrm{lw}}_{N, z}} = 1\,.
\end{equation}
\end{lemma}
\begin{proof}
Recall from \eqref{eq-morningcoffee} that $\E \Lambda^{\mathrm{lw}}_{N, z} =
\E \Lambda^{\mathrm{up}}_{N, z}$.
Combined with
Lemma~\ref {lem-Gamma-Lambda}  and \eqref{eq-lower-Gamma}, this implies
\begin{equation}\label{eq-Lambda-first-moment}
\liminf_{z\rightarrow\infty}\liminf_{N\rightarrow\infty} \frac{\E \Lambda^{\mathrm{lw}}_{N, z}}{ z \mathrm{e}^{-\sqrt{2\pi} z}} \gtrsim 1\,.
\end{equation}
The main work is to estimate the above second moment, which we rewrite as
\begin{equation}
\label{newequation73}
\E(\Lambda^{\mathrm{lw}}_{N, z})^2 =
\E\Lambda^{\mathrm{lw}}_{N, z}  +
\sum_{v,w\in \Xi_N,v\neq w} \P(E_{v, N}^{\mathrm{lw}}(z) \cap E_{w, N}^{\mathrm{lw}}(z))\,.
\end{equation}
We will proceed by decomposing the terms in the sum on the right hand side of (\ref{newequation73}) 
according to first the time and then the position at which pairs of sites ``branch".
(This is a modification of an argument that has been applied repeatedly in the 
literature for branching Brownian motion and branching random walk.)

To simplify notation, we drop the superscript $\mathrm{lw}$  in the remainder
of the proof of the lemma.  We also employ the following terminology.
Recalling the formula \eqref{eq-change-of-measure},
for any $v\in \Xi_N$,
we write $\bar X_{v,N}(t) = X_{v,N}(t) - {m_N t}/{n}$, with
$\bar X_{v,N}(t) = \bar X_{v,N}(n_{v,N})$ for $n_{v,N} \leq t\leq n$ and
$n^* = \max_{v\in \Xi_N} n_{v,N}$.
By \eqref{eq-a-N-v},
$|n^* -n_{v, N} |= O(1)$ uniformly in $v\in \Xi_N$.
For $v, w \in \Xi_N$, we say that $v$ and $w$ \emph{split}
at time $t_s = n^*-s$, denoted by $v\sim_s w$,
if $s$ is the maximal integer
such that $\{X_{v,N}(t) - X_{v,N}(t_s): t_s\leq t\leq n^*\}$
is independent of $\{X_{w,N}(t) - X_{w,N}(t_s): t_s\leq t\leq n^*\}$.
We note that $0\leq s \leq n^*- r$ and that the number of pairs of sites $v$ and $w$ satisfying $v,w \in \Xi_N$, and for which $v\sim_s w$,  is within a
constant multiple of $h^4 4^{n^*} 4^{s+r}$.  Recall that $r= C_{\delta}\log h$ and $h = L/\tilde{L} \le \log z$.

%

We will show that for large $z$, after letting $N\rightarrow\infty$, the sum in (\ref{newequation73}) is small in comparison with the first term on
the right hand side, by decomposing it into three parts, with $v$ and $w$ 
satisfying
$v\sim_s w$, with $s$ restricted to $[n^*- z^{1/10}, n^* - r]$,  $[z^{1/10}, n^*- z^{1/10})$ and
$[1, z^{1/10})$, respectively. For $v\sim_s w$, with given $s$ in either of 
        the first two intervals, we will employ the upper bound
\begin{align}\label{eq-P-A-v-w}
&\P(E_{v, N} (z) \cap E_{w, N}(z))\nonumber\\
= &\P(\bar X_{v,N}(t) , \bar X_{w,N}(t) \leq z \mbox{ for all }t\in [0,n^*];  \max_{u\in \tilde B_v} \eta_{u,N},  \max_{u\in \tilde B_w} \eta_{u,N} \geq m_N + z)\nonumber\\
= &\sum_{x\leq z}\P(\bar X_{v,N}(t), \bar X_{w,N}(t) \leq z \mbox{ for all }t\in [0,n^*]; \max_{u\in \tilde B_v} \eta_{u,N} , \max_{u\in \tilde B_w} \eta_{u,N} \geq m_N + z; \bar X_{v,N}(t_s) \in (x-1, x])\nonumber\\
\leq &\sum_{x\leq z}\P(\bar X_{v,N}(t) \leq z  \mbox{ for all } t\in [0,t_s], \bar X_{v,N}(t_s) \in [x-1, x]) \Gamma_{v, x, z, s} \Gamma_{w, x, z, s}\,,
\end{align}
where, in the above sum, $x\leq z$ is a shorthand notation for
$x=z,z-1,\ldots$, and
$$\Gamma_{v, x, z, s} := \sup_{\bar X_{v,N}(t_s)\in [x-1,x]}
\P(\bar X_{v,N} (t)\leq z \mbox{ for all } t_s<t\leq n^*,   \max_{u\in \tilde B_v} \eta_{u,N} \geq m_N + z
 \mid \bar X_{v,N}(t_s) ).$$

In order to obtain an upper bound on the sum in (\ref{newequation73}), we first consider the case $v\sim_s w$, with
$n^*- z^{1/10} \leq s \leq n^*$.
Here, $|v-w|_2 \asymp_{z, \tilde L} N$; therefore,
$\tilde B_v$ and $\tilde B_w$ belong
to different boxes in $\mathcal{B}_N$ when $N$ is large and, in particular,
$\{Y_{u,N}: u\in \tilde B_v\}$ is independent of $\{Y_{u,N}: u\in \tilde B_w\}$.
By a change of measure that transforms $\bar X_{v,N}(\cdot)$
into Brownian motion and by the
ballot theorem (see \cite[Theorem 1]{ABR08}),
\begin{align}\label{eq-Gamma-y-z-s}
\Gamma_{v, x, z, s } \leq & \sum_{y\leq z} \P(\bar X_{v,N}(t)\leq z - x \mbox{ for all } t \in [0,s], \bar X_{v,N}(s) \in [y-1-x, y-x]) \gamma_{v, y}\nonumber \\
\lesssim & \sum_{y\leq z} \mathrm{e}^{-\frac{\alpha_n^2}{2} s}
\mathrm{e}^{-\alpha_n (y-x)/\gamma}  \frac{(z-x)(z-y)}{s^{3/2}} \gamma_{v, y}\,,
\end{align}
where $\alpha_n = {m_N}/{\gamma n}$ (with $\gamma=\sqrt{2\log 2/\pi}$, as before)
and $\gamma_{v, y} = \P(\max_{u\in \tilde B_v} Y_{u,N} \geq \ell m_N/N + z - y) $.
On the other hand, obviously
\begin{align*}
\P(\bar X_{v,N}(t) \leq z,  \mbox{ for all } t\in [0,t_s], \bar X_{v,N}(t_s) \in [x-1, x]) & \leq \P(\bar X_{v,N}(t_s) \in [x-1,x])\\
&\lesssim  \frac{1}{\sqrt{t_s}} \mathrm{e}^{-\frac{(\alpha_n t_s + x/\gamma)^2}{2t_s}} \,.
\end{align*}
Substituting
the preceding inequality and \eqref{eq-Gamma-y-z-s} into \eqref{eq-P-A-v-w}, it follows that, for an absolute constant $C>0$,
\begin{align*}
\P(E_{v, N} (z) \cap E_{w, N}(z)) &\lesssim    \mathrm{e}^{C z^{1/10}}4^{-n^*-s}
\sum_{y\leq z} \mathrm{e}^{-\alpha_n y/\gamma} \gamma_{v, y}  (z-y)\sum_{y\leq z} \mathrm{e}^{-\alpha_n y/\gamma}  \gamma_{w, y} (z-y)\,.
\end{align*}
Recalling that $\alpha_n/\gamma = \sqrt{2\pi} + O(\log n)/n$,
an application of Lemma~\ref{lem-decorelate-tilde-B} therefore yields
\begin{align*}
\P(E_{v, N} (z) \cap E_{w, N}(z)) \lesssim &  \frac{ \mathrm{e}^{2C z^{1/10}}}{4^{n^*+s} \mathrm{e}^{2\sqrt{2\pi} z}}
\sum_{y_1, y_2\leq z} \mathrm{e}^{O(\frac{\log n (y_1+y_2)}{n})}   \frac{ (z-y_1)^2(z-y_2)^2}{ \ell^{3} \mathrm{e}^{((z- y_1)^2 + (z - y_2)^2 )/C \ell}} \\
\lesssim & \mathrm{e}^{2C z^{1/10}} 4^{-n^*-s}  \mathrm{e}^{-2\sqrt{2\pi}z} \,.
\end{align*}
Applying (\ref{eq-Lambda-first-moment}),  this implies
\begin{equation}
\label{eq-second-moment-regime-1}
\limsup_{z\rightarrow\infty}\limsup_{N\rightarrow\infty}
\frac{\sum_{n^* - z^{1/10} \leq s\leq n^* -r} \sum_{v\sim_s w} \P(E_{v, N}(z) \cap E_{w, N}(z))} {h^4 4^r\mathrm{e}^{2C z^{1/10}} \mathrm{e}^{-\sqrt{2\pi}z} \E \Lambda_{N, z}} \lesssim 1\,.
\end{equation}

We next consider the case $z^{1/10} \leq s < n^* - z^{1/10}$.
Here,
\eqref{eq-Gamma-y-z-s}
still
holds (since the distance between $v$ and $w$ is large
enough such that they belong to different boxes in $\mathcal{B}_N$).
By the
ballot theorem together with the change of measure that transforms
$\bar X_{v,N}(\cdot)$ into Brownian motion,
\begin{align}\label{eq-decompose-x}
\P(\bar X_{v,N}(t) \leq z,  \mbox{ for all } t\in [0,t_s], \bar X_{v,N}(t_s) \in [x-1, x]) & \lesssim \frac{z(z -x + 1)}{t_s^{3/2}}\mathrm{e}^{-\frac{\alpha_n^2 t_s}{2}} \mathrm{e}^{-\frac{\alpha_n x}{\gamma}}\,.
\end{align}
Substitution of
\eqref{eq-Gamma-y-z-s} and the above estimate into \eqref{eq-P-A-v-w} implies that
\begin{align*}
&\P(E_{v, N} (z) \cap E_{w, N}(z)) \\
\lesssim & \frac{z (n^*)^{3/2}}{ 4^{n^*+s} s^{3/2} t_s^{3/2}}
\sum_{x\leq z} \mathrm{e}^{\alpha_n x/\gamma} (z-x)^3  \sum_{y_1, y_2\leq z} 
\mathrm{e}^{-\alpha_n (y_1+y_2) /\gamma}  (z-y_1)(z-y_2)  
\gamma_{v, y_1}\gamma_{w, y_2}\\
\lesssim & \frac{z(n^*)^{3/2} \mathrm{e}^{\sqrt{2\pi}z}}{ 4^{n^*+s} s^{3/2} t_s^{3/2}}
\sum_{y_1, y_2\leq z} \mathrm{e}^{-\alpha_n (y_1+y_2) /\gamma}  (z-y_1)(z-y_2)  \gamma_{v, y_1} \gamma_{w, y_2}\,.
\end{align*}
Combining this with Lemma~\ref{lem-decorelate-tilde-B}, it follows that
\begin{align*}
\P(E_{v, N} (z) \cap E_{w, N}(z))  \lesssim &  \frac{ (\log z)^{C}z(n^*)^{3/2} \mathrm{e}^{-\sqrt{2\pi}z}}{ 4^{n^*+s}s^{3/2} t_s^{3/2}}
\sum_{y_1, y_2\leq z} \mathrm{e}^{O(\frac{\log n(y_1+y_2)}{ n})} \frac{ (z-y_1)^2(z-y_2)^2}{ \ell^{3} \mathrm{e}^{((z- y_1)^2 + (z - y_2)^2 )/C\ell}} \\
\lesssim&  \frac{ (\log z)^{C}z(n^*)^{3/2} \mathrm{e}^{-\sqrt{2\pi}z}}{ 4^{n^*+s}s^{3/2} t_s^{3/2}}\,.
\end{align*}
Therefore, again applying (\ref{eq-Lambda-first-moment}),
\begin{equation}\label{eq-second-moment-regime-2}
\limsup_{z\rightarrow\infty}\limsup_{N\rightarrow\infty}
\frac{\sum_{z^{1/10} \leq s 
< n^* - z^{1/10}} \sum_{v\sim_s w} \P(E_{v, N}(z) \cap E_{w, N}(z))}
%
%
{h^4 4^r(\log z)^C z^{-1/20}  \E \Lambda_{N, z}}  \lesssim 1\,.
\end{equation}

Lastly, we consider the case $1\leq s < z^{1/10}$.
Let $v\sim_s w$, with $v\neq w$, and define $\Gamma_{v, w, x, z, s}$ by
$$\sup_{\bar X_{v,N}(t_s)\in [x-1,x]}
\P(\bar X_{v,N} (t), \bar X_{w,N}(t)\leq z \mbox{ for all } t_s<t\leq n^*;  \max_{u\in \tilde B_v} \eta_{u,N},  \max_{u\in \tilde B_w} \eta_{u,N}  \geq m_N + z
 \mid \bar X_{v,N}(t_s) ).$$
 Analogous to \eqref{eq-P-A-v-w},
\begin{equation}\label{eq-P-A-v-w-2}
\P(E_{ v, N} (z) \cap E_{ w, N}(z)) \leq \sum_{x\leq z}\P(\bar X_{v,N}(t) \leq z,  \mbox{ for all } t\in [0,t_s], \bar X_{v,N}(t_s) \in [x-1, x]) \Gamma_{v, w, x, z, s} \,.
\end{equation}
Furthermore,
\begin{align*}
\Gamma_{v, w, x, z, s } \leq & \sum_{y_1, y_2\leq z} \P(\bar{X}_{v,N}(t), \bar X_{w,N}(t)\leq z - x \mbox{ for all } t \in [0,s]; \\
&\qquad \qquad \bar X_{v,N}(s) \in [y_1-1-x, y_1-x]; \bar X_{w,N}(s)\in [y_2 - x- 1, y_2 -x]) \gamma_{v, y_1, w, y_2}\nonumber \\
\lesssim & \sum_{y_1, y_2 \leq z} \mathrm{e}^{-\alpha_n^2 s}
\mathrm{e}^{-\alpha_n (y_1+y_2-2x)/\gamma}  s^{-1} \mathrm{e}^{-\frac{(y_1-x)^2 + (y_2-x)^2}{2\gamma^2 s}} \gamma_{v, y_1, w, y_2}\,,
\end{align*}
where $\gamma_{v, y_1, w, y_2} := \P(\max_{u\in \tilde B_v} Y_{u,N} \geq \ell m_N/N + z - y_1, \max_{u\in \tilde B_w} Y_{u,N} \geq \ell m_N/N + z - y_2)$.  Together with Lemma~\ref{lem-decorelate-tilde-B},
the last display implies
\begin{align*}
\Gamma_{v, w, x, z, s } &\lesssim (\log z)^C\mathrm{e}^{-\alpha_n^2 s} \mathrm{e}^{2\alpha_n x/\gamma} \mathrm{e}^{-2\sqrt{2\pi} z} \sum_{y_1, y_2 \leq z}
\frac{\mathrm{e}^{O(\log n(y_1+y_2)/n )}}{ \mathrm{e}^{((y_1-x)^2 + (y_2-x)^2)/2\gamma^2 s}}  \frac{(z-y_1)(z-y_2 )}{\ell^3  s} \\
&\lesssim  (\log z)^C\mathrm{e}^{-\alpha_n^2 s} \mathrm{e}^{2\alpha_n x/\gamma} \mathrm{e}^{-2\sqrt{2\pi} z} (z - x + \sqrt{s} +\log \ell)^2 \ell^{-3} \,.
\end{align*}
Note that \eqref{eq-decompose-x} also holds in this region. Plugging \eqref{eq-decompose-x} and the above inequality into \eqref{eq-P-A-v-w-2}, we obtain
\begin{align*}
\P(E_{v, N} (z) \cap E_{w, N}(z)) &\lesssim  4^{-n^*-s} s^{3/2}  (\log z)^{C} \ell^{-3} z \mathrm{e}^{-\sqrt{2\pi }z} \sum_{x\leq z} (z-x+\sqrt{s} +\log \ell)^3 \mathrm{e}^{-\alpha_n(z-x)/\gamma}\\
& \lesssim 4^{-n^*-s} s^{3/2} (\sqrt{s} + \log \ell)^3 (\log z)^{C} \ell^{-3}  z\mathrm{e}^{-\sqrt{2\pi} z}\,.
\end{align*}
Therefore, since $\ell\geq z^4$, another application of (\ref{eq-Lambda-first-moment}) implies
\begin{equation}
\label{newequation79}
\limsup_{z\rightarrow\infty}\limsup_{N\rightarrow\infty}
\frac{\sum_{1\leq s < z^{1/10}} \sum_{v\sim_s w, v\neq w} \P(E_{v, N}(z) \cap E_{w, N}(z))}
%
%
 {h^4 4^r (\log z)^{C} \ell^{-2}  \E \Lambda_{N, z}} \lesssim   1\,.
\end{equation}

Since $h^4 4^r \leq (\log z)^{4+C_\delta \log 4}$ and $\ell \geq z^4$, it follows that, in each of the limits \eqref{eq-second-moment-regime-1},
\eqref{eq-second-moment-regime-2} and (\ref{newequation79}), the ratio of the double sum to
$\E\Lambda_{N,z}$  goes to $0$ as first $N\rightarrow\infty$ and then  $z\rightarrow\infty$.
This shows that the sum in (\ref{newequation73}) is small in comparison with the preceding term for large
$z$, after letting $N\rightarrow\infty$, and completes the proof of the lemma.
\end{proof}

\begin{proof}[Proof of Proposition~\ref{prop-first-moment-is-king}]
We first show that, for appropriate $C$,
\begin{equation}
        \label{eq-sep162014}
        \P(G_N^{\mathrm{up}}(z)) 
        \leq C \mathrm{e}^{-\sqrt{2\pi}z}\,.
\end{equation}

To show (\ref{eq-sep162014}), let $\bar S_{v,N,r}^{\mathrm{up}}$ be defined
as $S_{v,N,r}^{\mathrm{up}}$ was defined,  using the
BRW construction rather than MBRW:
$$\bar S_{v,N,r}^{\mathrm{up}}=\sum_{j=\ell}^{n-r} \bar \phi_{N,j,\mathcal{BD}_j(v)}$$
(compare with \eqref{eq-def-BRW}). Using $\bar S_{v,N,r}^{\mathrm{up}}$,  
define 
$\bar X_{v,N}^{\mathrm{up}}(t)$
analogously to 
$X_{v,N}^{\mathrm{up}}(t)$.
Note that, for any $t\in [0,n_{v,N}]$ and $s\in [0,n_{u,N}]$,
$$\E 
\bar S_{v,N,r}^{\mathrm{up}}(t) \bar S_{u,N,r}^{\mathrm{up}}(s)=
\E\bar S_{v,N,r}^{\mathrm{up}}(t\wedge s) \bar S_{u,N,r}^{\mathrm{up}}(t\wedge
s)\leq
\E  S_{v,N,r}^{\mathrm{up}}(t)  S_{u,N,r}^{\mathrm{up}}(s)+C',$$
where $C'$ is a universal constant following from an explicit computation 
of the covariances for BRW and MBRW.
Define 
$\bar G_N^{\mathrm{up}}(z)$
analogously to 
$G_N^{\mathrm{up}}(z)$, using $N_v+\bar X_{v,N}^{\mathrm{up}}(t)$ 
instead of 
$\bar X_{v,N}^{\mathrm{up}}(t)$, where $\{N_v\}$ are i.i.d. standard
Gaussians.
Then, by Lemma \ref{lem-slepian}
(after rescaling space),
$$\P(G_N^{\mathrm{up}}(z))\leq
\P(\bar G_N^{\mathrm{up}}(z))\,.$$
The probability on the right hand side of this display is then 
dominated by the probability that
the event in Lemma \ref{lem-a-priori} occured  in one of 
$2^{2r}$ independent BRW's of depth $n+ \sqrt{C'}$. The conclusion
\eqref{eq-sep162014} follows.
Note that the slackness factor $z^{1/20}$ has been employed to kill
both the factor $2^{2r}$ (using $r= C_\delta \log h\leq
C_\delta \log\log z$) and
the prefactor $z$ in Lemma \ref{lem-a-priori}
(compare the definitions in \eqref{eq-big-definition} and \eqref{eq-def-G-N-prelim}).

Combining \eqref{eq-sep162014} with
Lemma \ref{lem-Gamma-Lambda}  and the trivial
estimate
$$\P(G_{N}^{\mathrm{up}}(z))+\E\Gamma_{N,z}^{\mathrm{up}}\geq
\P(\max_{v\in \tilde V_N} \eta_{v,N}^\mathrm{up}\geq m_N+z)\,,$$
the upper bound
$$\limsup_{z\to \infty} \limsup_{N\to \infty}
\frac{\P(\max_{v\in \tilde V_N} \eta_{v,N}^{\mathrm{up}} \geq m_N +z )}
{\E \Lambda_{N, z}^{\mathrm{up}}} \le 1$$
follows.
The lower bound
$$\liminf_{z\to \infty} \liminf_{N\to \infty} \frac{\P(\max_{v\in \tilde V_N}
\eta_{v,N}^{\mathrm{lw}} \geq m_N +z )}
{\E \Lambda_{N, z}^{\mathrm{lw}}}\geq 1$$follows from
%
%
Lemma~\ref{lem-second-moment} and
\begin{align*}
&\P(\max_{v\in \tilde V_N} \eta_{v,N}^\mathrm{lw}\geq m_N+z)
\geq
\P(\bigcup_{v\in \Xi_N} E_{v,N}^{\mathrm{lw}}(z))
\geq \frac{(\E\Lambda_{N,z}^\mathrm{lw})^2}
{\E(\Lambda_{N,z}^\mathrm{lw})^2}
\,.
\end{align*}
The other statements follow from
\eqref{eq-eta-up-lw} and \eqref{eq-morningcoffee}.
\end{proof}

For future reference, we note here that the same proof as for Lemma
\ref{lem-Gamma-Lambda} also implies that,
for any box $A\subset [\delta,1-\delta]^2$,
\begin{equation}
  \label{eq-290113a}
\liminf_{z\rightarrow\infty}\liminf_{N\rightarrow\infty}
 \frac{\E \Lambda_{N,z}^{\mathrm{lw}}(A)}  {ze^{-\sqrt{2\pi} z}}
  \gtrsim |A|\,,
\end{equation}
where $|A|$ denotes the area of $A$.

\subsection{Asymptotics for the enumeration of large clusters and completion
of the proof
of Proposition \ref{prop-limiting-tail-gff-delta}.}
This subsection is devoted to demonstrating Proposition \ref{prop-asymptotic-first-moment},
which gives the asymptotic behavior of $\E\Lambda_{N, z}^{\mathrm{lw}}$ for large $N$ and $z$.
\begin{prop}\label{prop-asymptotic-first-moment}
There exists a constant $\alpha^*_\delta>0$ such that
$$\lim_{z\to \infty}\limsup_{N\to \infty}
\frac{\E \Lambda^{\mathrm{lw}}_{N, z}}{\alpha^*_\delta
z \mathrm{e}^{-\sqrt{2\pi}z}} = \lim_{z\to \infty}\liminf_{N\to \infty}
\frac{\E \Lambda^{\mathrm{lw}}_{N, z}}{\alpha^*_\delta z \mathrm{e}^{-\sqrt{2\pi}z}} = 1\,.$$
Furthermore, there exist continuous functions
$\psi_\delta: [\delta, 1-\delta]^2 \mapsto (0, \infty)$, with
$\int_{[\delta, 1-\delta]^2} \psi_\delta(x) dx = 1$, and
a continuous function $\psi: (0, 1)^2 \mapsto (0, \infty)$
such that
$\psi_\delta(x) \to \psi(x)$ uniformly in $x$ on closed sets, as $\delta \searrow 0$, and such that,
for any open box $A\subseteq [\delta, 1-\delta]^2$,
$$\lim_{z\to \infty}\limsup_{N\to \infty}
\frac{\E \Lambda^{\mathrm{lw}}_{N, z}(A)}{\alpha^*_\delta
z \mathrm{e}^{-\sqrt{2\pi}z}} = \lim_{z\to \infty}
\liminf_{N\to \infty} \frac{\E \Lambda^{\mathrm{lw}}_{N, z}(A)}{\alpha^*_\delta
z \mathrm{e}^{-\sqrt{2\pi}z}} = \int_{A} \psi_\delta(x) dx\,.$$
\end{prop}

\noindent Together, Propositions~\ref{prop-gff-first-moment-dictates}
and \ref{prop-asymptotic-first-moment} imply
Proposition~\ref{prop-limiting-tail-gff-delta} for open boxes $A\subseteq [\delta, 1-\delta]^2$.
{Using 
Lemma \ref{lem-prelim-tail},
the statement extends to open sets  $A\subset [\delta, 1-\delta]^2$.}

To simplify
notation, we drop the $\mathrm{lw}$ superscript in the rest of the subsection.
For $v\in \Xi_{N}$, let $\nu_{v, N}(\cdot)$ be the
density function (of a subprobability measure on $\R$)
such that, for all $I\subseteq \mathbb{R}$,
$$\int_I\nu_{v, N}(y) dy = \P(X_{v,N}(t) \leq z + \frac{m_N}{n}t \mbox{ for all } 0\leq t\leq n_{v,N}; X_{v,N}(n_{v,N}) - (n-\ell) m_N/n \in I)\,.$$
Clearly, by (\ref{eqnewMB3}),
$$\P(E_{v, N}(z)) = \int_{-\infty}^z\nu_{v, N}(y) 
\P(\max_{u\in \tilde B_{v,N}}Y_{u,N}\geq \ell m_N/n + z - y) dy\,.$$
Recall the variables $\ell,\tilde \ell$ defined at the beginning
of the section and, for a given interval $J$, define
\begin{equation}\label{eq-def-lambda}
\lambda_{v, N, z , J} = \int_{J}\nu_{v, N}(y) 
\P(\max_{u\in \tilde B_{v,N}} Y_{u,N}\geq \ell m_N/n + z - y) dy\,.
\end{equation}

Set $J_{\ell} = [-\ell, -\ell^{2/5}]$.  The following estimate shows that the main contribution to
$\E\Lambda_{N,z}(A)$ is from values $y\in J_{\ell}$, as in (\ref{eq-def-lambda}).
(The choice of the exponent $2/5$ here is somewhat arbitrary; only $0<2/5<1/2$ is used.)
\begin{lemma}\label{lem-Lambda-J-k}
For any
box $A\subseteq [\delta, 1-\delta]^2$ and any sequences $x_{v,N}$ such that 
$|x_{v,N}| \lesssim \ell^{2/5}$,
$$\lim_{z\to \infty} \liminf_{N\to \infty} \frac{\sum_{v\in \Xi_N \cap N A} \lambda_{v, N, z, x_{v, N} +J_\ell}}{\E \Lambda_{N, z}(A)} = 1\,.$$
\end{lemma}
\begin{proof}
Note that, by containment, the above ratio is always at most $1$.  
We prove the lemma for the case when $x_{v,N} = 0$; the general case follows in the same manner.
Application of the reflection principle \eqref{eq-reflection-principle} to the
Brownian motion with drift,
$\bar X_{v,N}(\cdot)=X_{v,N}(\cdot)-m_Nt/n$, together with the change of measure
that removes the drift $m_N t/n$,  implies that
$$\nu_{v,N}(y)\lesssim e^{-\sqrt{2\pi} y}4^{-n_{v,N}} z\ell $$
for $y\le -\ell$, over the given range $z\in (0,\ell)$ (which implies 
$\ell+z\asymp
\ell$).
Together with Lemma~\ref{lem-decorelate-tilde-B}, this implies
the  crude bound
$$\int_{-\infty}^{-\ell}  \nu_{v, N}(y) \P(\max_{u\in \tilde B_{v, N}} Y_{u,N}\geq \ell m_N/n + z - y) dy \lesssim 4^{-n_{v,N}} \mathrm{e}^{-C^{-1} \ell }$$
for an absolute constant $C>0$. Similarly,
for $y\leq z$ (and therefore, for $z-y\geq 0$),
application of the reflection principle and
Lemma~\ref{lem-decorelate-tilde-B} again implies that
$$\int_{-\ell^{2/5}}^{z}  \nu_{v, N}(y) \P(\max_{u\in \tilde B_{v, N}} Y_{u,N} \geq \ell m_N/n + z - y) dy \lesssim 4^{-n_{v,N}} \ell^{-3/10} (\log z)^C z \mathrm{e}^{-\sqrt{2\pi} z}\,.$$
Together with
\eqref{eq-290113a},
this implies that $\E \Lambda_{N, z}(A) - \sum_{v\in \Xi_N \cap NA} \lambda_{v, N, z, J_\ell} \lesssim \ell^{-3/10} (\log z)^C \E \Lambda_{N, z}(A)$, as needed.
\end{proof}
\begin{lemma}\label{lem-Lambda-z}
There exists $\Lambda^*_z>0$ depending only on $z$
such that, for all functions $L$ and $\tilde L$ of $z$
satisfying \eqref{eq-L-tilde-L},
\begin{equation}
        \label{eq-092114a}
        \lim_{z\to \infty}\limsup_{N\to \infty} \frac{\E \Lambda_{N, z}}{\Lambda^*_z} = \lim_{z\to \infty}\liminf_{N\to \infty} \frac{\E \Lambda_{N, z}}{\Lambda^*_z} = 1\,.
\end{equation}
Furthermore, there exist continuous functions $\psi_\delta
: [\delta, 1-\delta]^2 \mapsto (0, \infty)$, with
$\int_{[\delta, 1-\delta]^2} \psi_\delta(x) dx =  1$, and
a continuous function $\psi: (0, 1)^2 \mapsto (0, \infty)$
such that
$\psi_\delta \to \psi$ uniformly in $x$ on closed subsets of $(0,1)^2$, as $\delta \searrow 0$,
and such that, for any box $A\subseteq [\delta, 1-\delta]^2$,
\begin{equation}
        \label{eq-092114b}
\lim_{z\to \infty}\limsup_{N\to \infty} \frac{\E \Lambda_{N, z}(A)}{\Lambda^*_z} = \lim_{z\to \infty}\liminf_{N\to \infty} \frac{\E \Lambda_{N, z}(A)}{\Lambda^*_z} =\int_{A} \psi_\delta(x) dx \,.
\end{equation}
\end{lemma}
\begin{proof}
        Applying Lemma \ref{lem-prelim-tail} and 
Proposition~\ref{prop-first-moment-is-king}, the limiting behavior of 
$\E\Lambda_{N,z}$ is the same for different choices of
$\tilde L(z)$. Similarly,  employing the analog of 
Proposition~\ref{prop-first-moment-is-king} with $\E\Lambda_{N,z}(A)$ instead
of $\E\Lambda_{N,z}$, the limiting behavior of $\E\Lambda_{N,z}(A)$ 
is also the same for different choices of 
$\tilde L(z)$.
So, it suffices to consider the case when
$\tilde L(z) = 2^{z^4}$.

Write
$x_{v,N} = m_N(1 - n_{v,N}/n) - \gamma \sqrt{2\pi} \ell$.
It follows from \eqref{eq-friday3}, \eqref{eq-def-X-up-lw} and
\eqref{eq-a-N-v} that $x_{v,N} = O(1)$.
(Recall that $\gamma=\sqrt{2\log 2/\pi}$ and
$\var(S_{v,N,r}^{\mathrm{lw}})=\gamma^2(n-\ell-r+1)$.)
For $\hat v\in [-L/2, L/2]^2$,
set $\Xi_{\hat v, N} = \{v\in \Xi_N: v - c_{B_v} = \hat v\}$, where
$c_{B_v}$ is the center of the box $B_v\in\mathcal{B}_N$ containing $v$.
(In words, 
$\Xi_{\hat v, N}$ consists of the centers of the boxes in
$\tilde {\mathcal{B}}_N$ that are translated by $\hat v$ from the centers
of the 
corresponding larger boxes
in $\mathcal{B}_N$. When $\Xi_{\hat v,N}$ is not empty,
$|\Xi_{\hat v,N}|$ is approximately
$(N/L)^2$.
Note that there are about $h^2 = (L/\tilde L)^2$ values
of $\hat v$ with  $\Xi_{\hat v, N} \neq  \emptyset$.)

Define
\begin{align*}
\Lambda_{\hat v, N, z, J_\ell}& = \sum_{v\in \Xi_{\hat v, N}}  \int_{J_\ell - x_{v,N}}\nu_{v, N}(y) \P(\max_{u\in \tilde B_v} Y_{u,N}\geq \ell m_N/n + z - y) dy\,,\\
\Lambda_{\hat v, N, z,  J_\ell}(A)& = \sum_{v\in \Xi_{\hat v, N}\cap A N}  \int_{J_\ell - x_{v,N}}\nu_{v, N}(y) \P(\max_{u\in \tilde B_v} Y_{u,N}\geq \ell m_N/n + z - y) dy\,,
\end{align*}
where $J_\ell$ is as in Lemma
\ref{lem-Lambda-J-k}.
We will show that, for arbitrary
$\hat v\in [-L/2, L/2]^2\cap \Z^2$ with
$\Xi_{\hat v, N} \neq \emptyset$, that
there exists $\Lambda^*_{\hat v, z}$ satisfying
\begin{eqnarray}
\limsup_{N\to \infty} \frac{\E \Lambda_{\hat v, N, z, J_\ell}}
{\Lambda^*_{\hat v, z}} &=
& 1+O(z^{-1})=
\liminf_{N\to \infty}
\frac{\E \Lambda_{\hat v, N, z,  J_\ell}}{\Lambda^*_{\hat v, z}}\,,
\label{eq-to-show-1}\\
\limsup_{N\to \infty}
\frac{\E \Lambda_{\hat v, N, z, J_\ell}(A)}{\Lambda^*_{\hat v, z}} &=
& (1+O(z^{-1}))
\int_{A} \psi_{\delta}(x) dx
=\liminf_{N\to \infty} \frac{\E \Lambda_{\hat v, N, z, J_\ell}(A)}
{\Lambda^*_{\hat v, z}}\,,
\label{eq-to-show-2}
\end{eqnarray}
where $\psi_{\delta}: [\delta, 1-\delta]^2\mapsto (0, \infty)$ is a continuous function with $\int_{[\delta, 1-\delta]^2} \psi_{\delta}(x) = 1$, 
and $\psi_{\delta}$ converges to a continuous function as $\delta \searrow 0$. 
From Lemma \ref{lem-Lambda-J-k}, \eqref{eq-092114a} and
\eqref{eq-092114b} will follow from 
\eqref{eq-to-show-1}
and
\eqref{eq-to-show-2}
after summing over the $h^2$ values of $\hat v$.
Note that, crucially, the function $\psi_{\delta}$ is required to be \emph{independent}
of the choice of $(\hat v, z, L, \tilde L)$.
It is clear that, for all $v\in \Xi_{\hat v, N}$, the distribution of
$M_{\hat v, v}:=\max_{u\in \tilde B_v} Y_{u,N}$ depends only on 
$\hat v$, $L$, $\tilde L$.

By \eqref{eq-change-of-measure} and the
reflection principle,
\begin{equation}\label{eq-mu-N-v}
  \nu_{v, N}(y + x_{v,N}) =  4^{-n_{v,N}} \mathrm{e}^{-\sqrt{2\pi} (y+x_{v,N})} \frac{z(z-y-x_{v,N})}{\sqrt{2\pi} \gamma} (1+ O(\ell^3/n)) \one_{\{x_{v,N} + y\leq z\}}\,.
\end{equation}
Therefore,
\begin{align*}
\Lambda_{\hat v, N, z, J_\ell} &= \sum_{v\in \Xi_{\hat v, N}} \int_{J_\ell}\nu_{v, N}(y + x_{v,N}) \P(M_{\hat v, v} \geq \sqrt{2\pi} \gamma \ell + z - y) dy\\
&= 
(1+ O(\ell^3/n))
\int_{J_\ell} \sum_{v\in \Xi_{\hat v, N}} 4^{-n_{v,N}} \frac{z(z-y-x_{v,N})}{\sqrt{2\pi} \gamma \mathrm{e}^{\sqrt{2\pi} (y+x_{v,N})} }   \P(M_{\hat v, v} \geq \sqrt{2\pi} \gamma \ell + z - y) dy 
\end{align*}
and
$$\Lambda_{\hat v, N, z, J_\ell} (A)  = 
(1+ O(\ell^3/n))
\int_{J_\ell} \sum_{v\in \Xi_{\hat v, N}\cap N A} 4^{-n_{v,N}} \frac{z(z-y-x_{v,N})}{\sqrt{2\pi} \gamma \mathrm{e}^{\sqrt{2\pi} (y+x_{v,N})} }   \P(M_{\hat v, v} \geq \sqrt{2\pi} \gamma \ell + z - y) dy \,.$$
Note that $x_{v,N}$ is a linear continuous function and $n_{v,N}$ is a quadratic continuous function of $a_{v, N}$
(as defined in \eqref{eq-def-X-up-lw}).
For any point $v^*$ in the unit square, denote by $ v^*_N$ the vertex in $\Xi_{\hat v, N}$ that is closest to $N v^*$.
By \eqref{eq-a-N-v-converge}, for any $v^*\in [\delta, 1-\delta]^2$, the limit
$$\phi_{\hat v}(v^*) =  \lim_{N\to \infty } 4^{n - \ell} 4^{-n_{v^*_N,N}} \mathrm{e}^{-\sqrt{2\pi} x_{v^*_N,N}} $$
exists (up to an $O(1/L)$ possible error) and $\phi_{\hat v}(v^*)$ is a continuous function on $[0, 1]^2$.
In fact, there exist constants $C,C_1,C_2$ such that
$$\phi_{\hat{v}}(v^*)=Ce^{-C_1g_1(v^*)-C_2 g_2(\hat{v})}\,.$$
Thus, 
\begin{equation}\label{eq-independent-hat-v}
\frac{\phi_{\hat v}(v^*)}{\phi_{\hat v}(u^*)} \mbox{ is a function depending only on $ (v^*, u^*)$}.
\end{equation}

By the bounded convergence theorem,
\begin{align*}
\limsup_{N\to \infty}\Lambda_{\hat v, N, z,  J_\ell} = (1+O(z^{-1}))
\left(\int_{J_\ell} \frac{z(z-y)}{\sqrt{2\pi} \gamma \mathrm{e}^{\sqrt{2\pi} y} }    \P(M_{\hat v, v} \geq \sqrt{2\pi} \gamma \ell + z - y) dy\right) 
h^2 \int_{[\delta, 1-\delta]^2} \phi_{\hat v}(x) dx\,,
\end{align*}
with a similar estimate holding when $\limsup$ is replaced by $\liminf$.
Since $\P( M_{\hat v, v} \geq \sqrt{2\pi} \gamma \ell + z - y)$
is a function of just $(\hat v, z, y)$, this
completes the proof of \eqref{eq-to-show-1}. Similarly,
\begin{align*}
\limsup_{N\to \infty}\Lambda_{\hat v, N, z, J_\ell} (A)= (1+O(z^{-1}))
\left(\int_{J_\ell} \frac{z(z-y)}{\sqrt{2\pi} \gamma \mathrm{e}^{\sqrt{2\pi} y} }    \P(M_{\hat v, v} \geq \sqrt{2\pi} \gamma \ell + z - y) dy \right) h^2 \int_{A} \phi_{\hat v}(x) dx\,,
\end{align*}
with a similar estimate holding when $\limsup$ is replaced by $\liminf$.
Setting $$\psi_{\delta, \hat v}(x) = \phi_{\hat v}(x)/\int_{[\delta, 1-\delta]^2} \phi_{\hat v}(x) dx\,,$$  it follows that $\psi_{\delta, \hat v}$ satisfies all of the desired properties. In particular,
by \eqref{eq-independent-hat-v}, the function
$\psi_{\delta, \hat v}$  is independent of $(\hat v, z, L, \tilde L)$. This completes the proof of \eqref{eq-to-show-2} and hence the proof of the lemma.
\end{proof}

We are now ready to prove
Proposition~\ref{prop-asymptotic-first-moment}.
\begin{proof}[Proof of Proposition~\ref{prop-asymptotic-first-moment}]
The second display
in Proposition~\ref{prop-asymptotic-first-moment} follows
directly from the first display
and the second display in Lemma~\ref{lem-Lambda-z}.
It therefore suffices to prove the first display
in Proposition~\ref{prop-asymptotic-first-moment}. To this end, consider $z_1 < z_2$, and set $\tilde L = 2^{z_2^4}$ and $h = \log z_1$.  For $v\in \Xi_N$ and $i=1,2$, recall  that
$$\lambda_{v, N,  z_i, z_i+J_\ell} = \int_{J_\ell + z_i}\nu_{v, N}(y) \P(\max_{u\in \tilde B_ {v, N}} Y_{u,N}\geq \ell m_N/n + z_i - y) dy\,.$$
By \eqref{eq-mu-N-v}, for any $y\in J_\ell$,
\begin{align*}
&\frac{\nu_ {v, N}(y + z_1) \P(\max_{u\in \tilde B_{ v, N}} Y_{u,N}\geq \ell m_N/n - y)}{\nu_{v, N}(y+z_2) \P(\max_{u\in \tilde B_{v, N}} Y_{u,N}\geq \ell m_N/n - y)}\\
& = \frac{\nu_{v, N}(y+z_1)}{\nu_{v, N}(y+z_2)} = 
(1+ O(\ell^3/n))
\frac{z_1 (z_1 - y)}{z_2 (z_2 - y)} \mathrm{e}^{-\sqrt{2\pi}(z_1-z_2)}
=
(1+ O(\ell^3/n))
\frac{z_1}{z_2} \mathrm{e}^{-\sqrt{2\pi}(z_1 - z_2)} (1 + z_2^{-3/5})\,.
\end{align*}
This implies that
$$\frac{\lambda_{v, N, z_1, z_1 + J_\ell}}{\lambda_{v, N, z_2, z_2 + J_\ell}} 
=(1+ O(\ell^3/n))
 \frac{z_1}{z_2} \mathrm{e}^{-\sqrt{2\pi}(z_1 - z_2)} (1 + z_2^{-3/5})\,.$$
Together with Lemma~\ref{lem-Lambda-J-k}, the above display implies that
$$\lim_{z_1, z_2\to \infty} \limsup_{N \to \infty} \frac{z_2e^{-\sqrt{2\pi}z_2}
\E
\Lambda_{N, z_1}}{z_1e^{-\sqrt{2\pi}z_1}\E
\Lambda_{N, z_2}} = \lim_{z_1, z_2\to \infty} \liminf_{N \to \infty} 
\frac{z_2e^{-\sqrt{2\pi}z_2} \E
\Lambda_{N, z_1}}{
        z_1e^{-\sqrt{2\pi}z_1}  \E\Lambda_{
N, z_2}}  =  1\,.$$
Along with Lemma~\ref{lem-Lambda-z}, this
completes the proof of the proposition.
\end{proof}


\section{A pair of approximations}
\label{sec-pairap}
The main results in this section are Propositions \ref{prop-Delta} and \ref{prop-local}.
Proposition \ref{prop-Delta} will be applied in Section \ref{sec-proof-of-theorem}, and allows
us to restrict our attention to the sets $V_N^{K,\delta}=\cup_i V_N^{K,\delta,i}$ when computing
the maximum of $\eta_{v,N}$.
\begin{prop}
  \label{prop-Delta}
  With notation as defined earlier,
  \begin{equation}
    \label{eq-smallerbox}
    \lim_{\delta\searrow 0}\limsup_{K\to\infty}\limsup_{N\to\infty}
    \P(\max_{v\in V_N^{K,\delta}} \eta_{v,N}\neq  \eta_N^*)=0.
  \end{equation}
\end{prop}
\begin{proof}
  Due to the tightness of the sequence of random variables $(\eta_N^*-m_N)$
  (see \cite{BZ10}), it suffices to show that, for any fixed $x\in \R$,
\begin{equation}
    \label{eq-dilul}
    \lim_{\delta\searrow0}\limsup_{K\to\infty}\limsup_{N\to\infty}
    \P( \max_{v\in \Delta_N^{K,\delta} }
    \eta_{v,N}-m_N\geq x)=0\,.
  \end{equation}
  The claim
  \eqref{eq-dilul} follows at once from  \eqref{eq-max-subset}.
  \end{proof}

Proposition \ref{prop-local} will be applied in conjunction with Proposition \ref{prop-Delta},
and implies that the local maxima of the GFF occur at the local maxima of the fine field, at least when
restricted to $V_K^{N,\delta}$.
\begin{prop}
  \label{prop-local}
  Let $z_i=z_i^{N,K,\delta} 
  {\in V_N^{K,\delta,i}}$ be such that
  $$ \max_{v\in V_N^{K,\delta,i}} X_v^f=X_{z_i}^f\,.$$
  Let $\bar z=\bar z(N,K,\delta) 
  {\in \{z_i: 1\leq i\leq K^2\}}$ be such that
  $$\max_i \eta_{z_i,N}=\eta_{\bar z,N}\,.$$
  Then, for any fixed $\epsilon>0$ and small enough $\delta>0$,
  \begin{equation}
    \label{eq-star0103a}
    \lim_{K\to\infty}
    \limsup_{N\to\infty}
    \P(\max_{v\in V_N^{K,\delta}}\eta_{v,N}\geq \eta_{\bar z,N}+\epsilon)
    =0\,.
  \end{equation}
  Furthermore,
  there exists a function $g:\N\to \R_+$, with
  $g(K)\to_{K\to\infty} \infty$, such that
  \begin{equation}
    \label{eq-star0103b}
    \lim_{K\to\infty}
    \limsup_{N\to\infty}
    \P(X^f_{\bar z}\leq m_{N/K}+g(K))=0\,.
  \end{equation}
\end{prop}

The proof of Proposition \ref{prop-local} occupies the remainder of the section.
\begin{proof}
 The strategy for the proof of (\ref{eq-star0103a}) is as follows.
Consider the event on the left hand side of \eqref{eq-star0103a},
 %
i.e., for the box $i$ where  
 $\eta_{v_i}:=
  \max_{v\in V_N^{K,\delta,i}} \eta_{v,N} = \max_{v\in V_N^{K,\delta}} \eta_{v,N}$, the event where $\eta_{v_i} - \eta_{z_i} \ge \epsilon$.
For appropriate $f(k)$, with $k = \log_2 K$, this event will not typically occur when 
$|v_i - z_i| \le f(k)$ due to the continuity of the field $X^c_{\cdot }$
whereas, on
$|v_i - z_i| > f(k)$, the event will not typically occur 
because \cite[Theorem 1.1]{DZ12}
that prohibits near-maxima from coexisting at intermediate distances.
(The argument actually requires consideration of all boxes $i$ where 
$\eta_{v_i}$ is not much smaller than $\max_{v\in V_N^{K,\delta}} \eta_{v,N}$\,.)


%
%

Turning to the actual proof of (\ref{eq-star0103a}), fix two constants $C,C'>0$ and
a function $f:\N\to \R_+$, with
  $f(k)\to_{k\to\infty}\infty$.
  Suppose that the event on the left hand side of \eqref{eq-star0103a} occurs.
  Then (keeping in mind the above description),
  one of the following events must occur:
  \begin{itemize}
    \item ${\cal A}_0:=  
    {\{}
  \max_{v\in V_N^{K,\delta}}\eta_{v,N} < m_N-C 
  {\}}$,
    \item ${\cal A}_1:= \{\max_{v\in V_N^{K,\delta}}
      \max_{w\in V_N^{K,\delta}: |v-w|\leq f(k)} |X_v^c-X_w^c|\geq
      \epsilon\}
      \,,$
    \item ${\cal A}_2:=\{
      \max_i   \max_{u,v\in V^{K,\delta,i}_N}
      (\eta_{u,N}+X_u^c-X_v^c)\geq m_N+C'\}\,,$
    \item ${\cal A}_3:=
      \{\exists i,v:
      N/K\geq d(v,z_i)> f(k),\,
      \eta_{v,N}\geq m_N-C, \eta_{z_i,N}\geq m_N-2C-C'\}.$
  \end{itemize}
 {To see this, we first claim that, on the event $({\cal A}_0 \cup {\cal A}_1 \cup {\cal A}_2 \cup {\cal A}_3)^c$, the maximizer $\tau$ for the field $\{\eta_{v, N}:v\in V_{N}^{K, \delta} \}$ is within distance $f(k)$ of $z_i$ (where we assumed $\tau \in V_{N}^{K, \delta, i}$). Otherwise, on $({\cal A}_0\cup {\cal A}_2)^c$,
  $$X_\tau^c - X_{z_i}^c \leq m_N + C' - (m_N - C) \leq C + C'$$
  and thus
  $$\eta_{z_i} \geq \eta_{\tau} - (X_\tau^c - X_{z_i}^c) \geq m_N - 2C - C'\,,$$
  which is not consistent with being in ${\cal A}^c_3$. Hence, $|\tau - z_i| \leq f(k)$.   But, on ${\cal A}^c_1$, this inequality and the event on the left hand side of \eqref{eq-star0103a} cannot simultaneously occur.  Consequently, at least one of the events ${\cal A}_i$ must occur when the event in \eqref{eq-star0103a} occurs, as claimed. }
  
Next, we will show that the limiting probability of each of these four events is small as 
$C,C' \rightarrow\infty$ appropriately, after $N\rightarrow\infty$.

To control ${\cal A}_0$, we employ an argument similar to that used in
the proof of \cite[Proposition 5.2]{BZ10}.
From Proposition \ref{prop-limiting-tail-gff} and \eqref{eq-max-subset}, one obtains, for some $\mu_0 < 1$, that
for large enough $C_0$ and $N$, and small 
enough $\delta >0$ (all not depending on $K$),
\begin{equation}
        \label{eq-18072014-a}
        \P(
  \max_{v\in V_N^{K,\delta}}\eta_{v,N} < m_N-C_0/2)\leq \mu_0\,.
  \end{equation}
  Decomposing $\eta_{v,N}$ as the sum of 
the coarse and fine fields $X^c_{v,N,2}$ and $X^f_{v,N,2}$
(with the latter field producing $4$ independent copies of the GFF in disjoint boxes of side length $N/2$), (\ref{eq-18072014-a}) implies that
  \begin{eqnarray*}
          \label{eq-18072014-b}
&&      \P(
  \max_{v\in V_N^{K,\delta}}\eta_{v,N} < m_N-C_0)\leq 
\left(  \P(
  \max_{v\in V_{N/2}^{K,\delta}}\eta_{v,N/2} < m_N-C_0/2)\right)^4+
  \P(\max_{v\in V_N^{K,\delta}} |X^c_{v,N,2}|\geq C_0/2)\\
  &\leq & \mu_0^4+\epsilon(C_0)\,,
  \end{eqnarray*} where $\epsilon(C_0)\to_{C_0\to\infty} 0$ by Lemma
  \ref{lem-ferniquecriterion}. Thus, 
  there exists $C_1>2C_0$ large enough so that,
  for large enough $N$,
  \begin{equation}
          \label{eq-18072014-c}
                \P(
  \max_{v\in V_N^{K,\delta}}\eta_{v,N} < m_N-C_1/2)\leq 
  2\mu_0^4=:\mu_1<\mu_0\,.
  \end{equation}
 Repeatedly applying this argument, with the analog of \eqref{eq-18072014-c} at each step replacing the analog of
  \eqref{eq-18072014-a}, one concludes that
  \begin{equation}
          \label{eq-18072014-d}
          \lim_{C\to\infty}\limsup_{K\to\infty}\limsup_{N\to\infty}     \P(
  \max_{v\in V_N^{K,\delta}}\eta_{v,N} < m_N-C)=0\,.
  \end{equation}

  We now consider ${\cal A}_i, i=1,2,3$.
  By a union bound and the upper bound in
  Lemma \ref{lem-coarsecov}, with
${\cal N}$ denoting a standard Gaussian random variable,
\begin{equation}
  \label{eq-A1}
  \P({\cal A}_1)\lesssim N^2 f(k)^2\P(\sqrt{K c_{\delta }f(k)/N} {\cal N}>\epsilon)
  \leq N^2  f(k)^2 e^{-\epsilon^2 N/2Kc_\delta f(k)}\to_{N\to\infty} 0\,.
\end{equation}
%
On the other hand, by \cite[Theorem 1.1]{DZ12}, for any fixed $C,C'$,
$$ \lim_{K\to\infty}\limsup_{N\to\infty} \P({\cal A}_3)=0\,.
$$
This limit employs $f(k)\to_{k\to\infty} \infty$.

We will show below that, for some constant $C_\delta$ not depending on $N,K,C$,
\begin{equation}
  \label{eq-010313d}
  \P({\cal A}_2)\leq \frac{C_\delta}{C'-C_\delta}\leq \frac{2C_\delta}{C'}\,,
\end{equation}
for large enough $C'$.
Taking $N\to\infty$, followed by $K\to\infty$,
then $C\to\infty$ and then $C'\to\infty$, implies $P({\cal A}_2) \rightarrow 0.$
 The above limits on $P({\cal A}_i)$, $i=0,1,2,3$, together imply \eqref{eq-star0103a}.

In order to estimate $\P({\cal A}_2)$, we require a couple of lemmas. 
Write $Y_{u,v}=Y_{u,v,N}:=
\eta_{u,N}+X_{u,N}^c-X_{v,N}^c$, and recall that $u,v$ belong to the same
  box $V_K^{N,\delta,i}$. Write $V_{N,K,\delta}^{\times 2}:=
  \{(u,v):u,v\in V_K^{N,\delta,i} \, \mbox{\rm for some } i\}$.
  The proof of the first lemma is
  a straightforward application of the upper bound in 
  Lemma \ref{lem-coarsecov}.
  \begin{lemma}
    \label{lem-covY}
    There exists a constant $c_1$ independent of $K, N$ such that,
    for $(u,v),(u',v')\in V_{N,K,\delta}^{\times 2}$, 
    \begin{equation}
    \label{eq-covY}
    \E(Y_{u,v}-Y_{u',v'})^2
    \leq \E(\eta_u-\eta_{u'})^2+
    c_1\left( {\bf 1}_{u=u'}\left(\frac{|v-v'|}{N/K}\right)^2+{\bf 1}_{u\neq u'}\right)\,.
  \end{equation}
  \end{lemma}
\begin{proof}
	Using the decomposition $\eta_{u,N}=X_{u,N}^f+
	X_{u,N}^c$ and the independence 
	of $\{X_{u,N}^f\}$ and $\{X_{u,N}^c\}$, 
	we have  
\begin{align}
	\label{e-satmorning}
	\E(Y_{u,v}-Y_{u',v'})^2 &= 
	\E\left((X_{u,N}^f-X_{u',N}^f)+  
	2\left(X^c_{u,N} - X^c_{u',N}\right) - \left(X^c_{v,N} - 
	X^c_{v',N}\right)\right)^2 \nonumber\\
	&= \E (\eta_{u,N}-\eta_{u',N})^2\\
	&+ 3 \E(X_{u,N}^c-X_{u',N}^c)^2+
	\E(X_{v,N}^c-X_{v',N}^c)^2-
	4 \E\left( (X_{v,N}^c-X_{v',N}^c)(X_{u,N}^c-X_{u',N}^c) \right)\,.\nonumber
\end{align}
 If $u=u'$, then 
    \eqref{eq-covY} follows by an application of Lemma 
  \ref{lem-coarsecov}. 
 On the other hand, if $u\neq u'$, but $u,u'$ belong to the
	  same box $V_K^{N,\delta,i}$, then again \eqref{eq-covY}
	  follows from Lemma \ref{lem-coarsecov}.
  
If $u,u'$ do not belong to the
  same box $V_K^{N,\delta,i}$,
  we can rewrite
  \eqref{e-satmorning} as
  \begin{equation}
	  \label{eq-kavita}
	  \E(Y_{u,v}-Y_{u',v'})^2  
	 =\E (\eta_{u,N}-\eta_{u',N})^2
	 + \E\left(\Delta(\Delta-2(X^c_{u,N}-X^c_{u',N}))\right), 
 \end{equation}
	where 
	$\Delta=X_{v,N}^c-X_{u,N}^c-X_{v',N}^c+X_{u',N}^c$. 
	By Lemma \ref{lem-coarsecov}, $\E\Delta^2\leq c$, for some $c$, and hence 
	  $$\E(Y_{u,v}-Y_{u',v'})^2  
	 \leq \E (\eta_{u,N}-\eta_{u',N})^2 +c-
	 2\E(\Delta(X^c_{u,N}-X^c_{u',N}))\,.$$
Also, by (\ref{eq-100113anewadd}) of Lemma \ref{lem-coarsecov}, 
$\E(X^c_{u,N} - X^c_{v,N})^2 \ge 2|\E(X^c_{u,N})^2 - \E 
X^c_{u,N}X^c_{v,N}| - c'$
for some $c'$, with the analogous inequality holding with $(u',v')$ in place of
$(u,v)$;
 it follows that
	  $$\E(Y_{u,v}-Y_{u',v'})^2  
	 \leq \E (\eta_{u,N}-\eta_{u',N})^2 +c''+
	2 |\E(2X^c_{u,N}X^c_{u',N}-X^c_{v,N}X^c_{u',N}-X^c_{v',N}X^c_{u,N})|\,,$$
	for appropriate $c''$. 
Since $u$ and $u'$ are in distinct boxes $V^{N,\delta,i}_K$, the previous
inequality continues to hold if $X^c_{\cdot,N}$ is replaced by $\eta_{\cdot,N}$
everywhere inside the absolute values on the right hand side.  Since $u$ and $v$ 
(respectively, $u'$ and $v'$) belong
to the same box, \eqref{eq-covY} follows from this and Lemma \ref{lem-covariance}.
\end{proof}

 We next construct a MBRW $\mbrw_u$ in a box of size $N$, using only the
  top $k$ levels. That is,
  with
   $\BB_j^N$ denoting the collection of subsets of $\Z^2$ consisting of
squares of side length $2^j$ with lower left corner in $V_N$,
 and with
  $\{b_{j,B}\}_{j\geq 0, B\in \BB_j^N}$ denoting  an i.i.d.
family of centered Gaussian random variables of variance $2^{-2j}$,
independent of $\eta_{\cdot,N}$,
let
$$\mbrw_{u,N}=\sum_{j=n-k}^n \sum_{B\in \BB_j(u)} b_{j,B}^N\,.$$
(Here, $\BB_j(u)$ denotes those elements of $\BB_j^N$ that contain $u$.) Let $\{\mbrw^{i}_N\}_{i}$ denote an i.i.d. family of copies of
$\mbrw_N$ and,
for $u,v\in
V^N_{K,\delta,i}$, set
$Z_{u,v}^N=\mbrw^i_{u,N}-\mbrw^i_{v,N}$.
Thus, $Z^N$ is a Gaussian field with index set
$V_{N,K,\delta}^{\times 2}$.
Similarly, introduce the field $\bar Z_{u,v}$ so that,
for $u\in V_N^{K,\delta}$,
the random vectors
$\{\bar Z_{u,v}\}_{v\in B_{N/K,\delta}(u)}$
have the same law as
$\{Z_{u,v}\}_{v\in B_{N/K,\delta}(u)}$, but these random vectors
are independent for
different $u$.
(Here,
$B_{N/K,\delta}(u)$ denotes the box $V_N^{K,\delta,i}$ to which $u$ belongs.)

Fix a constant $C_2>0$ and set $\bar Y_{u,v}=\eta_{u,N}+ C_2( \bar 
Z_{u,v}^N+ N_u)$ where $\{N_u\}_u$ are i.i.d. standard centered 
Gaussian random variables.
It follows immediately from Lemma \ref{lem-covY} and a direct computation
(using the fact that, for $u\neq u'$, the vectors  $\{\bar Z_{u,v}^N\}_v$ and
$\{\bar Z_{u',v}\}_v$ are independent)
that there is a choice of $C_2$ such that, for any
    $(u,v),(u',v')\in V_{N,K,\delta}^{\times 2}$,
\begin{equation}
  \label{eq-010313a}
  \E(\bar Y_{u,v}-\bar Y_{u',v'})^2 \geq
  \E( Y_{u,v}- Y_{u',v'})^2.
\end{equation}
In particular, with $Y^*_N:=\max_{(u,v)\in V_{N,K,\delta}^{\times 2}} Y_{u,v}$ and
$\bar Y^*_N:=\max_{(u,v)\in V_{N,K,\delta}^{\times 2}} \bar Y_{u,v}$,
by lemma \ref{lem-sudakov-fernique},
\begin{equation}
  \label{eq-010313b}
  \E Y^*_N\leq \E\bar Y^*_N\,.
\end{equation}

We make one more comparison. Let $\brw_{\cdot,N}$ be the 4-ary
BRW  indexed by $V_N$, chosen independently
of $Z^N_{\cdot,\cdot}$,
and set $\tilde Y_{u,v}=\brw_{u,N}+C_2 (\bar Z_{u,v}^N+N_u)$
and $\tilde Y^*_N=\max_{(u,v)\in V_{N,K,\delta}^{\times 2}}\tilde Y_{u,v}$.
By the domination of the correlation distance of
GFF by that of the BRW (see \cite{BZ10}),
Lemma~\ref{lem-sudakov-fernique}
and \eqref{eq-010313b}, it follows that
\begin{equation}
  \label{eq-010313c}
  \E Y^*_N\leq \E\bar Y^*_N\leq \E\tilde Y^*_N\,.
\end{equation}

We need the following lemma, whose proof is postponed until later in this section.
\begin{lemma}
  \label{lem-step4ofwriteup}
  There exists a constant $c_\delta>0$, not depending on $K,N$,
so that
  \begin{equation}
    \E \tilde Y^*_N\leq m_N+c_\delta\,.
  \end{equation}
\end{lemma}
By Lemma \ref{lem-step4ofwriteup} and \eqref{eq-010313c},
$\E Y^*_N\leq m_N+c_\delta$. On the other hand,
by definition, $Y^*_N\geq X^*_N := \max_{v\in V_{N}^{K,\delta}}\eta_{v,N}$.
Together with $\E(X^*_N-m_N)_-\leq C'' $, which follows from
 the same argument as for $\E(\eta_N^*-m_N)_-< C''$ (see
 \cite[Pages 12--15]{BZ10}),
this implies that
$\E|Y^*_N-m_N|\leq C_\delta$ for some constant $C_\delta$ not depending on
$N,K$. This demonstrates \eqref{eq-010313d} and hence \eqref{eq-star0103a}.

To demonstrate \eqref{eq-star0103b}, fix $\epsilon'>0$ and, using Proposition \ref{prop-Delta}, recall that
$$\lim_{K\to\infty}  \liminf_{N\to\infty} \P(\max_{z\in V_N^{K,\delta}} \eta_{z,N}\geq m_N-\epsilon' \log k)=1\,.$$
By \eqref{eq-star0103a}, this implies
$$\lim_{K\to\infty}  \liminf_{N\to\infty} \P( \eta_{\bar z,N}\geq m_N-\epsilon' \log k)=1\,.$$
Therefore, since $m_N-m_{N/K} = c^*k + C_N(K)$,
with $c^*=2\log 2\cdot \sqrt{2/\pi}$
and $C_N(K)\to_{N\to\infty} 0$, for any fixed $K$,
\eqref{eq-star0103b} will follow from
$$\lim_{K\to\infty}\limsup_{N\to\infty} \P(\max_{i=1}^{K^2} X_{z_i}^c\geq c^*k-\epsilon' \log k-g(K))=0\,.$$
But, for appropriate $g(\cdot)$, the latter is a consequence of a simple union bound:
Setting $g(K)=\alpha \log k$, with $\alpha > 0$, and $\alpha' = \alpha + \epsilon'$,
$$\P(\max_{i=1}^{K^2}
X_{z_i}^c\geq c^*k-\alpha' \log k))
  \leq
  \sum_{i=1}^{K^2} \P(X_{z_i}^c\geq c^*k-\alpha' \log k)\,.$$
  Applying Lemma \ref{lem-covariance} together with the analog of \eqref{eq-110113g},
the mean zero normal $X_{z_i}^c$
  has variance bounded above by
  $(c^*)^2k/4\log 2 + c'$, for appropriate $c'$.
So, the
  right hand side of the last display is bounded above by
  $$C K^2  \frac{\mathrm{e}^{- 2(\log 2)(c^*k -\alpha' \log k)^2/(c^*)^2 k}}{k^{1/2}}
  \leq C \mathrm{e}^{((4\alpha' (\log 2)/c^*)-1/2)\log k}\,.$$
  Choosing $\alpha\in (0,c^*/(8\log 2))$ and $\epsilon'$ small enough
implies that the right hand side of this display $\rightarrow 0$ as $K\rightarrow\infty$, which
  completes the proof of \eqref{eq-star0103b}.
\end{proof}

We turn to the proof of Lemma \ref{lem-step4ofwriteup}.
\medskip

\noindent
\textit{Proof of Lemma \ref{lem-step4ofwriteup}.}
For $u\in V_N^{K,\delta}$, set
$\zeta_u=C_2 \max_{v\in B_{N/K,\delta}(u)}  \bar Z_{u,v}^N +N_u$.
Note that $\E |\bar Z_{u,v}-\bar Z_{u,v'}|^2\leq C_\delta |v-v'|/(N/K)$.
For $u\in V_N\setminus V_{N^{K,\delta}}$, set $\zeta_u=0$.
A direct application of Fernique's criterion (Lemma \ref{lem-ferniquecriterion},
with the box $B$ taken to be $B_{N/K,\delta}$)
shows that $\E \zeta_u\leq c_0=c_0(\delta)$.
On the other hand, since for any $u,v$, $\E  (\bar Z_{u,v}+N_u)^2\leq c_1=c_1(\delta)$,
we conclude by Lemma~\ref{lem-gaussian-concentration} that
$$ \P(\zeta_u\geq c_0+y)\leq 2\mathrm{e}^{-y^2/2c_1}\,,\quad y\geq 0\,.$$
An application of Lemma
  \ref{lem-gff-perturb} then completes the proof of Lemma
 \ref{lem-step4ofwriteup}.

\section{Proofs of Theorems \ref{limit-law} and \ref{explicit-limit-law}}
\label{sec-proof-of-theorem}
We first prove Theorem \ref{limit-law}; Theorem \ref{explicit-limit-law} will then
follow quickly.  The proof of Theorem \ref{limit-law} is based on 
a coupling of the independent random variables
$(Y_i^K,z_i^K)$ of Subsection \ref{subsec-limit} with the values and locations of the local
maxima of the fine field $X^f$. We begin with a construction
of this coupling, 
which relies heavily on
Proposition \ref{prop-jian}. We next prove a continuity property
of the coarse field and then employ these two steps to demonstrate
Theorem \ref{limit-law}.

For probability measures $\nu_1,\nu_2$ on $\R$,
we denote by
$d(\nu_1,\nu_2)$ the L\'{e}vy distance between $\nu_1$ and $\nu_2$,
i.e.,
$$ d(\nu_1,\nu_2)=\min \{\delta>0: \nu_1(B)\leq \nu_2(B^\delta)+\delta\quad
\mbox{\rm for all open sets $B$}\},$$
where $B^\delta=\{y: |x-y|<\delta \mbox{  for some }x\in B\}$.
With a slight abuse of notation, when $X$ and $Y$ are random variables
with laws $\mu_X$ and $\mu_Y$, respectively, we will also write
$d(X,Y)$ for $d(\mu_X,\mu_Y)$.
\subsection{The coupling construction}
\label{subsec-coupling}
We begin with a preliminary lemma.  Recall that $k = \log_2 K$.
\begin{lemma}
  \label{lem-upperbound}
  There exists a constant $C^* > 0$ so that, for all $K$,
  \begin{equation}
    \label{eq-200113a}
    \limsup_{N\to\infty} \P(\max_{v\in V_{N/K}}\eta_{v,N/K}\geq m_{N/K}+C^* k)
    \leq K^{-3}\,.
  \end{equation}
\end{lemma}
  \begin{proof} Apply
    \eqref{eq-301212b} with $A=(0,1)^2$. \end{proof}

Let $g(K)$ be as in
Proposition \ref{prop-local}, and
recall from the proof of the proposition that we can choose
$g(K)=\alpha \log k$ for an appropriate $\alpha > 0$.
Also, recall the variables $Y_i^K$ in \eqref{eq-indeptails}.
Set $\theta_K(x)= \mathrm{e}^{\sqrt{2\pi}(x + g(K))}/(g(K)+x)$, for $x\geq 0$, and set
 $\eta_{N/K,\delta}^*=\max_{v\in V_{N/K}^\delta}\eta_{v,N/K}$.
\begin{lemma}
  \label{lem-intervalsplit}
  There exist $\epsilon_K\rightarrow_{K\rightarrow\infty} 0$ and a sequence of numbers
  $\alpha_{1,K,N}<\alpha_{2,K,N}<\ldots< \alpha_{_{C^*k,K,N}}$
  satisfying
\begin{equation}
\label{equation97a}
|\alpha_{j,K,N}-(g(K)+j-1)|\leq \epsilon_K
\end{equation}
such that, for all $i$,
  \begin{equation}
    \label{eq-200113b}
    \theta_K(0)^{-1}
    \P(Y_i^K\in [j-1,j])=\P(\eta_{N/K}^*-m_{N/K}\in  [\alpha_{j,K,N},\alpha_{j+1,K,N}))
    \,,\text{ for } j=1,\ldots,C^*k\,.
  \end{equation}
\end{lemma}
  \begin{proof}
Setting $\beta=
    \beta_{N,K,\delta}=\theta_K(0) \P({\cal A}_{N,K,\delta})$,
    it follows from
    \eqref{eq-301212a} of Proposition \ref{prop-jian}
    that, for all $N$ large, $|\beta-\alpha^*m_\delta|\leq \delta_K$,
    with $\delta_{K}\rightarrow_{K\rightarrow\infty} 0$. 
    Using the uniform continuity of
    the function $1/\theta_K(\cdot)$ and \eqref{eq-301212b}, with
    $A=(0,1)^2$, the conclusion follows for an appropriate choice of
    $\epsilon_K$.
  \end{proof}

  We now construct the required coupling.
Choose the enumeration
   of $W^i$ in Section \ref{subsec-limit} so that
  $N{ W}^1\cap \Z^2=V_{N/K}$.
  Denote by $(\B,Y,z^{K,\delta})$ a copy of the random vector $(\B_1^K,Y_1^K,z_1^{K,\delta})$.
  Recall the random variable $v^*_\delta$ defined by 
  $\eta_{v^*_\delta,N/K}=\max_{v\in V_{N/K}^\delta} \eta_{N/K}(v)$.
  \begin{prop}
    \label{prop-coupling}
    There exists a sequence $\bar\epsilon_K=\bar\epsilon_K(\delta)\rightarrow_{K\rightarrow\infty} 0$ such that
    $(\bar \B_{N,K,\delta},\eta_{N/K,\delta}^*-m_{N/K}, v^*_\delta)$ and $(\B,Y,z^{K,\delta})$
    can be constructed on the same
    probability space, with $\B=\bar \B_{N,K,\delta}:=
    1_{\{\eta_{N/K,\delta}^* - m_{N/K}\geq
    \alpha_{1,K,N}\}}$ holding with probability $1$, and such that,
    on the event where
    $\eta_{N/K,\delta}^*-m_{N/K}
    \leq \alpha_{C^*k,K,N}$,
    \begin{equation}
      \label{eq-210113d}
      \B|g(K)+Y-(\eta_{N/K,\delta}^*-m_{N/K})|+K|z^{K,\delta}
      -v^*_\delta/N|\leq \bar\epsilon_K\,.
    \end{equation}
  \end{prop}
In words, Proposition \ref{prop-coupling} states that the two processes can be coupled so that,
on the
rare
set where $\B = 1$, $\eta_{N/K,\delta}^*-m_{N/K}$ is always closely approximated
by $g(K) + Y$, and $v^*_\delta/N$ is closely approximated by $z^{K,\delta}$.
  \begin{proof}
    Employing Lemma
    \ref{lem-intervalsplit}, there is a piecewise linear map $\ell:[\alpha_{1,K,N},\alpha_{C^*k,K,N})
    \mapsto [g(K),g(K)+C^*k)$,
with $\ell (\alpha_{j,K,N}) = g(K) + j-1$, for $j=1,\ldots,C^*k$, and
having Lipschitz constant contained in
    $(1-\epsilon_K,1+\epsilon_K)$ and
     satisfying $\epsilon_K\rightarrow_{K\rightarrow\infty} 0$,
     so that
    $$\P(\ell(\eta^*_{N/K,\delta}-m_{N/K}) - g(K) \in [j-1,j))=
    \P( Y\in [j-1,j))\P(\B=1)\,.$$
  Because of (\ref{equation97a}), it suffices to prove \eqref{eq-210113d} with
    $\ell(\eta^*_{N/K,\delta}-m_{N/K})$ in place of $\eta^*_{N/K,\delta}-m_{N/K}$.

    We restrict attention
    to a fixed interval $[j-1,j)$.
    Let $\mu_g^j$ and $\mu_c^j$
    denote the  probability measures
    on $[j-1,j)\times {\cal W}^\delta$ defined by
    \begin{eqnarray*}
      \mu_g^j(I_1\times I_2)&=&
      \frac{
      \P(\ell(\eta^*_{N/K,\delta}-m_{N/K})-g(K)\in I_1,Kv^*_\delta/N\in I_2)}
      {
      \P(\ell(\eta^*_{N/K,\delta}-m_{N/K})-g(K)
    \in [j-1,j))},\\
    \mu_c^j(I_1\times I_2)&=&
    \frac{
            \P(Y\in I_1,Kz^{K,\delta}\in I_2)}
            {\P(Y\in [j-1,j))},
  \end{eqnarray*}
for intervals $I_1 \subseteq [j-1,j)$ and $I_2 \subseteq \mathcal W^{\delta}$.
  Note that $\mu_c^j$ has a positive density on $[j-1,j)\times
  {\cal W}^{\delta}$, which is uniformly  bounded from below with a bound not depending on
either $j$ or $K$,
 and that
  the L\'{e}vy distance between $\mu_g^j$ and $\mu_c^j$ is bounded from above
  by $\epsilon_K'\rightarrow_{K\rightarrow\infty} 0$, due to \eqref{eq-301212b-new}.
  Since $[j-1,j)\times {\cal W}^\delta$ is a bounded subset of
  $\R^3$, an elementary coupling
  (see, e.g., \cite[Theorem 1.2]{BJR}; the analog for one-dimensional couplings is easy to check) yields a coupling satisfying the analog of (\ref{eq-210113d}), but
restricted to $[j-1,j)$.
     The claim  \eqref{eq-210113d} then follows by combining the couplings
     for different $j$.
Further details are omitted.
  \end{proof}

  \subsection{A continuity lemma}
  \label{subsec-contlem}
  We will also need the following continuity result, which shows
  that the maximum value of the GFF is not affected by
  slightly changing the position at which the coarse field is sampled.
In what follows,
  let $z_i$ be as in Proposition \ref{prop-local}, let
$\bar \epsilon_K\rightarrow_{K\rightarrow\infty} 0$
  be as in Proposition
  \ref{prop-coupling}, and let $\{z_i'\}_{i=1}^{K^2}$ denote a family of
  independent random variables chosen so that $z_i'$ is measurable with respect to
  $\sigma(X_v^f, v\in V_N^{K,i})$ and that satisfies $K|z_i-z_i'|/N\leq \bar\epsilon_K$.
(Recall that $\{X_v^f, v\in V_N^{K,i}\}$ are independent for distinct $i$.)
  \begin{lemma}
    \label{lem-cont}
    With notation as above,
    \begin{equation}
      \label{eq-contlem}
    \lim_{K\to\infty} \limsup_{N\to\infty}
    d(\max_{i=1}^{K^2} (X_{z_i}^f+X_{z_i}^c), \max_{i=1}^{K^2}
   (X_{z_i}^f+X_{z_i'}^c)))=0\,.
  \end{equation}
\end{lemma}
\begin{proof}
  The argument is similar to that for $\P({\cal A}_2)\to_{N\to\infty} 0$, which was employed
  while proving Proposition \ref{prop-local}.
Denote by $V_{N, K}^{\times 2} = \{(u, v): u, v \in V_N, |u-v| \leq \bar \epsilon_K N/K\} $. For $(u, v)\in V_{N, K}^{\times 2}$, set $\zeta_{u, v, N, K} = \eta_{u, N} + X_u^c - X_v^c$ and $\zeta^*_{N, K} = \max_{(u, v) \in V_{N, K}^{\times 2}}\zeta_{u, v, N, K}$.

By Proposition \ref{prop-Delta} and \eqref{eq-star0103a} of Proposition \ref{prop-local}, for given $\epsilon >0$,
$$ |\max_{i=1}^{K^2} (X_{z_i}^f+X_{z_i}^c)- \eta^*_{N,K}|\leq \epsilon$$
with  probability $\rightarrow 1$, as first $N\to\infty$ and then $K\to\infty$.
On the other hand, for $z_i'$ as chosen above,
 $$|\max_{i=1}^{K^2} (X_{z_i}^f+X_{z_i}^c)- \max_{i=1}^{K^2}
 (X_{z_i}^f+X_{z_i'}^c)|\leq \zeta^*_{N,K}-\eta_{N,K}^*\,.$$
 It therefore follows from the definition of the L\'{e}vy distance that
  \begin{equation}\label{eq-levy-metric}
 \limsup_{K\to \infty} \limsup_{N\to \infty} d(\max_{i=1}^{K^2} (X_{z_i}^f+X_{z_i}^c), \max_{i=1}^{K^2}
   (X_{z_i}^f+X_{z_i'}^c)) \lesssim \E(\zeta^*_{N, K} - \eta^*_{N})\,.
  \end{equation}
By arguments that are essentially identical to those in the proof of Lemma~\ref{lem-step4ofwriteup} (where we used Lemma~\ref{lem-ferniquecriterion} and 
the upper bound in Lemma~\ref{lem-coarsecov}),
$\E \zeta^*_{N, K} \leq \E (\max_{u\in V_N} \eta_{u, N} + \tilde \epsilon_K \phi_{u, N})$, where $\tilde \epsilon_{K} \to_{K\to \infty} 0$ and $\{\phi_{u, N}\}$ are independent variables satisfying
$$\P(\phi_{u, N} \geq 1 + \lambda) \leq \mathrm{e}^{-\lambda^2} \mbox{ for all }  u\in V_N\,.$$
Application of Lemma~\ref{lem-gff-perturb} then implies $\E \zeta^*_{N, K} \leq \E \eta^*_N + C \sqrt{ \tilde \epsilon_K}$. Together with \eqref{eq-levy-metric}, this implies (\ref{eq-contlem}).
    \end{proof}
\subsection{Proofs of Theorems \ref{limit-law} and \ref{explicit-limit-law}}
We first prove Theorem \ref{limit-law}.
\proof[Proof of Theorem \ref{limit-law}]
Fix $\epsilon>0$.
Let $\alpha_{1,K,N}$ be as in Proposition \ref{prop-coupling} and
recall that $|\alpha_{1,K,N}-g(K)|\leq  \epsilon_K$.
Set
$$\bar \eta_{N}^*=\max_{\{i: X_{z_i}^f> m_{N/K}+g(K)\}}
(X_{z_i}^f+X_{z_i}^c)\,.$$
By applying Proposition \ref{prop-Delta} together with
\eqref{eq-star0103a} and \eqref{eq-star0103b} of
Proposition \ref{prop-local}, it follows, for small enough $\delta =
\delta(\epsilon)>0$
and large enough $K_0$, that,  for each $K\ge K_0$ and large enough $N$,
$$\P(\eta_N^*>\bar \eta^*_{N}+\epsilon)
<\epsilon \,.$$
  Let
$\nu_{N}^{K,\delta}$ denote the law of
$\bar\eta_{N}^*-m_N$.
Since $\eta_N^*\geq \bar\eta^*_{N}$, it follows from the
definition of $\mu_N$ that
$d(\mu_N,\nu_N^{K,\delta})\leq \epsilon$.

    Set $X_i^{f,*}=\max_{v\in V_N^{K,\delta,i}} X_v^f$, for $i=1,\ldots,K^2$,
    and recall that $X_i^{f,*}=X_{z_i}^f$.
    Set $\bar \B_i=1_{\{X_i^{f,*}
    -m_{N/K}
    \geq \alpha_{1,K,N}\}}$, and
    let $\{\B_i^{K,\delta},Y_i^K,z_i^{K,\delta}\}_{i=1}^{K^2}$ be an i.i.d. sequence
    of random vectors, with
    $(\B_i^{K,\delta},Y_i^K,z_i^{K,\delta})$ coupled to $(\bar \B_i,X_i^{f,*}-m_{N/K},z_i)$
    as in
    Proposition
    \ref{prop-coupling}.
    (Note that, for each $N$, the variables
    $(\B_i^{K,\delta},Y_i^K,z_i^{K,\delta})$
    require a different coupling; since their law does not depend $N$,
    we omit $N$ from the notation.)

By Proposition
    \ref{prop-coupling},
    $K|z_i^{K,\delta}-z_i/N|\leq \bar \epsilon_K$.
    Let $\bar \nu_N^{K,\delta}$ denote the law of
    $\max_{\{i: \B_i^K=1\}}(g(K)+Y_i^K+X_{z_i^K}^c-(m_N-m_{N/K}))$.
    It follows from Proposition \ref{prop-coupling} and
    Lemma \ref{lem-cont} that, for some $K_1\geq K_0$,
    $d(\bar \nu_N^{K,\delta},\nu_N^{K,\delta})<\epsilon$
    for each $K>K_1$ and large enough $N$.

    Finally, by the convergence of $X^c_N$ to $Z^c_{K,\delta}$, as $N\to\infty$,
    it follows that
    $d(\bar \nu_N^{K,\delta},\mu_{K,\delta})\to_{N\to\infty} 0$.
    We conclude that
    \begin{equation}
            \label{eq-Cauchy}
            \limsup_{K\to\infty} \limsup_{N\to\infty} 
    d(\mu_N,\mu_{K,\delta})<2\epsilon\,.
    \end{equation}
    Letting $\delta,\epsilon\to 0$ appropriately, demonstrates
    \eqref{eq-limiteq}. Furthermore, \eqref{eq-Cauchy} implies that,
    for given $\epsilon>0$, there exists $K(\epsilon)$ so that 
    $$\limsup_{N\to\infty} d(\mu_N,\mu_{K(\epsilon),\delta(\epsilon)})<
    3\epsilon\,.$$
    In particular,
$\mu_N$ is a Cauchy sequence, which implies the existence
of a limiting measure $\mu_{\infty}$ and
completes the proof of Theorem \ref{limit-law}.
\qed

\medskip

Before proving Theorem \ref{explicit-limit-law}, we need a preliminary
estimate on $Z^{c,*}_{K,\delta}:=\max_i Z_{K,\delta}^c(z_i^{K,\delta})$.

\begin{lemma}
        \label{lem-coarsecont}
        With notation as above, there exists $\gamma>0$ so that
        $$\lim_{K\to\infty}
        \P(Z^{c,*}_{K,\delta}\geq 2\sqrt{2/\pi}\log K-\gamma \log\log K)=0.$$
\end{lemma}
\noindent
\proof
Recall from \eqref{eq-page4} that, conditionally on
the collection  $\{z_i^{K,\delta}\}$,
the Gaussian random variables
$Z_{K,\delta}^c(z_i^{K,\delta})$ have mean zero and variance bounded
above by $\sigma_{K,\delta}^2:=(2/\pi)\log K+c_\delta$. Therefore,
by the obvious union bound,
\begin{eqnarray*}
        &\P(Z^{c,*}_{K,\delta}\geq 2\sqrt{2/\pi}\log K-\gamma\log\log K)
        \leq \,\, K^2 
        \max_i\P(Z^c_{K,\delta}(z_i^{K,\delta})\ge 2\sqrt{2/\pi}\log K-\gamma 
        \log\log K)\\
        &\leq \frac{C K^2}{\sigma_{K,\delta}} \exp(-(4/\pi)(\log K)^2/\sigma_{K,\delta}^2)
        \exp (\gamma \sqrt{2\pi}\log\log K)
        \le C(\delta)/(\log K)^{1/2-\gamma \sqrt{2\pi}}\,,
\end{eqnarray*}
for appropriate $C(\delta)$.
The conclusion follows for $\gamma<1/(2\sqrt{2\pi})$. \qed

\medskip

The proof of Theorem \ref{explicit-limit-law} follows from estimates on the
right tails of the random variables underlying the distributions $\mu_{K,\delta}$,
which are employed along with Theorem \ref{limit-law}.

\proof[Proof of Theorem \ref{explicit-limit-law}]
We set $g(K)=\gamma' \log\log K$, with $\gamma' \in (0,\gamma)$, and,
as in Subsection \ref{subsec-limit},  denote by $G^*_{K,\delta}$ a random
variable with law $\mu_{K,\delta}$. 
We will construct
random variables $Z_{K,\delta}$ so that
\begin{equation}
        \label{eq-gumbel1}
\lim_{\delta\searrow 0} \lim_{K\to\infty} 
\frac{\mu_{K,\delta}( (-\infty,x] )}{
        \E(\text{e}^{-\alpha^*Z_{K,\delta} \text{e}^{-\sqrt{2\pi}x}})} =1 
\end{equation}
for all $x$.
Theorem \ref{explicit-limit-law} then follows quickly from this and 
Theorem  \ref{limit-law}.  (Note that, after a change in parameters, the denominator
 is a Laplace transform.)

To demonstrate \eqref{eq-gumbel1}, let ${\cal F}^c$ denote the sigma-algebra generated
by the random variables $\{Z_{K,\delta}^c(z_i^{K,\delta})_ {i=1,\ldots, K^2}\}$. Then,
for any real $x$,
\begin{equation}
        \label{eq-gumbel2}
        \P(G^*_{K,\delta}\leq x)= 
        \E\left(\prod_{i=1}^{K^2}\left(1-
        \P(\B_i^{K,\delta}(Y_i^K+g(K))
        > x- \bar Z^c_{K,\delta}(i)\,|\,{\cal F}^c)
        \right)\right)\,,
\end{equation}
        where $\bar Z^c_{K,\delta}(i):=Z^c_{K,\delta}(z_i^{K,\delta})-
        2\sqrt{2/\pi}\log K$. For $H>0$, introduce the events
        ${\cal D}_i(x,H,K,\delta):=
        \{\bar Z^c_{K,\delta}(i)< -x-H\}$
        and ${\cal D}(x,H,K,\delta)=\cap_{i=1}^{K^2} {\cal D}_i(x,H,K,\delta)$.
        By Lemma \ref{lem-coarsecont} there exists a sequence $H=H(K)\to\infty$
        so that $\P({\cal D}(x,H,K,\delta))\to_{K\to\infty} 1$. 
        On the event ${\cal D}(x,H,K,\delta)$,
                $$\P(\B_i^{K,\delta}(Y_i^K+g(K))
        > x- \bar Z^c_{K,\delta}(i)\,|\,{\cal F}^c)=
        \alpha^* m_\delta 
        (x-\bar Z^c_{K,\delta}(i))
        e^{-\sqrt{2\pi}(x-\bar Z^c_{K,\delta}(i))}\to_{K\to\infty} 0\,.  $$
        Therefore, on ${\cal D}(x,H,K,\delta)$,
        \begin{eqnarray}        
                \label{eq-gumbel3}
                &\exp(-(1+\epsilon_{K,\delta})
        \alpha^* 
        (-\bar Z^c_{K,\delta}(i))
        \text{e}^{-\sqrt{2\pi}(x-\bar Z^c_{K,\delta}(i))})\,\leq \,
        \P(\B_i^{K,\delta}(Y_i^K+g(K))
        \leq x- \bar Z^c_{K,\delta}(i)\,|\,{\cal F}^c)\nonumber \\
        &\leq 
        \exp(-(1-\epsilon_{K,\delta})
        \alpha^* 
        (-\bar Z^c_{K,\delta}(i))
        \text{e}^{-\sqrt{2\pi}(x-\bar Z^c_{K,\delta}(i))})\,,
\end{eqnarray}
for $\epsilon_{K,\delta} > 0$ with 
$$\limsup_{\delta\searrow 0}\limsup_{K\to\infty}
\epsilon_{K,\delta}= 0.$$
(In the last display, $m_\delta\to 1$ as $\delta\searrow 0$ was used.)
        Define $Z_{K,\delta}=\sum_{i=1}^{K^2} (-\bar Z^c_{K,\delta}(i))e^{
                \sqrt{2\pi}\bar Z^c_{K,\delta}(i)}$. 
                Substituting  \eqref{eq-gumbel3} into \eqref{eq-gumbel2}
                and using that 
                $\P({\cal D}(x,H,K,\delta))
                \to_{K\to\infty} 1$ completes the proof of \eqref{eq-gumbel1} and hence of the theorem. \qed

\noindent{\bf Acknowledgment} We thank Pascal Maillard for the reference to \cite{BJR}, and 
Oren Louidor, 
Marek Biskup and Javier Acosta for helpful comments on an earlier 
version of the manuscript. We thank 
Russ Lyons and Yuval Peres for a suggestion 
that led to the proof of Lemma \ref{lem-coarsecov} presented here. 
We also thank the referee for a very careful reading of the paper 
and for his/her various constructive comments.

\end{document}